\documentclass[twoside,11pt]{amsart}

\usepackage{amsmath,latexsym,amssymb,mathptm, times,verbatim, enumerate,longtable}

\input amssym.def
\input amssym
\input xypic
\input xy
\xyoption{all}

\newtheorem{theorem}{Theorem}[section]
\newtheorem{lemma}[theorem]{Lemma}
\newtheorem{corollary}[theorem]{Corollary}
\newtheorem{proposition}[theorem]{Proposition}

\newtheorem{claim}[theorem]{Claim}
\newtheorem{observation}[theorem]{Observation}
\newtheorem*{Main Theorem}{Corollary~\ref{main}}

\newtheorem{question}[theorem]{Question}

\theoremstyle{definition}
\newtheorem{definition}[theorem]{Definition}
\newtheorem{example}[theorem]{Example}
\newtheorem{remark}[theorem]{Remark}
\newtheorem{remarks}[theorem]{Remarks}
\newtheorem{data}[theorem]{Data}

\newtheorem{convention}[theorem]{Convention}
\newtheorem{case}[theorem]{Case}
\newtheorem{conventions}[theorem]{Conventions}

\newtheorem{chunk}[theorem]{}

\newtheorem*{Remark}{Remark}
\newtheorem*{proof-free}{Proof of Lemma \ref{main-lemma}}

\numberwithin{equation}{theorem}
\numberwithin{table}{theorem}

\begin{document}

\baselineskip=16pt

\title[The structure of Gorenstein-linear  resolutions of  Artinian   algebras]{\bf 
The structure of Gorenstein-linear  resolutions\\ of  Artinian   algebras}
\date{\today}

\author[Sabine El Khoury and Andrew R. Kustin]
{Sabine El Khoury and Andrew R. Kustin}

\thanks{AMS 2010 {\em Mathematics Subject Classification}.
Primary 13D02,  13E10, 13H10, 13A02, 20G05.}

\thanks{The second author was partially supported by the Simons Foundation.}

\thanks{Keywords: Artinian rings, 
Gorenstein rings, Duality,  Linear resolution, Macaulay inverse system,   Pure resolution,
 Resolutions, Schur and Weyl modules for hooks, Weak Lefschetz property.}

\address{Mathematics Department,
American University of Beirut,
Riad el Solh 11-0236,
Beirut,
Lebanon}
\email{se24\@aub.edu.lb}

\address{Department of Mathematics, University of South Carolina,
Columbia, SC 29208} \email{kustin@math.sc.edu}

\begin{abstract} Let $\pmb k$ be a field, $A$  a standard-graded Artinian Gorenstein $\pmb k$-algebra, $S$   the standard-graded polynomial ring $\operatorname{Sym}_{\bullet}^{\pmb k}A_1$, $I$ the kernel of the natural map $\xymatrix{S\ar@{->>}[r]&A}$, $d$  the vector space dimension $\dim_{\pmb k}A_1$, and $n$  the least index with $I_n\neq 0$. 
Assume that $3\le d$ and $2\le n$. In this paper, we give the structure of the minimal homogeneous resolution ${\mathbf B}$ of $A$ by free $S$-modules, provided ${\mathbf B}$ is  Gorenstein-linear.
(Keep in mind  that if $A$ has even socle degree and is generic, then $A$ has a Gorenstein-linear minimal resolution.) 

Our description of  ${\mathbf B}$ depends on a fixed decomposition of $A_1$ of the form $\pmb k x_1\oplus V_0$, for some non-zero element $x_1$  and some $(d-1)$ dimensional subspace $V_0$ of $A_1$. 
Much information about ${\mathbf B}$ is already contained in the complex 
$\overline{{\mathbf B}}={\mathbf B}/x_1{\mathbf B}$, which we call the skeleton of ${\mathbf B}$.
One striking 
feature of ${\mathbf B}$ is the fact that the skeleton of ${\mathbf B}$ is completely determined by the data $(d,n)$; no other information about $A$ is used in the construction of  $\overline{{\mathbf B}}$. 

The skeleton $\overline{{\mathbf B}}$ is the mapping cone of ${\mathrm {zero}}:\mathbb K\to \mathbb L$, where $\mathbb L$ is a well known  resolution of Buchsbaum and Eisenbud;  $\mathbb K$ is the dual of $\mathbb L$; and $\mathbb L$ and $\mathbb K$ are comprised of Schur and Weyl  modules associated to hooks, respectively.
The decomposition of $\overline{{\mathbf B}}$ into  Schur and Weyl  modules lifts to a decomposition of 
${\mathbf B}$; furthermore, ${\mathbf B}$ inherits the natural self-duality of $\overline{{\mathbf B}}$.

The differentials of ${\mathbf B}$ are explicitly given,  in a polynomial manner,  
in terms of the coefficients of a Macaulay inverse system for $A$. In light of the properties of 
$\overline{{\mathbf B}}$, the description of the differentials of ${\mathbf B}$ amounts to giving a minimal generating set of $I$, and, for the interior differentials, giving the coefficients of $x_1$. As an application we observe that every non-zero element of  $A_1$ is a weak Lefschetz element for $A$.

\end{abstract}

\maketitle

\section{Introduction.}\label{Intro}

\begin{center}
{\sc    Table of Contents }
\end{center}

\begin{enumerate}
\item[\ref{Intro}.] Introduction. 
\item[\ref{C-and-B}.] Conventions and background.

\item[\ref{describe-B}.] The modules and maps of $(\mathbb B,b)$.

\item[\ref{describe-G}.] The complex $(\mathbb G,g)$.

\item[\ref{tau}.] The commutative diagram $\tau:\mathbb B \to \mathbb G$.

\item[\ref{main-theorem}.] The main theorem.

\item[\ref{Apps}.] Applications and examples

\item[\ref{expl-desc}.] A matrix  description of $\mathbb B$.
\end{enumerate}

\bigskip

Let $\pmb k$ be a field, $A$  a standard-graded Artinian Gorenstein $\pmb k$-algebra, $S$   the standard-graded polynomial ring $\operatorname{Sym}_{\bullet}^{\pmb k}A_1$, $I$ the kernel of the natural map $\xymatrix{S\ar@{->>}[r]&A}$, $d$  the vector space dimension $\dim_{\pmb k}A_1$, and $n$  the least index with $I_n\neq 0$. 
Assume that $3\le d$ and $2\le n$. In this paper, we give the structure of the minimal homogeneous resolution ${\mathbf B}$ of $A$ by free $S$-modules, provided ${\mathbf B}$ is  Gorenstein-linear.
There are numerous ways to determine a priori if $A$ has a Gorenstein-linear minimal resolution;
see for example,
Proposition~\ref{J18}. Furthermore,  
 if $A$ has even socle degree and is generic, then $A$ has a Gorenstein-linear minimal resolution; see for example, 
(\ref{generic}).

Our description of  ${\mathbf B}$ depends on a fixed decomposition of $A_1$ of the form $\pmb k x_1\oplus V_0$, for some non-zero element $x_1$  and some $(d-1)$ dimensional subspace $V_0$ of $A_1$. 
Much information about ${\mathbf B}$ is already contained in the complex 
$\overline{{\mathbf B}}={\mathbf B}/x_1{\mathbf B}$, which we call the skeleton of ${\mathbf B}$.
One striking 
feature of ${\mathbf B}$ is the fact that the skeleton of ${\mathbf B}$ is completely determined by the data $(d,n)$; no other information about $A$ is used in the construction of  $\overline{{\mathbf B}}$. Furthermore, the skeleton of ${\mathbf B}$ is sufficiently uncomplicated that we  can give its complete description here. 
 Let $\bar S=S/(x_1)$. Notice that $\bar S$ is a polynomial ring over the field $\pmb k$ in $d-1$ variables with homogeneous maximal ideal $\bar S_+$. 
A coordinate-free resolution of $\bar S/\bar S_+^n$  by free $\bar S$-modules was given by 
Buchsbaum and Eisenbud \cite{BE75}; call this resolution $\mathbb L$.
If $\mathbb K$ is the $\bar S$-dual of $\mathbb L$, then $\overline{{\mathbf B}}$ is the mapping cone of
$\mathbb K \xrightarrow{\ 0\ } \mathbb L$. (For example, the skeleton of ${\mathbf B}$ when $(d,n)=(4,2)$ is given in Example~\ref{d=4matrices-ext}.)

Now that we have given the complete description of  $\overline{{\mathbf B}}$, we lift as much information as possible from $\overline{{\mathbf B}}$ back to ${\mathbf B}$ and describe, precisely, what still needs to be identified. The modules in $\mathbb L$ are Schur modules for hooks. In our language  $\mathbb L$ is 
$$0\to \bar S\otimes_{\pmb k} L_{d-2,n}^{\pmb k} V_0\to \dots 
\to \bar S\otimes_{\pmb k} L_{1,n}^{\pmb k} V_0 \to \bar S\otimes_{\pmb k} L_{0,n}^{\pmb k} V_0\to \bar S.$$
   The modules in $\mathbb K$ are Weyl modules for hooks. In our language  $\mathbb K$ is 
$$0\to \bar S\otimes_{\pmb k}{\textstyle\bigwedge^{d-1}}V_0\to \bar S\otimes_{\pmb k} K_{d-2,n-1}^{\pmb k} V_0\to \dots 
\to \bar S\otimes_{\pmb k} K_{1,n}^{\pmb k} V_0 \to \bar S\otimes_{\pmb k} K_{0,n}^{\pmb k} V_0.$$
We give a complete description of the Schur module $L_{i,j}^{\pmb k}V_0$ and the Weyl module
 $K_{i,j}^{\pmb k}V_0$
 in (\ref{conv1}.\ref{L-K}); but notice at this point that these modules are built over $\pmb k$, or even, really, over $\mathbb Z$; and therefore, ${\mathbf B}$ is
$$0\to {\mathbf B}_d\xrightarrow{\ {\mathbf b}_d\ } \dots\xrightarrow{\ {\mathbf b}_2\ } {\mathbf B}_1\xrightarrow{\ {\mathbf b}_1\ } {\mathbf B}_0,\qquad\text{with}$$
$${\mathbf B}_r=\begin{cases} S,&\text{if $r=0$}\\S\otimes_{\pmb k}(K_{r-1,n-1}^{\pmb k}V_0\oplus L_{r-1,n}^{\pmb k}V_0),
&\text{if $1\le r\le d-1$, and}\\S\otimes_{\pmb k}{\textstyle\bigwedge^{d-1}}V_0,&\text{if $r=d$}.\end{cases}$$The part of the differential ${\mathbf b}_r:{\mathbf B}_r\to {\mathbf B}_{r-1}$ which involves $V_0$ is already known from $\overline{{\mathbf B}}$. In this paper we describe the part of ${\mathbf b}_r$ which involves $x_1$. That is, when $r$ equals $1$ or $d$, we identify a minimal generating set of $I$; and, for each $r$, with $2\le r\le d-1$, we identify a $\pmb k$-linear 
map \begin{equation}\label{k-lin}K_{r-1,n-1}^{\pmb k}V_0\oplus L_{r-1,n}^{\pmb k}V_0 \longrightarrow  K_{r-2,n-1}^{\pmb k}V_0\oplus L_{r-2,n}^{\pmb k}V_0\end{equation} so that $(S/V_0S)\otimes_S {\mathbf b}_r$ is equal to $x_1$ times (\ref{k-lin}).

 As already observed, $A$ provides the data $(d,n)$ for $\overline{{\mathbf B}}$; but otherwise, 
$\overline{{\mathbf B}}$ is completely independent of $A$
. It follows that  all of the information about $A$ is contained in (\ref{k-lin}) for any particular $r$, or, of course, in the minimal generating set for $I$. (So, ``finding the coefficients'' for $x_1$ is the entire problem and is highly non-trivial.)

Our description of a minimal generating set for $I$ has the immediate consequence that $x_1$ is a weak Lefschetz element for $A$; see Proposition~\ref{6.1}. We observe in Remark~\ref{WLEisEZ} that one can obtain a complete proof that $x_1$ is a weak Lefschetz element
once one has some knowledge of ${\mathbf B}$, but well
before one has mastered the entire construction of ${\mathbf B}$ and proof that ${\mathbf B}$ resolves $A$. 

Recall that $\overline{{\mathbf B}}$ is the mapping cone of $\operatorname{Hom}_{\overline{S}}(\mathbb L,\bar S)\xrightarrow{\ 0\ } \mathbb L$. It follows that there is a natural duality among the modules which form $\overline{{\mathbf B}}$. The Schur and Weyl modules which  underlie $\overline{{\mathbf B}}$ also underlie ${\mathbf B}$; and therefore, there is a natural duality among the modules  of ${\mathbf B}$. At the end of Section~\ref{Apps} we prove that the perfect pairing $\operatorname{pp}:{\mathbf B}_r\otimes {\mathbf B}_{d-r}\to {\mathbf B}_d$, which is induced by this natural duality of Schur and Weyl modules, is 
graded-commutative and satisfies the graded product rule. We wonder if $\operatorname{pp}$ is the backbone of  an associative DG-algebra structure on ${\mathbf B}$.

 The usefulness of the skeleton ${\mathbf B}$ (which captures the self-duality of ${\mathbf B}$ in a most trivial manner and reduces the problem of describing the differential $b_r$, for $2\le r\le d-1$, to ``finding the coefficient of $x_1$'') makes us wonder which other classes of resolutions  have similar skeletons. In particular, we ask the following question. 
\begin{question}\label{skeleton?}  Do other classes of pure or linear resolutions have interesting and meaningful skeletons?\end{question}
\noindent In fact, we first found the formulas for ${\mathbf B}$ and proved that ${\mathbf B}$ resolves $A$. 
It was only after we had 
${\mathbf B}$ in hand that we observed the interesting properties of ${\mathbf B}/(x_1{\mathbf B})$. Nonetheless, with an eye toward Question~\ref{skeleton?}, we  tackle the question ``How much about $\overline{{\mathbf B}}$ could be known before all of ${\mathbf B}$ is known?''. Our best answer is that if one knows that $x_1$ is a weak Lefschetz element, then one can prove the existence of $\overline{{\mathbf B}}$ with the desired properties; see Proposition~\ref{6.4}. 

\begin{chunk}\label{generic}
In order to make sense of the assertion ``the generic, standard-graded, Artinian, Gorenstein, $\pmb k$-algebra, with even socle degree, has a Gorenstein-linear minimal resolution'',
it is necessary to parameterize the $\pmb k$-algebras  under consideration. To this end, 
we introduce Macaulay inverse systems, which is the main tool in the paper. 
Start with the standard-graded polynomial ring  $S=\operatorname{Sym}_{\bullet}^{\pmb k}V$, where $V$ is a vector space of dimension $d$ over $\pmb k$, and  let $D=D_{\bullet}^{\pmb k}(V^*)$ be  
the graded $S$-module of graded $\pmb k$-linear homomorphisms from $S$ to $\pmb k$. 
Macaulay proved that
each non-zero element $\phi$ of $D$ determines  an 
 Artinian Gorenstein ring $A_{\phi}=S/\operatorname{ann}(\phi)$; furthermore, each Artinian Gorenstein quotient of $S$ is obtained in this manner. The element $\phi$ of $D$ is called the {\it Macaulay inverse system} for the Artinian Gorenstein ring $A_{\phi}$.
 Of course, $\phi$ determines everything about the quotient $A_{\phi}$; so in particular,  when $\phi$ is a homogeneous element of $D$, then $\phi$ determines a minimal resolution of $A_{\phi}$ by free $S$-modules. If one starts with $\phi$, then the standard way to find the minimal resolution of $A_{\phi}$ is to first solve some equations in order to determine a minimal  generating set for $\operatorname{ann}(\phi)$ and then to use Gr\"obner basis techniques in order to find a minimal resolution of $A_{\phi}$ by free $S$-modules. In the generic case (as explained below), we  by-pass all of the intermediate steps and describe a minimal  resolution of $A_{\phi}$ directly (and in a polynomial manner) in terms of  the coefficients of $\phi$.  In \cite{EKK}, we proved that if $\phi$ is homogeneous of even degree $2n-2$ and the pairing \begin{equation}\label{pp}A_{n-1}\times A_{n-1}\to \pmb k,\end{equation} given by
$(f,g)\mapsto fg(\phi)$, is perfect, then a minimal resolution for $A_{\phi}$ may be read directly, and in a polynomial manner, from the coefficients of $\phi$. Furthermore, there is one such resolution for each pair $(d,n)$. Please notice that the pairing (\ref{pp}) 
is perfect if and only if the determinant
of the matrix \begin{equation}\label{crit-det}((m_im_j)\phi),\end{equation} (as $m_i$ and $m_j$ roam over the monomials in $S$ of degree $n-1$), is non-zero. This is an open condition on the coefficients of 
$\phi$ (which are precisely the values of $m\phi$ as $m$ roams over the monomials of $S$ of degree $2n-2$); hence 
the pairing (\ref{pp})
is perfect whenever $\phi$ is chosen generically. Furthermore, the pairing (\ref{pp}) is perfect if and only if the minimal resolution of $A_{\phi}$ by free $S$-modules is Gorenstein-linear, see Proposition~\ref{J18}.

The paper \cite{EKK} proves the existence of a unique generic Gorenstein-linear resolution for each pair $(d,n)$; but exhibits this resolution only for the pair $(d,n)=(3,2)$. The  paper \cite{EK-K-2}, 
exhibits the unique generic Gorenstein-linear resolution for each pair $(d,n)$ when $d=3$ and $n\ge 2$ is arbitrary. 
The present paper exhibits the unique generic Gorenstein-linear resolution 
for all pairs $(d,n)$, provided $d\ge 3$ and $n\ge 2$. 
\end{chunk}

We think of the present  paper as a companion to the recent paper of Rossi and \c{S}ega \cite{RS} which proves that the Poincar\'e series of a finitely generated module $M$ over a compressed local Artinian ring $A$ is rational, provided the socle degree of $A$ is not $3$. The rings $A$ of the present paper satisfy the hypothesis of \cite{RS}; on the other hand, \cite{RS} also applies to non-graded rings and rings with odd socle degree. The paper \cite{RS} is about the Betti numbers in a resolution of $M$ by free $A$-modules. The present paper is about the differentials in a resolution of $A$ by free $S$-modules. The connection is that, traditionally, one has learned about the Poincar\'e series of the $A$-module $M$ by studying the $S$-resolution of $A$. So the present paper supplies information that \cite{RS} might have used if the information had been available. Furthermore, \cite{RS} pointed us in the direction of Fr\"oberg's paper \cite{F}. Fr\"oberg may have had some insight into the skeleton of $\mathbf B$. His paper is about the Poincar\'e series of two types of rings: the rings that we study (``graded extremal rings'') and rings defined by monomial ideals (``monomial rings''). Both of the complexes $\mathbb L$ and $\mathbb K$ which constitute the skeleton of $\mathbf B$ are complexes associated to monomials.

We give the structure of the minimal resolution for all standard-graded Artinian Gorenstein $\pmb k$-algebras $A$. For each relevant pair $(d,n)$, there is exactly one resolution with
$$\binom{2n+d-3}{d-1}$$ parameters. The parameters may be filled in with elements of $\pmb k$ chosen at random, with the one constraint that the determinant of  the corresponding $\binom{n+d-2}{d-1}\times \binom{n+d-2}{d-1}$ symmetric matrix be non-zero. In particular, there is one flat family of such algebras with minimal homogeneous resolution 
\begin{equation}\label{9-16}0\to S(-6)\to S^9(-4)\to S^{16}(-3)\to S^9(-2)\to S\end{equation} and each algebra $A$ in this family has Hilbert series $\operatorname{HS}_A=1+4t+t^2$. The situation is much different if the rules are changed. Suppose one keeps the Betti numbers $(1,9,16,9,1)$, continues to insist that $S$ be a positively graded  polynomial ring over $\pmb k$, and that $A$ be graded and Gorenstein, but no longer insists that the grading on $S$ be standard, that $A$ be Artinian, or that the minimal resolution of $A$ be Gorenstein-linear. This second set of hypotheses describes the situation in \cite{BKR}
which studies codimension four $\mathbb Q$-Fano threefolds; the corresponding section ring is a codimension four graded Gorenstein quotient of a positively graded polynomial ring with Betti numbers $(1,9,16,9,1)$. 
  Altinok's thesis \cite{A} shows that 
there are 145 potential families of codimension four Fano threefolds, given by considering the possible Hilbert series of the section ring. In \cite{BKR}   the authors show that
115 of these possibilities actually do occur. In fact they show that in each case there are
two Fano threefolds with the same numerical invariants but a different topology, so that
in each case there are at least two different flat families of Gorenstein algebras with the chosen Betti numbers and the same Hilbert series.  

Over time, it will be interesting to see how the present paper fits into Miles Reid's program \cite{R} for classifying codimension four Gorenstein algebras.

The case $n=2$ of our main result (including, for example, (\ref{9-16})) is already known by Behnke \cite{B}. Indeed, when $n=2$ and the minimal resolution of $A$ is Gorenstein-linear, then $A$ has minimal multiplicity (among Gorenstein algebras); and this is precisely the hypothesis in \cite{B}.  There are two signs of Behnke's influence in  our construction of the minimal resolution: the centrality of the perfect pairing map (\ref{pp})  and the importance of the distinguished element $x_1\in V_1$.   

\bigskip

In fact, we calculate over $\mathbb Z$ 
and our ultimate complex $(\mathbb B,b)$ is built over the ring $${\mathfrak R}={\operatorname{Sym}^{\mathbb Z}}_{\bullet}(U\oplus {\operatorname{Sym}^{\mathbb Z}}_{2n-2}U)$$ using a generic element $\Phi\in {\mathfrak R}\otimes_{{\mathbb Z}}D_{2n-2}^{\mathbb Z} (U^*)$ and a distinguished element $\pmb \delta$ in ${\mathfrak R}$, where $U$ is a free ${\mathbb Z}$-module of rank $d$ and $U^*$ is the $\mathbb Z$-dual of $U$.
The point is that if $V$ is a vector space of dimension $d$ over the field $\pmb k$, $\alpha:\pmb k\otimes_{\mathbb Z}U \to V$ is an isomorphism, and $\phi$ is an element in $D_{2n-2}^k(V^*)$ (where $V^*$ is the $\pmb k$-dual of $V$), then there is a unique ring homomorphism $\rho:\pmb k\otimes_{\mathbb Z}{\mathfrak R}\to S$ which factors through  $\alpha$ and sends $\Phi$ to $\phi$; see Proposition~\ref{transition}. The ring homomorphism $\rho$ sends $\pmb \delta$ 
 to
the critical determinant  of (\ref{crit-det}). The main result in the paper is Corollary~\ref{main} which states that if the determinant of (\ref{crit-det}) is non-zero, then 
${\mathbf B}=S\otimes_{\rho} \mathbb B$ is a minimal homogeneous resolution of $S/\operatorname{ann}(\phi)$ 
by free $S$-modules.

We define the ${\mathfrak R}$-module homomorphisms
$$(\mathbb B,b):\quad 0\to B_d\xrightarrow{\ b_d\ }B_{d-1}\xrightarrow{\ b_{d-1}\ }\cdots \xrightarrow{\ b_{2}\ } B_1\xrightarrow{\ b_1\ }B_{0}$$ in Definition~\ref{diff}. It is shown in Corollary~\ref{main-lemma}.\ref{main-lemma-e} that $(\mathbb B,b)$ is a complex and in Corollary~\ref{main-lemma}.\ref{main-lemma-g} that $(\mathbb B, b)_{\pmb\delta}$ is a resolution. 
A bi-homogeneous version of $(\mathbb B,b)$ may be found in Example~\ref{bi-gr}.

Section~\ref{C-and-B} contains all of the data, notation, and conventions that are used throughout the paper.
The modules and maps that eventually become the complex $\mathbb B$ are introduced in Section~\ref{describe-B}.
In Definition \ref{defG}, we recall the complex $\mathbb G$
of \cite{EKK}. Theorem \ref{Properties-of-G} recalls the wonderful properties of  $\mathbb G_{\pmb\delta}$. The complex  $\mathbb G_{\pmb\delta}$ is almost the dream complex; it has all of the desired properties; the only problem is that it is difficult to state exactly what $\mathbb G_{\pmb\delta}$ is. In the present paper, $(\mathbb B,b)$ is given explicitly. An important  result of the present paper, Corollary~\ref{main-lemma}.\ref{main-lemma-f}, is that $\mathbb B_{\pmb\delta}$ is isomorphic to $\mathbb G_{\pmb\delta}$. Thus, $\mathbb B_{\pmb\delta}$ has all of the wonderful properties of $\mathbb G_{\pmb\delta}$ and we know \underline{exactly} what $\mathbb B_{\pmb\delta}$ is. 
In Section~\ref{tau} we make an initial step toward proving that $\mathbb B_{\pmb\delta}$ and $\mathbb G_{\pmb\delta}$ are isomorphic
by exhibiting a commutative diagram $\tau:\mathbb B \to \mathbb G$. 
Theorem~\ref{thm-1} is the heart of the paper. This is where we identify a basis for each $(G_r)_{\pmb \delta}$. The complex $\mathbb G_{\pmb\delta}$ of \cite{EKK} and Theorem 3.4 has all of the desired properties, except it is not clear exactly what $\mathbb G_{\pmb\delta}$ is. In Theorem~\ref{thm-1} we determine a precise description of $\mathbb G_{\pmb\delta}$.

At the beginning of Section~\ref{Apps}
we apply our results in order to prove the existence of weak Lefschetz elements. In Proposition~\ref{6.4} we reverse directions and point out that if one knew the existence  of weak Lefschetz elements ahead of time, then one could deduce the form of skeleton of the minimal resolution before actually knowing the entire minimal resolution. In Example~\ref{d=3} we compare the complex $\mathbb B$ of \cite{EK-K-2}, where $d=3$, to the complex $\mathbb B$ of Definition~\ref{diff} and Corollary~\ref{main-lemma}.\ref{main-lemma-e}, where $d$ is arbitrary. Example~\ref{bi-gr} describes $\mathbb B$ as  a bi-homogeneous complex of free ${\mathfrak R}$-modules. Section~\ref{Apps} concludes with a discussion of the natural perfect pairing $\operatorname{pp}: B_r\otimes B_{d-r}\to B_d$, which is induced by the skeleton of $\mathbb B$. This pairing is
  graded-commutative and satisfies the graded product rule.

In Section \ref{expl-desc} we make  $\mathbb B$ significantly more explicit. We describe $\mathbb B$ in terms of  elements of the ring ${\mathfrak R}$, rather than in terms of  maps. 
Proposition~\ref{Best} contains a version of $b_r$ which is close to an explicit matrix for all $r$ and $d$. In Proposition~\ref{Best} the value of $b_r$ applied to a standard basis element of
$B_r$ is given explicitly as a linear combination of elementary generators of $B_{r-1}$ with coefficients in ${\mathfrak R}$. (Each elementary generator of $B_{r-1}$
can be expressed in terms of the standard basis elements of $B_{r-1}$, but that step is not carried out in Proposition~\ref{Best}.) 
Theorem~\ref{47.4} gives an explicit matrix version of $b_r$ for all $r$ and all $d$.
The proof of Theorem~\ref{47.4} requires a careful analysis of the standard bases for various Schur and Weyl modules, and the duality among theses bases. Further examples are included at the end of Section~\ref{expl-desc}.

\section{\bf Conventions and background.}\label{C-and-B}

An expanded version of this section may be found in \cite{EK-K-2}.

\begin{convention}In this paper, every unadorned operation is a functor in the category of $\mathbb Z$-modules. In particular, we write  $\otimes$, $\bigwedge^i$, $\operatorname{Sym}_j$, $D_j$, $L_{i,j}$, $K_{i,j}$, to mean $\otimes_{\mathbb Z}$, $\bigwedge^i_{\mathbb Z}$, $\operatorname{Sym}_j^{\mathbb Z}$, $D_j^{\mathbb Z}$, $L_{i,j}^{{\mathbb Z}}$ and $K_{i,j}^{{\mathbb Z}}$, respectively. Furthermore, we usually write $U^*$ to mean $\operatorname{Hom}_{{\mathbb Z}}(U,{\mathbb Z})$; and we will make a special note when we violate this convention.
\end{convention}

\begin{data}\label{Opening-Data} Let $U$ be a free $\mathbb Z$-module of rank $d$ and $n\ge 2$ be an integer.  

\begin{enumerate}[(a)]
\item Define ${\mathfrak R}$ to be the bi-graded ring
${\mathfrak R}={\operatorname{Sym}}_{\bullet}(U\oplus {\operatorname{Sym}}_{2n-2}U)$, where $$U\oplus 0={\mathfrak R}_{(1,0)}\quad\text{and}\quad 0\oplus{\operatorname{Sym}}_{2n-2}U={\mathfrak R}_{(0,1)}.$$ 

\item Define $\Psi:{\operatorname{Sym}}_{\bullet}U\to {\mathfrak R}$ to be the $\mathbb Z$-algebra homomorphism which is induced by the inclusion 
$$U={\mathfrak R}_{(1,0)}\subseteq {\mathfrak R}.$$ 

\item Define $\Phi: {\operatorname{Sym}}_{2n-2}U\to {\mathfrak R}$ to be the $\mathbb Z$-module homomorphism 
$${\operatorname{Sym}}_{2n-2}U={\mathfrak R}_{(0,1)}\subseteq {\mathfrak R}.$$
\end{enumerate}
\end{data}

\begin{definition}\label{manuf-data} Retain  Data  \ref{Opening-Data}. Let  $\mathrm{top}$ represent $\binom{n+d-2}{d-1}$, which is the rank of the free $\mathbb Z$-module ${\operatorname{Sym}}_{n-1}U$, and let $\Theta$ be a basis element for the rank one free $\mathbb Z$-module $\bigwedge^\mathrm{top} ({\operatorname{Sym}}_{n-1}U)$.
\begin{enumerate}[(a)]

\item We  think of $\Phi$ as an element of ${\mathfrak R}\otimes D_{2n-2}(U^*)$. 

\item Define $\mathbb Z$-module homomorphism $p:{\operatorname{Sym}}_{n-1}U\to {\mathfrak R}\otimes D_{n-1} (U^*)$ by
$[p(\mu)](\mu')=\Phi(\mu\mu')$,
for $\mu$ and $\mu'$ in ${\operatorname{Sym}}_{n-1} U$. 
 
\item Define the element  $\pmb \delta$ of ${\mathfrak R}$ by  
$$\pmb \delta= [{\textstyle (\bigwedge^\mathrm{top} p)}(\Theta)](\Theta)\in {\mathfrak R}_{(0,\mathrm{top})}.$$ 

\item Define the $\mathbb Z$-module homomorphism $q: D_{n-1}U^* \to {\mathfrak R}\otimes{\operatorname{Sym}}_{n-1}U$ by
$$\textstyle q(\nu)= [(\bigwedge^{\mathrm{top}-1}_{{\mathfrak R}}p)(\nu(\Theta))](\Theta) $$
for $\nu\in D_{n-1}U^*$.
\end{enumerate}\end{definition}
\begin{remark}\label{rmk3}It is shown in \cite[Obs.~2.5]{EK-K-2} that
$$[q(\nu)](\nu')=[q(\nu')](\nu)\quad\text{and}\quad [q(\nu)](\Phi)=\pmb\delta \nu,$$for all $\nu,\nu'\in D_{n-1}U$.
\end{remark}

\begin{data}\label{data1} Retain Data~\ref{Opening-Data}. 
\begin{enumerate}[\rm(a)]
\item Fix elements $x_1\in U$ and $x_1^*\in U^*$ with $x_1^*(x_1)=1$. 
\item Let $U_0=\ker(x_1^*:U\to \mathbb Z)$. Observe that  $U_0$ is a free ${\mathbb Z}$-module of rank $d-1$ and  $U={\mathbb Z} x_1\oplus U_0$. 
\item View $U_0^*$ as the sub-module of $U^*$ with $U_0^*(x_1)=0$.
\item Let ${\widetilde{\Phi}}$ be the element of ${\mathfrak R}\otimes D_{2n-1}(U^*)$ with $x_1({\widetilde{\Phi}})=\Phi$ in ${\mathfrak R}\otimes D_{2n-2}(U^*)$ and $\mu({\widetilde{\Phi}})=0$ for all $\mu$ in ${\operatorname{Sym}}_{2n-1}(U_0)$.
\end{enumerate}\end{data}

\begin{conventions}\label{conv1} Adopt~\ref{data1}. Let $N$, $a$, and $b$ be  integers.

\begin{enumerate}[\rm(a)]\item\label{conv1-a} Our results (except in Section~\ref{expl-desc}) are coordinate-free in $U_0$. There are times that we give a basis for some coordinate-free functor applied to $U$ or $U_0$ in terms of a basis for $U$ or $U_0$. At those times, we use  $x_1,x_2,\dots,x_d$ as a basis for $U$ and $x_2,\dots,x_d$ as a basis for $U_0$. Furthermore, we use $x_1^*,\dots,x_d^*$ to represent the basis for $U^*$ which is dual to the basis $x_1,\dots,x_d$ for $U$. 

\item For any set of variables $\{x_{i_1},\dots,x_{i_r}\}$ and any degree $s$, we write
$\binom{x_{i_1},\dots,x_{i_r}} s$ for   the set of monomials of degree $s$ in the variables $x_{i_1},\dots,x_{i_r}$.

\item\label{conv-f} If $m$ is the monomial $x_1^{a_1}x_2^{a_2}\cdots x_d^{a_d}$ of ${\operatorname{Sym}}_{N}U$, then $m^*$ is defined to be the element 
$${x_1^*}^{(a_1)}{x_2^*}^{(a_2)}\cdots {x_d^*}^{(a_d)}$$
of $D_{N}(U^*)$. The module action of ${\operatorname{Sym}}_{\bullet}U$ on $D_{\bullet}(U^*)$ makes 
$\{ m^*\mid m\in \binom{x_1,\dots,x_d}{N}\}$ be the ${{\mathbb Z}}$-module basis for the free ${\mathbb Z}$-module $D_{N}(U^*)$ which is dual to the ${{\mathbb Z}}$-module basis $\binom{x_1,\dots,x_d}{N}$ of ${\operatorname{Sym}}_{N}U$. (More information about divided power modules may be found, for example,  in \cite[subsect.~1.3]{EKK} or \cite[1.1]{EK-K-2}). 
\item\label{ev*} The evaluation map $\operatorname{Sym}_NU_0\otimes D_N(U_0^*)\xrightarrow{\ \operatorname{ev}\ }\mathbb Z$ is a coordinate-free map; consequently, the dual
$$\mathbb Z\xrightarrow{\ \operatorname{ev}^*\ }D_N(U_0^*)\otimes  \operatorname{Sym}_NU_0$$is also a coordinate-free map. Thus, $$\operatorname{ev}^*(1) =\sum\limits_{m\in \binom{x_2,\dots, x_{d}}{N}}m^*\otimes m$$ is an element of $D_N(U_0^*)\otimes  \operatorname{Sym}_NU_0$ which is completely coordinate-free.

\item\label{L-K} The operations $L_{a,b}$ and $K_{a,b}$ represent the Schur and Weyl functors 
for a hook
as described, for example,  in \cite[Data~2.1]{EKK}. (The translation between our notation and the notation of Buchsbaum and Eisenbud and the notation of Weyman is contained in \cite[2.1 and 2.5]{EKK}.)
In particular, 
$$\textstyle \begin{array}{lll}
L_{a,b}U_0&=&\ker \left(\bigwedge^aU_0\otimes{\operatorname{Sym}}_bU_0\xrightarrow{\kappa}
\bigwedge^{a-1}U_0\otimes{\operatorname{Sym}}_{b+1}U_0\right)\quad\text{and}\vspace{5pt}\\
K_{a,b}U_0&=&\ker \left(\bigwedge^aU_0\otimes D_b(U_0^*)\xrightarrow{\eta}
\bigwedge^{a-1}U_0\otimes D_{b-1}(U_0^*)\right),\end{array}$$where $\kappa$ is a Koszul complex map and $\eta$ is an Eagon-Northcott complex map.
In the language of (\ref{conv1-a}), 
\begin{equation}\label{kappa-eta}\kappa(\theta\otimes \mu)=\sum\limits_{i=2}^dx_i^*(\theta)\otimes x_i\mu\quad \text{and}\quad \eta(\theta\otimes \nu)=\sum\limits_{i=2}^dx_i^*(\theta)\otimes x_i(\nu)\end{equation} 
for $\theta\in \bigwedge^aU_0$, $\mu\in \operatorname{Sym}_bU_0$, and $\nu\in D_{b}(U_0^*)$. In light of (\ref{ev*}), the description of $\kappa$ and $\eta$ given in (\ref{kappa-eta}) is coordinate-free.

\item If $S$ is a statement then 
$$\chi(S)=\begin{cases} 1,&\text{if $S$ is true,}\vspace{5pt}\\0,&\text{if $S$ is false.}\end{cases}$$

\item If $m$ is a monomial in the variables $x_1,\dots, x_d$, then $x_i|m$ is the statement ``$x_i$ divides $m$''.

\item\label{conv-init} If $m$ is a monomial of positive degree in some variables $x_1,\dots,x_d$, then  let ``${\operatorname{init}}(m)=x_i$'' and ``$\operatorname{least}(m)=i$'' both  mean that $i$ is the least index for which $x_i|m$. 

\item\label{conv1-i} If $a$ and $b$ are integers, then let $[a,b]$ represent the set of integers
$[a,b]=\{i\in \mathbb Z\mid a\le i\le b\}$.

\end{enumerate}\end{conventions}

If $R$ is a graded ring, then a homogeneous   complex of free $R$-modules  is  {\it Gorenstein-linear} if it has the form 
$$0\to R(-2n-t+2)^{s_{t}}\xrightarrow{\ d_{t}\ } R(-n-t+2)^{s_{t-1}}\xrightarrow{\ d_{t-1}\ }\dots\xrightarrow{\ d_{3}\ } R(-n-1)^{s_2}\xrightarrow{\ d_{2}\ } R(-n)^{s_1}\xrightarrow{\ d_{1}\ }R^{s_0},$$
for some integers $n$, $t$, and $s_i$. In other words, all of  the entries in all of the matrices $d_i$, except the first matrix and the last matrix, are homogeneous linear forms; and all of the entries in the first and last matrices are homogeneous forms of the same degree. 

A graded ring $R=\bigoplus_{0\le i}R_i$ is called {\it standard-graded over} $R_0$, if $R$ is generated as an $R_0$-algebra by $R_1$ and $R_1$ is finitely generated as an $R_0$-module.

Let $\pmb k$ be a field.  View the standard-graded polynomial ring $S=\pmb k[x_1,\dots,x_d]$ as the symmetric algebra $\operatorname {Sym}^{\pmb k}_{\bullet}V$, where $V$ is the $\pmb k$-vector space with basis $x_1,\dots,x_d$ and view  $D_{\bullet}^{\pmb k}(V^*)$ as a $\operatorname {Sym}_{\bullet}^{\pmb k}V$-module. (At this point, $V^*$ means $\operatorname{Hom}_{\pmb k}(V,\pmb k)$.)
Macaulay \cite{M16} 
exhibited a duality between the ideals of $\operatorname {Sym}^{\pmb k}_{\bullet}V$
which define an Artinian quotient and the non-zero finitely-generated $\operatorname {Sym}^{\pmb k}_{\bullet}V$-submodules of $D_{\bullet}^{\pmb k}(V^*)$. The ideal $I$ of $\operatorname {Sym}^{\pmb k}_{\bullet}V$
corresponds to the submodule $$\operatorname{ann}(I)=\{\nu\in D_{\bullet}^{\pmb k}(V^*)\mid \mu(\nu)=0\in  D_{\bullet}^{\pmb k}(V^*) \text{ for all $\mu\in I$}\}$$ of $D_{\bullet}^{\pmb k}(V^*)$ and the submodule $M$ of 
$D_{\bullet}^{\pmb k}(V^*)$ corresponds to the ideal
$$\operatorname{ann}(M)=\{\mu \in \operatorname {Sym}^{\pmb k}_{\bullet}V \mid \mu(\nu)=0\in  D_{\bullet}^{\pmb k}(V^*) \text{ for all $\nu\in M$}\}$$ of  $\operatorname {Sym}^{\pmb k}_{\bullet}V$. A generating set for $\operatorname{ann}(I)$ is called a {\it Macaulay inverse system} for the ideal $I$. Under this duality, the quotient  ring $\operatorname {Sym}^{\pmb k}_{\bullet}V/I$ is Gorenstein if  and only if the module  $\operatorname{ann}(I)$ is cyclic. The following result, which gives many alternative formulations for the main hypothesis in the present paper, appears as \cite[Prop.~1.8]{EKK}.

\begin{proposition} \label{J18} Let $\pmb k$ be a field, $n$ and $d$ be positive integers,  $I$ be a homogeneous, grade $d$, Gorenstein ideal in the standard-graded polynomial ring $S=\pmb k[x_1,\dots,x_d]$, $V$ be the $d$-dimensional vector space $[S]_1$, and $\phi\in D_{\bullet}^{\pmb k}(V^*)$  be a homogeneous  Macaulay inverse system for $I$. Then the  following statements are equivalent\,{\rm:}
\begin{enumerate}[\rm(a)]
\item 
the minimal homogeneous resolution of $S/I$ by free $S$-modules is Gorenstein-linear, and   $I$ is generated by forms of degree $n;$

\item\label{J18-b} the minimal homogeneous resolution of $S/I$ by free $S$-modules has the form
$$0\to S(-2n-d+2)\to S(-n-d+2)^{\beta_{d-1}}\to \dots \to S(-n-1)^{\beta_2}\to   S(-n)^{\beta_1} \to S,$$
with $$  \beta_i =\frac{2n+d-2}{n+i-1}\binom{n+d-2}{i-1}\binom{n+d-i-2}{n-1},$$for $1\le i\le d-1;$
\item all of the minimal generators of $I$ have degree  $n$  and the socle of $S/I$ has degree $2n-2;$
\item $[I]_{n-1}=0$ and $[S/I]_{2n-1}=0;$
\item $\phi\in D_{2n-2}^{\pmb k}(V^*)$ and the homomorphism $\pmb p^\phi:\operatorname {Sym}_{n-1}^{\pmb k}V\to D_{n-1}^{\pmb k}(V^*)$, with $\pmb p^\phi(\mu)=\mu(\phi)$ for $\mu$ in $\operatorname {Sym}_{n-1}^{\pmb k}V$,   is an isomorphism$;$ and 
\item $\phi\in D_{2n-2}^{\pmb k}(V^*)$ and the $\binom{d+n-2}{d-1} \times \binom{d+n-2}{d-1}$ matrix 
$((m_im_j)\phi)$, as $m_i$ and $m_j$ roam over a list of the monomials of degree $n-1$ in $\{x_1,\dots,x_d\}$,  is invertible.
\end{enumerate}
\end{proposition}

We describe the transition between 
the generic data of  Data~\ref{Opening-Data}  
and
the specific data 
of a   Macaulay inverse system for a
standard-graded Artinian Gorenstein algebra over a field $\pmb k$.

\begin{proposition}\label{transition}
Adopt Data~{\rm\ref{Opening-Data}}. Let $V$ be a vector space of dimension $d$ over a field $\pmb k$,  $S$   the standard-graded polynomial ring $\operatorname{Sym}_{\bullet}^{\pmb k}(V)$, and $V^*$   the $\pmb k$-dual of $V$.
\begin{enumerate}[\rm(a)]
 \item\label{trans-a} If  $\alpha: \pmb k\otimes U\to V$ is a vector space isomorphism and $\phi$ is an element of $D_{2n-2}^{\pmb k}(V^*)$, then there exists a ring homomorphism $\rho:{\mathfrak R}\to S$ with the property that the composition $$ 
{\mathfrak R}\otimes D_{2n-2}U^* 
\xrightarrow{\ \rho\otimes 1 \ }S\otimes D_{2n-2}U^* \xrightarrow{\ 1\otimes (D\alpha^*)^{-1}
 \ }
S\otimes_{\pmb k}D_{2n-2}^{\pmb k}(V^*)$$ sends  the element $\Phi$ of  ${\mathfrak R}\otimes D_{2n-2}U^*$
 to the element 
$\phi$ 
of $$D_{2n-2}^{\pmb k}(V^*)=S_0\otimes_{\pmb k}D_{2n-2}^{\pmb k}(V^*)\subseteq S\otimes_{\pmb k}D_{2n-2}^{\pmb k}(V^*).$$ Furthermore, the ring homomorphism $\rho$ satisfies $\rho ({\mathfrak R}_{(1,0)})=S_1$
 and $\rho({\mathfrak R}_{(0,1)})\subseteq {S_0}$.
\item\label{trans-b} If $\rho: {\mathfrak R} \to S$  is a ring homomorphism which satisfies  
$\rho ({\mathfrak R}_{(1,0)})=S_1$
 and $\rho({\mathfrak R}_{(0,1)})\subseteq {S_0}$, then there exists a vector space isomorphism $\alpha: \pmb k\otimes U\to V$  and an element   $\phi$  of $D_{2n-2}^{\pmb k}(V^*)$ with 
$$\phi=((1\otimes (D\alpha^*)^{-1})\circ (\rho\otimes 1)) (\Phi),$$ as described in {\rm(\ref{trans-a})}.
\end{enumerate}
$($In this result, the $^*$ in  $\alpha^*$  represents $\pmb k$-dual.$)$\end{proposition}
 
\begin{proof} (\ref{trans-a}) Consider the $\mathbb Z$-module homomorphism
$$\beta :{\mathfrak R}_{(1,0)}\oplus 
{\mathfrak R}_{(0,1)} \xrightarrow{\ \bmatrix \beta_1&\beta_2 \endbmatrix\ }S,$$ where 
$\beta_1$ is the composition 
$${\mathfrak R}_{(1,0)}=U\to \pmb k\otimes U\xrightarrow{\ \alpha\ }V=S_1\subseteq S $$
and $\beta_2$ is the composition
$${\mathfrak R}_{(0,1)}=\operatorname{Sym}_{2n-2}U\to\pmb k\otimes\operatorname{Sym}_{2n-2}U\xrightarrow{\ \operatorname{Sym}(\alpha) \ }\operatorname{Sym}_{2n-2}^{\pmb k}
V\xrightarrow{\ \phi\ }\pmb k=S_0\subseteq S.$$
According to the universal mapping property of symmetric algebras, there is a unique $\mathbb Z$-algebra homomorphism ${\rho:{\mathfrak R} \to S}$ which extends the $\mathbb Z$-module homomorphism $\beta$. 

 Let $(\{\mu_i\},\{\nu_i\})$ be a pair of dual bases for 
${\operatorname{Sym}}_{2n-2} U$ and $D_{2n-2}(U^*)$, respectively. The generic element 
$\Phi$ of ${\mathfrak R}\otimes D_{2n-2}U^*$, as defined in Data~\ref{Opening-Data}, is equal to
\begin{equation}\label{Phi=}\Phi=\sum\limits_i \mu_i\otimes \nu_i,\end{equation} with $\mu_i\in {\mathfrak R}_{(0,1)}$. (The value of $\Phi$  in (\ref{Phi=}) is independent of the choice of $(\{\mu_i\},\{\nu_i\})$ because of (\ref{conv1}.\ref{ev*}).) It follows that   
$$\begin{array}{lll}[(1\otimes (D\alpha^*)^{-1})\circ (\rho\otimes 1)](\Phi)&=&(1\otimes (D\alpha^*)^{-1})\left(\sum\limits_i  \phi(\operatorname{Sym}(\alpha)(\pmb k\otimes \mu_i))\otimes \nu_i\right)\\
&=&\sum\limits_i  \phi(\operatorname{Sym}(\alpha)(\pmb k\otimes \mu_i))\otimes (D\alpha^*)^{-1}(\pmb k\otimes \nu_i)\\
&=&\phi.\end{array}$$
The final equality holds since $(\{\pmb k\otimes \mu_i\},\{\pmb k\otimes \nu_i\})$ is a pair of dual bases for ${\pmb k}\otimes{\operatorname{Sym}}_{2n-2} U$ and $\pmb k\otimes D_{2n-2}(U^*)$, respectively; and therefore,
$$(\{\operatorname{Sym}(\alpha)(\pmb k\otimes \mu_i)\},\{(D\alpha^*)^{-1}(\pmb k\otimes \nu_i)\})$$ is a pair of dual bases for ${\operatorname{Sym}}^{\pmb k}_{2n-2} V$ and $D_{2n-2}^{\pmb k}(V^*)$, respectively.

\medskip\noindent (\ref{trans-b}) 
The hypothesis $\rho ({\mathfrak R}_{(1,0)})=S_1$ 
ensures that $\rho: {\mathfrak R}_{(1,0)}=U\to S_1$ factors through a uniquely determined vector space isomorphism
$$\pmb k\otimes U \to S_1=V;$$
which we call $\alpha$.  The isomorphism $\alpha$ induces a vector space isomorphism 
$$\pmb k\otimes D_{2n-2}(U^*)\xrightarrow{\ D(\alpha^*)^{-1}\ } D_{2n-2}^{\pmb k}(V^*).$$
Retain the notation of (\ref{Phi=}). The hypothesis $\rho({\mathfrak R}_{(0,1)})\subseteq S_0$ ensures that $\rho(\mu_i)\in S_0$.
It now makes sense to define $\phi$ to be
$$((1\otimes (D\alpha^*)^{-1})\circ (\rho\otimes 1)) (\Phi)=\sum_i \rho(\mu_i)\otimes D(\alpha^*)^{-1}(\pmb k\otimes \nu_i) \in S_0\otimes_{\pmb k} D_{2n-2}^{\pmb k}(V^*)=D_{2n-2}^{\pmb k}(V^*).$$ 
\end{proof}

\section{\bf The modules and maps of $(\mathbb B,b)$.}\label{describe-B}

We define the ${\mathfrak R}$-module homomorphisms
$$(\mathbb B,b):\quad 0\to B_d\xrightarrow{\ b_d\ }B_{d-1}\xrightarrow{\ b_{d-1}\ }\cdots \xrightarrow{\ b_{2}\ } B_1\xrightarrow{\ b_1\ }B_{0}.$$ It is shown in Corollary~\ref{main-lemma}.\ref{main-lemma-e} that $(\mathbb B,b)$ is a complex and in Corollary~\ref{main-lemma}.\ref{main-lemma-g} that $(\mathbb B, b)_{\pmb\delta}$ is a resolution. The universal property of $(\mathbb B,b)$ is recorded in Corollary~\ref{main}. A bi-homogeneous version of $(\mathbb B,b)$ may be found in Example~\ref{bi-gr}.
\begin{definition}\label{diff} Adopt Data~{\rm\ref{data1}}. Let $\operatorname{proj}:\operatorname{Sym}_{n-1}U\to\operatorname{Sym}_{n-1}U_0$ the projection map induced  by $U={\mathbb Z} x_1\oplus U_0$.

\begin{enumerate}[\rm (a)]

\item\label{def1-a} Define the free ${\mathfrak R}$-module $B_r$ by
$$B_r=\begin{cases}
{\mathfrak R}&\text{if $r=0$}\\{\mathfrak R}\otimes (K_{r-1,n-1}U_0\oplus L_{r-1,n}U_0)&\text{if $1\le r\le d-1$}\\
{\mathfrak R}\otimes \bigwedge^{d-1}U_0&\text{if $r=d$}.\end{cases}$$
\item\label{diff-gen}  The ${\mathfrak R}$-module homomorphism  $b_1:B_1\to B_0={\mathfrak R}$ is defined by  $$\begin{array}{ll} b_1(\nu)=x_1q(\nu),&\text{for $\nu\in K_{0,n-1}U_0$,}\vspace{5pt}\\b_1(\mu)=\pmb\delta \mu-x_1q(\mu(\widetilde{\Phi})),&\text{for $\mu\in L_{0,n}U_0$.}\end{array}$$

\item\label{diff-a}  For each  integer $r$, with $2\le r\le d-1$, the ${\mathfrak R}$-module homomorphism $$b_r:\underbrace{{\mathfrak R}\otimes K_{r-1,n-1}U_0}_{\subseteq B_r}\to B_{r-1}=\left\{\begin{array}{c} {\mathfrak R}\otimes K_{r-2,n-1}U_0\\\oplus\\{\mathfrak R}\otimes L_{r-2,n}U_0\end{array}\right.$$ 
is defined by $$b_r({\textstyle\sum_i}\theta_i\otimes \nu_i)
=\begin{cases}-x_1\otimes \sum_i\eta\left(\theta_i\otimes [q(\nu_i)](\widetilde{\Phi})\right)
\vspace{5pt}\\
+\pmb \delta \sum_i{{\mathrm {Kos}}}^{\Psi}(\theta_i)\otimes \nu_i
\\\hline
-x_1\otimes \sum_i\kappa (\theta_i\otimes \operatorname{proj}(q(\nu_i)))
\end{cases}$$
for $\theta_i\in \bigwedge^{r-1}U_0$, $\nu_i\in D_{n-1}(U_0^*)$,  and $\sum_i\theta_i\otimes \nu_i\in K_{r-1,n-1}U_0$.
\item\label{diff-b}  For each  integer $r$, with $2\le r\le d-1$, the ${\mathfrak R}$-module homomorphism  $$b_r:\underbrace{{\mathfrak R}\otimes L_{r-1,n}U_0}_{\subseteq B_r}\to B_{r-1}=\left\{\begin{array}{c} {\mathfrak R}\otimes K_{r-2,n-1}U_0\\\oplus\\{\mathfrak R}\otimes L_{r-2,n}U_0\end{array}\right.$$  is defined by 
$$b_r({\textstyle\sum_i}\theta_i\otimes \mu_i)
=\begin{cases}\phantom{+}x_1\otimes \sum_i\eta \left(\theta_i\otimes [q(\mu_i(\widetilde{\Phi}))](\widetilde{\Phi})\right)
\vspace{5pt}\\\hline
+x_1\otimes \sum_i\kappa \left(\theta_i\otimes \operatorname{proj}(q(\mu_i(\widetilde{\Phi})))\right)\vspace{5pt}\\
+\pmb \delta \sum_i{{\mathrm {Kos}}}^{\Psi}(\theta_i)\otimes \mu_i
\end{cases}$$ for $\theta_i\in \bigwedge^{r-1}U_0$, $\mu_i\in \operatorname{Sym}_{n}U_0$, and $\sum_i\theta_i\otimes \mu_i\in L_{r-1,n}U_0$.

\item The ${\mathfrak R}$-module homomorphism  $b_d:B_d\to B_{d-1}$  is defined by 
$$b_d: {\mathfrak R}\otimes {\textstyle\bigwedge^{d-1}}U_0=B_d\longrightarrow B_{d-1}=\left\{\begin{array}{c} {\mathfrak R}\otimes K_{d-2,n-1}U_0\\\oplus\\ {\mathfrak R}\otimes L_{d-2,n}U_0\end{array}\right.$$
with
$$b_d(\omega)=\begin{cases} 
\phantom{+}\sum\limits_{m\in\binom{x_2,\dots,x_d}{n}} [\pmb\delta m-x_1q(m(\widetilde{\Phi}))]\otimes \eta(\omega\otimes m^*)\vspace{5pt}\\\hline -\sum\limits_{m\in\binom{x_2,\dots,x_d}{n-1}} x_1q(m^*)\otimes \kappa(\omega\otimes m),
\end{cases}$$ for $\omega\in \bigwedge^{d-1} U_0$.

\end{enumerate}
\end{definition}

\begin{remark}\label{meantime} The complex $(\mathbb G,g)$ is described in Definition \ref{defG}. 
An important result in the present paper is 
Corollary~\ref{main-lemma}.\ref{main-lemma-f},  where we prove that $\mathbb B_{\pmb\delta}$ and $\mathbb G_{\pmb\delta}$ are isomorphic complexes. In the meantime, we point out that the free ${\mathfrak R}$-modules $B_r$ and $G_r$ have the same rank. (The rank of $G_r$ is  given in Proposition~\ref{J18}.\ref{J18-b}, and also in 
\cite[Prop.~1.8~and~Thm.~6.5]{EKK}.) At any rate, if $1\le r\le d-1$, then 
$$\begin{array}{lllll}
\operatorname{rank} B_r&=&\operatorname{rank} K_{r-1,n-1}U_0+\operatorname{rank} L_{r-1,n}U_0&\text{Def.~\rm{\ref{diff}.\ref{def1-a}}}
\vspace{5pt}\\
&=&\binom{d+n-2}{r-1}\binom{d+n-r-2}{n-1}+\binom{d+n-2}{r-1+n}\binom{r+n-2}{r-1}
&\text{\rm\cite[(2.3)~and~(5.8)]{EKK}}\vspace{5pt}\\
&=&\frac{2n+d-2}{n+r-1}\binom{n+d-2}{r-1}\binom{n+d-r-2}{n-1}\vspace{5pt}\\
&=&\operatorname{rank} G_r.\end{array}$$\end{remark}

\begin{remark}\label{skeleton-2.2}Let $\overline{{\mathfrak R}}={\mathfrak R}/(x_1)$. We call $\overline {\mathbb B}=\overline{{\mathfrak R}}\otimes_{{\mathfrak R}}\mathbb B $ the skeleton of $\mathbb B$ and we think of $\mathbb B$ as a perturbation of its skeleton. Furthermore, the skeleton of $\mathbb B$ is the mapping cone of
the following commutative diagram:
$$\xymatrix{0\ar[r]& \overline{{\mathfrak R}}\otimes {\textstyle \bigwedge^{d-1}U_0}\ar[rr]^{\pmb\delta \eta\circ\operatorname{ev}^*}\ar[d]^{0}&&\overline{{\mathfrak R}}\otimes K_{d-2,n-1}U_0\ar[r]^{\pmb\delta{{\mathrm {Kos}}}^{\Psi}}\ar[d]^{0}&
\overline{{\mathfrak R}}\otimes K_{d-3,n-1}U_0\ar[r]^{\pmb\delta{{\mathrm {Kos}}}^{\Psi}}\ar[d]^{0}&\cdots\phantom{\otimes K_{d-3,n-1}}\\
0\ar[r]& \overline{{\mathfrak R}}\otimes L_{d-2,n}U_0\ar[rr]^{\pmb\delta{{\mathrm {Kos}}}^{\Psi}}&&
\overline{{\mathfrak R}}\otimes L_{d-3,n}U_0\ar[r]^{\pmb\delta{{\mathrm {Kos}}}^{\Psi}}&
\overline{{\mathfrak R}}\otimes L_{d-4,n}U_0\ar[r]^{\pmb\delta{{\mathrm {Kos}}}^{\Psi}}&\cdots\phantom{\otimes K_{d-3,n-1}}
}$$

\begin{equation}\label{skeleton}\hfill\xymatrix{
{\phantom{\otimes K_{d-3,n-1}U_0}}\cdots\ar[r]^{\ \ \ \pmb\delta{{\mathrm {Kos}}}^{\Psi}}&\overline{{\mathfrak R}}\otimes K_{2,n-1}U_0\ar[r]^{\pmb\delta{{\mathrm {Kos}}}^{\Psi}}\ar[d]^{0}
&\overline{{\mathfrak R}}\otimes K_{1,n-1}U_0\ar[r]^{\pmb\delta{{\mathrm {Kos}}}^{\Psi}}\ar[d]^{0}
&\overline{{\mathfrak R}}\otimes K_{0,n-1}U_0\ar[d]^{0}
\\
{\phantom{\otimes K_{d-3,n-1}U_0}}\cdots\ar[r]^{\ \ \ \pmb\delta{{\mathrm {Kos}}}^{\Psi}}&
\overline{{\mathfrak R}}\otimes L_{1,n}U_0\ar[r]^{\pmb\delta{{\mathrm {Kos}}}^{\Psi}}&\overline{{\mathfrak R}}\otimes L_{0,n}U_0\ar[r]^{\pmb\delta\Psi}&\overline{{\mathfrak R}}.}\end{equation}
Let $\mathbb K$, and $\mathbb L$, represent the top row, and the bottom row, of (\ref{skeleton}), respectively. We notice that $\mathbb L_{\pmb\delta}$ is a resolution of $\overline{{\mathfrak R}}/(x_2,\dots, x_d)^n$ by free $\overline{{\mathfrak R}}$-modules and $\mathbb K_{\pmb \delta}$, which is isomorphic to 
$$[\operatorname{Hom}_{\overline{{\mathfrak R}}}(\mathbb L_{\pmb \delta},\overline{{\mathfrak R}})\otimes {\textstyle\bigwedge^{d-1}}U_0](-d+1),$$
is a resolution of the canonical module of  $\overline{{\mathfrak R}}/(x_2,\dots, x_d)^n$ by free $\overline{{\mathfrak R}}$-modules. Complexes isomorphic to  $\mathbb L_{\pmb\delta}$ and $\mathbb K_{\pmb \delta}$ were introduced by Buchsbaum and Eisenbud in \cite{BE75}. Our discussion of these complexes is contained in \cite[Sect.~2]{EKK}: the definition is given in 2.6, a very explicit version of the last map in $\mathbb K$ is given in 2.10, and the assertions about duality and acyclicity in Theorem~2.12. (Do keep in mind that $U_0$ is a free ${\mathbb Z}$-module of rank $d-1$; hence the top non-zero exterior power of $U_0$ is $\bigwedge ^{d-1}U_0$ and the complexes $\mathbb L$ and $\mathbb K$ both have length $d-1$.)
\end{remark}

\section{\bf The complex $(\mathbb G,g)$.}\label{describe-G}

In Definition \ref{defG}, we recall the complex $$(\mathbb G,g):\quad0\to G_g\xrightarrow{\ g_d\ }G_{d-1}\xrightarrow{\ g_{d-1}\ }\cdots \xrightarrow{\ g_{2}\ } G_1\xrightarrow{\ g_1\ }G_{0}.$$
of \cite[Def.~6.6 and Thm.~6.15]{EKK} and \cite[(3.14)]{EK-K-2}. Theorem \ref{Properties-of-G} recalls the wonderful properties of  $\mathbb G_{\pmb\delta}$. The complex  $\mathbb G_{\pmb\delta}$ is almost the dream complex; it has all of the desired properties; the only problem is that it is difficult to state exactly what $\mathbb G_{\pmb\delta}$ is. In the present paper, $(\mathbb B,b)$ is given explicitly. An important result of the present paper, Corollary~\ref{main-lemma}.\ref{main-lemma-f}, is that $\mathbb B_{\pmb\delta}$ is isomorphic to $\mathbb G_{\pmb\delta}$. Thus, $\mathbb B_{\pmb\delta}$ has all of the wonderful properties of $\mathbb G_{\pmb\delta}$ and we know \underline{exactly} what $\mathbb B_{\pmb\delta}$ is.

\begin{definition}\label{defG} Adopt Data \ref{Opening-Data}.

\begin{enumerate}[\rm(a)]
\item The ${\mathfrak R}$-modules $G_0$ and $G_d$ are defined by $G_0={\mathfrak R}$ and $G_d={\mathfrak R}\otimes \bigwedge^dU$.

\item\label{back2} If $1\le r\le d-1$, then the ${\mathfrak R}$-module $G_r$ is defined to be 
$$G_r=\ker \left(\underline{\mathrm v}_r: {\mathfrak R}\otimes L_{r-1,n}U\to {\mathfrak R}\otimes K_{r-1,n-2}U\right),$$
where $\underline{\mathrm v}_r$ is described by the commutative diagram
$$\xymatrix{
G_r\ar@{^{(}->}[d]\vspace{5pt}\\
{\mathfrak R}\otimes L_{r-1,n}U\ar@{^{(}->}[r]\ar[d]^{\underline{\mathrm v}_r=1\otimes \pmb p^{\Phi}}&{\mathfrak R}\otimes \bigwedge^{r-1}U\otimes \operatorname{Sym}_nU\ar[d]^{1\otimes1\otimes\pmb p^{\Phi}}\vspace{5pt}\\
{\mathfrak R}\otimes K_{r-1,n-2}U\ar@{^{(}->}[r]&{\mathfrak R}\otimes \bigwedge^{r-1}U\otimes D_{n-2}U^*,}$$
and \begin{equation}\label{eq1}\pmb p^{\Phi}(\mu)=\mu(\Phi)\in {\mathfrak R}\otimes  D_{n-2}(U^*),\end{equation} for $\mu\in \operatorname{Sym}_nU$.

\item\label{g1} The differential $g_1:G_1\to G_0$ is $$G_1=\ker \underline{\mathrm v}_1\subseteq {\mathfrak R}\otimes L_{0,n}U\longrightarrow {\mathfrak R}=G_0,$$ and if $\Theta=\sum_i a_i\otimes \mu_i$ is in $G_1$ with $a_i\in {\mathfrak R}$ and $\mu_i\in \operatorname{Sym}_nU=L_{0,n}U$, then \begin{equation}\label{eq10}g_1(\Theta)=\sum_ia_i\mu_i\in {\mathfrak R} =G_0.\end{equation} 

\item\label{gr} If $1\le r\le d-1$, then the differential $g_r:G_r\to G_{r-1}$ is induced by the following commutative diagram:
$$\xymatrix {{\mathfrak R}\otimes \bigwedge^{r-1}U\otimes\operatorname{Sym}_nU\ar[r]^{{{\mathrm {Kos}}}^{\Psi}\otimes 1}&
{\mathfrak R}\otimes \bigwedge^{r-2}U\otimes\operatorname{Sym}_nU\\  {\mathfrak R}\otimes L_{r-1,n}U\ar@{-->}[r]\ar@{^(->}[u]&{\mathfrak R}\otimes L_{r-2,n}U\ar@{^(->}[u]
\\ G_r\ar@{-->}[r]^{g_r}\ar@{^(->}[u]&G_{r-1}\ar@{^(->}[u]}$$

\item The differential $g_d: G_d= {\mathfrak R}\otimes {\textstyle\bigwedge^{d-1}}U\to G_{d-1}=\ker\underline{\mathrm v}_{d-1}\subseteq {\mathfrak R} \otimes L_{d-2,n}U$ is given by
$$g_d(\Omega)=\sum\limits_{j,k=1}^d \sum\limits_{m\in\binom{x_1,\dots,x_d}{n-1}} x_km\otimes (x_k^*\wedge x_j^*)( \Omega)\otimes x_j q(m^*),$$ for $\Omega\in \bigwedge^{d} U$.
\end{enumerate}
\end{definition}

\begin{remark}\label{de-emp} In other versions of our work \cite{EKK,EK-K-2}, we have said that the differential $g_1$ of (\ref{g1}) is $\Psi$. In this paper, we have de-emphasized the name $\Psi$. When we want to apply this algebra homomorphism, we merely move the element from the ``basis-portion'' of the given module to the ``${\mathfrak R}$-portion'' of the module as is shown in (\ref{eq10}). Along similar lines, the map ${{\mathrm {Kos}}}^{\Psi}:{\mathfrak R}\otimes \bigwedge^{r-1}U \to {\mathfrak R}\otimes \bigwedge^{r-2}U$ of (\ref{gr}) sends $\theta$ in 
$\bigwedge^{r-1}U$ to $\sum_{i=1}^d x_i\otimes x_i^*(\theta)$ in ${\mathfrak R}\otimes \bigwedge^{r-2}U$. (The conventions of \ref{conv1}.\ref{ev*} are in effect.)
\end{remark}

\begin{definition}\label{I}Adopt Data \ref{Opening-Data}. Let $I$ be the image, in ${\mathfrak R}$, of the kernel of
$$1\otimes \pmb p^{\Phi}: {\mathfrak R}\otimes \operatorname{Sym}_{\bullet} U\to {\mathfrak R}\otimes D_{\bullet} (U^*);$$ 
in other words,
$$I=\left\{\sum_ia_i\mu_i\in {\mathfrak R} \left\vert \begin{array}{l}\text{$\sum_i{a_i}\otimes \mu_i\in {\mathfrak R}\otimes \operatorname{Sym}_{\bullet}U$, with $a_i\in {\mathfrak R}$, $\mu_i\in \operatorname{Sym}_{\bullet}U$,}\vspace{5pt}\\ \text{and  $\sum a_i\otimes \mu_i(\Phi)=0\in {\mathfrak R}\otimes D_\bullet(U^*)$}\end{array}\right.\right\}.$$\end{definition}

\bigskip

The following result is established in \cite[Thms.~6.15 and 4.16]{EKK} and is restated (for $d=3$) in \cite[Thm.~3.2]{EK-K-2}. For item (\ref{p-of-p}) one must also use 
the ``Persistence of Perfection Principle'', which is also known as  the ``transfer of perfection'' (see 
\cite[Prop.~6.14]{Ho-T} or \cite[Thm.~3.5]{BV}).

\begin{theorem}\label{Properties-of-G} Adopt Data~{\rm{\ref{Opening-Data}}}.    Recall $\mathbb G$ and $I$ from Definitions~{\rm\ref{defG}} and {\rm\ref{I}}.
The following statements hold.
  \begin{enumerate}[\rm(a)]
\item The ${\mathfrak R}$-module homomorphisms $(\mathbb G,g)$ form a complex.
\item The localization $\mathbb G_{\pmb \delta}$ is a  resolution of ${\mathfrak R}_{\pmb \delta}/I{\mathfrak R}_{\pmb \delta}$ by projective ${\mathfrak R}_{\pmb\delta}$-modules. 
\item\label{p-of-p} If $R$ is a Noetherian ring, $\rho:{\mathfrak R}\to R$ is a ring homomorphism, $(\rho({\mathfrak R}_{(1,0)}))$ is an ideal of $R$ of grade at least $d$, and $\rho(\pmb \delta)$ is a unit of $R$, then $R\otimes_{{\mathfrak R}}\mathbb G$ is a resolution of $R/\rho(I)$ by projective $R$-modules.

\item\label{field} Let $V$ be a vector space of dimension $d$ over a field $\pmb k$, $S$  the standard-graded polynomial ring ${\operatorname{Sym}}^{\pmb k}_\bullet V$, 
and $\rho: {\mathfrak R} \to S$  be a ring homomorphism which satisfies 
\begin{enumerate}[\rm(i)]
\item $\rho ({\mathfrak R}_{(1,0)})=S_1$
\item
$\rho({\mathfrak R}_{(0,1)})\subseteq {\pmb k}$,  and  \item $\rho(\pmb \delta)$ is a unit in ${\pmb k}$.
\end{enumerate} 
If $\phi$ is the element of $D_{2n-2}^{\pmb k}(V^*)$ which corresponds to the ring homomorphism $\rho$ in the sense of Proposition~{\rm\ref{transition}.\ref{trans-b}},
then 
 $S\otimes_{\mathfrak R} \mathbb G$ is a minimal homogeneous resolution of $S/\operatorname{ann} (\phi)$ by free $S$-modules and 
$S\otimes_{\mathfrak R} \mathbb G$ is a Gorenstein-linear resolution of the form
$$\begin{array}{l}0\to S(-2n-d+2)\to S(-n-d+2)^{\beta_{d-1}}\to 
\dots \to S(-n-2)^{\beta_3}\to S(-n-1)^{\beta_2}\to S(-n)^{\beta_1}\to S\end{array}$$

\item\label{by-way-of} Let $V$ be a vector space of dimension $d$ over a field $\pmb k$, $S$  the standard-graded polynomial ring ${\operatorname{Sym}}^{\pmb k}_\bullet V$, 
 $\phi$ an element of $D^{\pmb k}_{2n-2}(V^*)$, and $\alpha:\pmb k\otimes U\to V$ an isomorphism. 
Let $\rho : {\mathfrak R}\to S$ be the ring homomorphism which corresponds to the data $(\alpha,\phi)$ in the sense of Proposition~{\rm\ref{transition}.\ref{trans-a}}. If the resolution of $S/\operatorname{ann} (\phi)$ by free $S$-modules is Gorenstein-linear, then $\rho$ satisfies all of the hypotheses of {\rm(\ref{field})} and the conclusion of {\rm(\ref{field})} holds for $S/\operatorname{ann} (\phi)$. 
\end{enumerate}
\end{theorem}
\begin{remarks} \begin{enumerate}[\rm(a)]\item 
The paper \cite{EKK} only promises that the ${\mathfrak R}_{\pmb \delta}$-modules  $(G_r)_{\pmb\delta}$ of Theorem~\ref{Properties-of-G} are projective. In Theorem~\ref{thm-1}, we  prove that  the ${\mathfrak R}_{\pmb \delta}$-modules $(G_r)_{\pmb\delta}$  are free. So, Theorem~\ref{thm-1} shows that each  ``projective'' in Theorem~\ref{Properties-of-G} may be replaced by ``free''. \item 
In (\ref{field}) and (\ref{by-way-of}), $V^*$ is the $\pmb k$-dual of $V$.
\end{enumerate} 
\end{remarks}

\section{\bf The commutative diagram $\tau:\mathbb B \to \mathbb G$.}\label{tau}
In Proposition \ref{prop} we exhibit a commutative diagram $\tau:\mathbb B \to \mathbb G$. The properties of the complex $\mathbb G_{\pmb\delta}$ are listed in Theorem~\ref{Properties-of-G} and $\mathbb B$ is the sequence of maps and modules which is explicitly given in Definition~\ref{diff}. In Corollary~\ref{main-lemma}.\ref{main-lemma-f} we prove that $\mathbb B_{\pmb\delta}$ and $\mathbb G_{\pmb\delta}$ are isomorphic; thereby merging the explicitness of $\mathbb B_{\pmb\delta}$ with the properties of $\mathbb G_{\pmb\delta}$.
In the present section we make an initial step toward proving that $\mathbb B_{\pmb\delta}$ and $\mathbb G_{\pmb\delta}$ are isomorphic.

The ${\mathfrak R}$-module homomorphisms $\tau_r$ are introduced in Definition~\ref{def1}. Observation~\ref{ob1} is a handy result which allows us to show, in Observation~\ref{ob2}, that the image of $\tau_r$ is contained in $G_r$ (rather than the larger ${\mathfrak R}$-module ${\mathfrak R}\otimes L_{r-1,n}U$. The majority of the section is the proof of Proposition \ref{prop}. 

\begin{definition}\label{def1}Adopt Data~{\rm\ref{data1}}. Recall $(\mathbb B,b)$ and $(\mathbb G,g)$ from Definitions~{\rm\ref{diff}} and {\rm\ref{defG}}.
\begin{enumerate}[\rm(a)]

\item Define the ${\mathfrak R}$-module map $\tau_0: B_0\to G_0$ to be the identity map on ${\mathfrak R}$.

\item\label{def1-b} Fix $r$ with $1\le r\le d-1$. Define the restriction of $\tau_r:B_r\to {\mathfrak R}\otimes L_{r-1,n}U$ to the summand $${\mathfrak R} \otimes K_{r-1,n-1}U_0\text{ of }B_r$$ to be the ${\mathfrak R}$-module homomorphism
 induced by
$$\textstyle\bigwedge^{r-1}U_0\otimes D_{n-1}U_0^* \to {\mathfrak R}\otimes L_{r-1,n}U$$ with
$$\theta\otimes \nu \mapsto \pmb\delta^{r-1}\kappa( (x_1\wedge\theta)\otimes q(\nu)),$$ for $\theta\in \bigwedge^{r-1}U_0$ and $\nu\in D_{n-1}(U_0^*)$. 
\item\label{def1-c}  Fix $r$ with $1\le r\le d-1$. Define the restriction of $\tau_r:B_r\to {\mathfrak R}\otimes L_{r-1,n}U$ to the summand $${\mathfrak R} \otimes L_{r-1,n}U_0\text{ of }B_r$$  to be the ${\mathfrak R}$-module homomorphism
  induced by 
$$\textstyle\bigwedge^{r-1}U_0\otimes \operatorname{Sym}_nU_0\to \textstyle{\mathfrak R}\otimes \bigwedge^{r-1}U\otimes \operatorname{Sym}_nU$$ with
$$\theta\otimes \mu\mapsto \pmb\delta^{r-1}\left[\pmb\delta \theta\otimes \mu - \kappa \left((x_1\wedge \theta) \otimes q(\mu(\widetilde{\Phi}))\right)\right],$$for $\theta\in \bigwedge^{r-1}U_0$ and $\mu\in \operatorname{Sym}_{n}U_0$. 
\item Define $\tau_d: B_d \to G_d$ by $\tau_d(\omega)=\pmb\delta^{d-1} x_1\wedge \omega$, for $\omega\in \bigwedge^{d-1}U_0$. 
\end{enumerate}
\end{definition}

\begin{observation}\label{ob1} The ${\mathfrak R}$-module homomorphism $$\underline{\mathrm v}_r: {\mathfrak R} \otimes L_{r-1,n}U\to {\mathfrak R}\otimes K_{r-1,n-2}U$$ of Definition~{\rm\ref{defG}.\ref{back2}} satisfies 
$$\underline{\mathrm v}_r(\kappa(\theta\otimes q(\nu)))=\pmb\delta \eta(\theta\otimes \nu),$$
for any $r$ with $1\le r\le d-1$, any $ \theta\in \bigwedge^{r}U$, and any $\nu\in D_{n-1}(U^*)$.
\end{observation}
\begin{proof} One uses the fact that 
$\pmb p^{\Phi}\circ \kappa=\eta\circ \pmb p^{\Phi}$, together with
Remark~\ref{rmk3},  to see  that $$\underline{\mathrm v}_r(\kappa(\theta\otimes q(\nu)))=\eta\left(
\theta\otimes [q(\nu)](\Phi)\right)=\pmb\delta \eta(\theta\otimes \nu).$$\end{proof}

\begin{observation}\label{ob2}Adopt  Data~{\rm\ref{data1}}. The image of the map $\tau_r:B_r\to {\mathfrak R}\otimes L_{r-1,n}U$, for $1\le r\le d-1$, as described in Definition~{\rm \ref{def1}.\ref{def1-b}} and  {\rm \ref{def1}.\ref{def1-c}}, is contained in $G_r$. 
\end{observation}

\begin{proof}It is clear that $\tau_r(B_r)$ is contained in ${\mathfrak R}\otimes L_{r-1,n}U$. 
We show that $\tau_r(B_r)\subseteq G_r$ by showing that $(\underline{\mathrm v}_r\circ \tau_r)(B_r)=0$.
We treat the two summands of $B_r$ separately.
First, we consider $\Theta$ in the summand $K_{r-1,n-1}U_0$ of $B_r$. Write $\Theta=\sum_h\theta_h\otimes \nu_h$, with $\theta_h\in \bigwedge^{r-1}U_0$, $\nu_h\in D_{n-1}(U^*_0)$, and $\eta(\Theta)=0$. We see that
$$\begin{array}{llll}
(\underline{\mathrm v}_r\circ \tau_r)(\Theta)&=&\pmb \delta^{r-1}\sum_h\underline{\mathrm v}_r(\kappa(x_1\wedge\theta_h)\otimes q(\nu_h))&{\rm
(\ref{def1}.\ref{def1-b})}\vspace{5pt}\\
&=&\pmb \delta^r\sum_h\eta((x_1\wedge\theta_h)\otimes \nu_h)&\rm{Obs.~\ref{ob1}}\vspace{5pt}\\
&=&-\pmb \delta^r x_1\wedge \eta(\sum_h\theta_h\otimes \nu_h)&x_1(U_0^*)=0\vspace{5pt}\\
&=&0&\eta(\Theta)=0.\end{array}$$ 
Now, we consider $\Theta\in L_{r-1,n}U_0\subseteq B_r$. Write $\Theta
=\sum_h\theta_h\otimes \mu_h$ with $\theta_h\in \bigwedge^{r-1}U_0$, $\mu_h\in \operatorname{Sym}_{n}U_0$, and 
$\kappa (\Theta)=0$. We see that
$$\begin{array}{llll}
(\underline{\mathrm v}_r\circ \tau_r)(\Theta)&=&\pmb \delta^{r-1}\sum_h\underline{\mathrm v}_r\left(
\pmb\delta \theta_h\otimes \mu_h - \kappa \left((x_1\wedge \theta_h) \otimes q(\mu_h(\widetilde{\Phi}))\right)\right)
&{\rm
(\ref{def1}.\ref{def1-c})}\vspace{5pt}\\
&=&\pmb \delta^{r}
 \sum_h\left[ \theta_h\otimes \mu_h(\Phi)-  \eta\left((x_1\wedge \theta_h) \otimes \mu_h(\widetilde{\Phi})\right)\right]
&\rm{Obs.~\ref{ob1}}\vspace{5pt}\\
&=&\pmb \delta^{r}\sum_h\left[ \theta_h\otimes \mu_h(\Phi)-   \theta_h\otimes \mu_h(\Phi)+ x_1\wedge  \eta\left(\theta_h \otimes \mu_h(\widetilde{\Phi})\right)\right]
&x_1(\widetilde{\Phi})=\Phi\vspace{5pt}\\
&=&
\pmb \delta^r \left[x_1\wedge (1\otimes \pmb p^{\widetilde{\Phi}})\left(\sum_h \kappa(\theta_h \otimes \mu_h)\right)\right]&
 \eta\circ \pmb p^{\widetilde{\Phi}}=\pmb p^{\widetilde{\Phi}}\circ \kappa 
\vspace{5pt}\\
&=&0
&\kappa(\Theta)=0,\end{array}$$
where $\textstyle\pmb p^{\widetilde{\Phi}}:{\mathfrak R}\otimes \operatorname{Sym}_{n+1}U\to {\mathfrak R} \otimes D_{n-2}U^*$
is the ${\mathfrak R}$-module homomorphism which sends $\mu$ to $\mu(\widetilde{\Phi})$, for  $\mu\in \operatorname{Sym}_{n+1}U$. \end{proof}

\begin{proposition}\label{prop} 
The ${\mathfrak R}$-module homomorphisms $\tau_i:B_i\to G_i$ of Definition~{\rm\ref{def1}} give rise to a commutative diagram
$$\xymatrix{
0 \ar[r]&B_d\ar[r]^{b_d}\ar[d]^{\tau_d}&B_{d-1}\ar[r]^{b_{d-1}}\ar[d]^{\tau_{d-1}}&\cdots\ar[r]^{b_{2}}
&B_1\ar[r]^{b_1}\ar[d]^{\tau_1}&B_0\ar[d]^{\tau_{0}}\\
0 \ar[r]&G_d\ar[r]^{g_d}&G_{d-1}\ar[r]^{g_{d-1}}&\cdots\ar[r]^{g_{2}}
&G_1\ar[r]^{g_1}&G_0}$$
\end{proposition}

\begin{proof}  
We show that
\begin{equation}\label{goal-a!}(\tau_{r-1}\circ b_r)(\Theta)-(g_r\circ \tau_{r})(\Theta)=0\end{equation}
for $\Theta\in \mathbb B$. We consider five distinct cases:
\begin{case}\label{case1}$\Theta\in K_{0,n-1}U_0\subseteq B_1$,\end{case}
\begin{case}\label{case2}$\Theta\in L_{0,n}U_0\subseteq B_1$,\end{case}
\begin{case}\label{case3}$\Theta\in K_{r-1,n-1}U_0\subseteq B_r$ for $2\le r\le d-1$,\end{case}
\begin{case}\label{case4}$\Theta\in L_{r-1,n}U_0\subseteq B_r$ for $2\le r\le d-1$, and \end{case}
\begin{case}\label{case5}$\Theta\in  B_d$. \end{case}

\medskip\noindent {Case~\ref{case1}.} If $\Theta=\nu\in D_{n-1}U_0^*=K_{0,n-1}U_0\subseteq B_1$, then 
$(\tau_0\circ b_1)(\Theta)=x_1q(\nu)= (g_1\circ \tau_1)(\Theta)$ in ${\mathfrak R}=G_0$.

\medskip\noindent Case~\ref{case2}. If $\Theta=\mu\in \operatorname{Sym}_{n}U_0=L_{0,n}U_0\subseteq B_1$, then $(\tau_0\circ b_1)(\Theta)=\pmb\delta \mu -x_1q(\mu(\widetilde{\Phi}))= (g_1\circ \tau_1)(\Theta)$. 

\medskip\noindent Case~\ref{case3}.
Fix $r$ with $2\le r\le d-1$. Let $\Theta=\eta(\theta\otimes \nu) \in K_{r-1,n-1}U_0\subseteq B_r$, with
$\theta\in \bigwedge^rU_0$ and $\nu$ in $D_{n}(U_0^*)$. 
Recall Convention~\ref{conv1} and observe that
$\Theta=\sum_{j=2}^d x_j^*(\theta)\otimes x_j(\nu)$.
Apply Definitions~\ref{diff}.\ref{diff-a} and \ref{def1} to see that $$(\tau_{r-1}\circ b_r)(\Theta)=
\tau_{r-1}\left.\begin{cases}
-x_1\otimes \sum\limits_{j,k=2}^d ([x_k^*\wedge x_j^*](\theta))\otimes [x_kq(x_j(\nu))](\widetilde{\Phi})\vspace{5pt}\\
+\pmb \delta \sum\limits_{j,k=2}^d x_k\otimes ([x_k^*\wedge x_j^*](\theta))\otimes x_j(\nu)\vspace{5pt}\\\hline
-x_1\otimes \sum\limits_{j,k=2}^d ([x_k^*\wedge x_j^*](\theta))\otimes x_k\operatorname{proj}(q(x_j(\nu)))
\end{cases}\right\}=\pmb\delta^{r-2}\sum\limits_{i=1}^7S_i, \text{ for}$$
\begin{longtable}{l}
$S_1=-x_1\otimes \sum\limits_{j,k=2}^d ([x_k^*\wedge x_j^*](\theta))\otimes x_1q\left([x_kq(x_j(\nu))](\widetilde{\Phi})\right)$,\\
$S_2=\phantom{+}x_1\otimes \sum\limits_{j,k,\ell=2}^d (x_1\wedge ([x_\ell^*\wedge x_k^*\wedge x_j^*](\theta)))\otimes x_{\ell}q\left([x_kq(x_j(\nu))](\widetilde{\Phi})\right)$,\\
$S_3=\phantom{+}\pmb \delta \sum\limits_{j,k=2}^d x_k\otimes ([x_k^*\wedge x_j^*](\theta))\otimes x_1q(x_j(\nu))$,\\
$S_4=-\pmb \delta \sum\limits_{j,k,\ell=2}^d x_k\otimes (x_1\wedge ([x_{\ell}^*\wedge x_k^*\wedge x_j^*](\theta)))\otimes x_{\ell}q(x_j(\nu))$,\\
$S_5=-\pmb\delta x_1\otimes \sum\limits_{j,k=2}^d ([x_k^*\wedge x_j^*](\theta))\otimes x_k\operatorname{proj}(q(x_j(\nu)))$,\\
$S_6=\phantom{+}x_1\otimes \sum\limits_{j,k=2}^d ([x_k^*\wedge x_j^*](\theta))\otimes 
x_1q\left([x_k\operatorname{proj}(q(x_j(\nu)))](\widetilde{\Phi})\right)
$,\text{ and}\\
$S_7=-x_1\otimes \sum\limits_{j,k,\ell=2}^d (x_1\wedge ([x_{\ell}^*\wedge x_k^*\wedge x_j^*](\theta)))\otimes 
x_{\ell}q\left([x_k\operatorname{proj}(q(x_j(\nu)))](\widetilde{\Phi})\right)$.
\end{longtable}

\noindent A similar calculation gives
$$ -(g_{r}\circ \tau_r)(\Theta)
=-\pmb \delta^{r-1} g_r \left.\begin{cases}
\phantom{+}\sum\limits_{j=2}^d x_j^*(\theta)\otimes x_1q(x_j(\nu))\vspace{5pt}\\
-\sum\limits_{j,k=2}^d (x_1\wedge [x_k^*\wedge x_j^*](\theta))\otimes x_kq(x_j(\nu))
\end{cases}\right\}=\pmb \delta^{r-2}\sum_{i=8}^{10}S_i, \text{ for}
$$
\begin{longtable}{l}
$S_8=-\pmb \delta \sum\limits_{j,k=2}^d x_k\otimes [x_k^*\wedge x_j^*](\theta)\otimes x_1q(x_j(\nu))$,
\\
$S_9=\phantom{+}\pmb \delta 
x_1\otimes \sum\limits_{j,k=2}^d [x_k^*\wedge x_j^*](\theta)\otimes x_kq(x_j(\nu))$, and
\\
$S_{10}=-\pmb \delta \sum\limits_{j,k,\ell=2}^d x_{\ell}\otimes (x_1\wedge [x_{\ell}^*\wedge x_k^*\wedge x_j^*](\theta))\otimes x_kq(x_j(\nu))$.
\end{longtable}

\noindent We see that $S_3+S_8=S_4+S_{10}=0$ and
\begin{longtable}{l}
$S_1+S_6=-x_1\otimes \sum\limits_{j,k=2}^d ([x_k^*\wedge x_j^*](\theta))\otimes x_1q\left([x_k(1-\operatorname{proj})q(x_j(\nu))](\widetilde{\Phi})\right)$,\\
$S_2+S_7=\phantom{+}x_1\otimes \sum\limits_{j,k,\ell=2}^d (x_1\wedge ([x_\ell^*\wedge x_k^*\wedge x_j^*](\theta)))\otimes x_{\ell}q\left([x_k(1-\operatorname{proj})q(x_j(\nu))](\widetilde{\Phi})\right)$,\\
$S_5+S_9=\phantom{+}\pmb \delta 
x_1\otimes \sum\limits_{j,k=2}^d [x_k^*\wedge x_j^*](\theta)\otimes x_k(1-\operatorname{proj})q(x_j(\nu))$.
\end{longtable}

\noindent If $Y\in \operatorname{Sym}_{n-1}U$, then $(1-\operatorname{proj})(Y)\in x_1\operatorname{Sym}_{n-2}U$; so, $(1-\operatorname{proj})(Y)=x_1Y_0$ for a unique element $Y_0$ in   $\operatorname{Sym}_{n-2}U$. It is convenient to write $\frac{(1-\operatorname{proj})(Y)}{x_1}$ for the unique element $Y_0$ in   $\operatorname{Sym}_{n-2}U$ with $(1-\operatorname{proj})(Y)=x_1Y_0$. We apply this technique to $S_1+S_6$ with 
$q(x_j(\nu))$ playing the role of $Y$. Recall that $x_1(\widetilde{\Phi})=\Phi$ and that $[q(\nu)](\Phi)=\pmb \delta \nu$ for all $\nu\in D_{n-1}(U^*)$. We see that 
\begingroup
\allowdisplaybreaks
\begin{eqnarray*}
S_1+S_6&=&-x_1\otimes \sum\limits_{j,k=2}^d ([x_k^*\wedge x_j^*](\theta))\otimes x_1q\left([x_k(1-\operatorname{proj})q(x_j(\nu))](\widetilde{\Phi})\right)
\\
&=&-x_1\otimes \sum\limits_{j,k=2}^d ([x_k^*\wedge x_j^*](\theta))\otimes x_1q\left(\left[x_k\frac{(1-\operatorname{proj})q(x_j(\nu))}{x_1}x_1\right](\widetilde{\Phi})\right)\\
&=&-x_1\otimes \sum\limits_{j,k=2}^d ([x_k^*\wedge x_j^*](\theta))\otimes x_1q\left(\left[x_k\frac{(1-\operatorname{proj})q(x_j(\nu))}{x_1}\right](\Phi)\right)\\
&=&-\pmb\delta x_1\otimes \sum\limits_{j,k=2}^d ([x_k^*\wedge x_j^*](\theta))\otimes x_1x_k\frac{(1-\operatorname{proj})q(x_j(\nu))}{x_1}\\
&=&-\pmb\delta x_1\otimes \sum\limits_{j,k=2}^d ([x_k^*\wedge x_j^*](\theta))\otimes x_k(1-\operatorname{proj})q(x_j(\nu));
\end{eqnarray*}
\endgroup
thus, $(S_1+S_6)+(S_5+S_9)=0$. In a similar manner, we compute
\begingroup
\allowdisplaybreaks
\begin{eqnarray*}
S_2+S_7&=&x_1\otimes \sum\limits_{j,k,\ell=2}^d (x_1\wedge ([x_\ell^*\wedge x_k^*\wedge x_j^*](\theta)))\otimes x_{\ell}q\left([x_k(1-\operatorname{proj})q(x_j(\nu))](\widetilde{\Phi})\right)\\
&=&x_1\otimes \sum\limits_{j,k,\ell=2}^d (x_1\wedge ([x_\ell^*\wedge x_k^*\wedge x_j^*](\theta)))\otimes x_{\ell}q\left(\left[x_k\frac{(1-\operatorname{proj})q(x_j(\nu))}{x_1}x_1\right](\widetilde{\Phi})\right)\\
&=&x_1\otimes \sum\limits_{j,k,\ell=2}^d (x_1\wedge ([x_\ell^*\wedge x_k^*\wedge x_j^*](\theta)))\otimes x_{\ell}q\left(\left[x_k\frac{(1-\operatorname{proj})q(x_j(\nu))}{x_1}\right](\Phi)\right)\\
&=&\pmb\delta x_1\otimes \sum\limits_{j,k,\ell=2}^d (x_1\wedge ([x_\ell^*\wedge x_k^*\wedge x_j^*](\theta)))\otimes x_{\ell}x_k\frac{(1-\operatorname{proj})q(x_j(\nu))}{x_1}=0.
\end{eqnarray*}
\endgroup
We have established (\ref{goal-a!}) in Case \ref{case3}. 

\medskip\noindent Case~\ref{case4}.
Keep $r$ in the range $2\le r\le d-1$. Let $\Theta=\kappa(\theta\otimes \mu) \in L_{r-1,n}U_0\subseteq B_r$, with
$\theta\in \bigwedge^rU_0$ and $\mu\in \operatorname{Sym}_{n-1}U_0$. Recall Convention~\ref{conv1} and observe that
$\Theta=\sum_{j=2}^d x_j^*(\theta)\otimes x_j\mu$.
Apply Definitions~\ref{diff}.\ref{diff-b} and \ref{def1} to see that $$(\tau_{r-1}\circ b_r)\Theta=\tau_{r-1}\left.\begin{cases}
\phantom{+}x_1\otimes \sum\limits_{j,k=2}^d ([x_k^*\wedge x_j^*](\theta))\otimes \left[x_kq([x_j\mu](\widetilde{\Phi}))\right](\widetilde{\Phi})\\
\hline
+x_1\otimes \sum\limits_{j,k=2}^d ([x_k^*\wedge x_j^*](\theta))\otimes x_k\operatorname{proj}(q([x_j\mu](\widetilde{\Phi})))\\
+\pmb\delta\sum\limits_{j,k=2}^d x_k\otimes ([x_k^*\wedge x_j^*](\theta))\otimes x_j\mu
\end{cases}\right\}=\pmb\delta^{r-2}\sum_{i=1}^8S_i,$$ for
\begin{longtable}{l}
$S_1=\phantom{+}x_1\otimes  \sum\limits_{j,k=2}^d    ([x_k^*\wedge x_j^*](\theta))\otimes 
x_1q\left(\left[x_kq([x_j\mu](\widetilde{\Phi}))\right](\widetilde{\Phi})\right)$,\\
$S_2=-x_1\otimes \sum\limits_{j,k,\ell=2}^d  (x_1\wedge  [x_\ell^*\wedge x_k^*\wedge x_j^*](\theta))\otimes 
x_\ell q\left(\left[x_kq([x_j\mu](\widetilde{\Phi}))\right](\widetilde{\Phi})\right)$,
\\
$S_3=\phantom{+}\pmb\delta x_1\otimes \sum\limits_{j,k=2}^d ([x_k^*\wedge x_j^*](\theta))\otimes x_k\operatorname{proj}(q([x_j\mu](\widetilde{\Phi})))$,
\\
$S_4=\phantom{+}\pmb\delta^2\sum\limits_{j,k=2}^d x_k\otimes ([x_k^*\wedge x_j^*](\theta))\otimes x_j\mu$,
\\
$S_5=-x_1\otimes \sum\limits_{j,k=2}^d 
([x_k^*\wedge x_j^*](\theta))\otimes x_1q\left(\left[x_k\operatorname{proj}(q([x_j\mu](\widetilde{\Phi})))\right](\widetilde{\Phi})\right)$,
\\
$S_6=\phantom{+}x_1\otimes \sum\limits_{j,k,\ell=2}^d 
(x_1\wedge([x_\ell^*\wedge x_k^*\wedge x_j^*](\theta)))\otimes x_\ell q\left(\left[x_k\operatorname{proj}(q([x_j\mu](\widetilde{\Phi})))\right](\widetilde{\Phi})\right)$,
\\
$S_7=-\pmb\delta\sum\limits_{j,k=2}^d x_k\otimes ([x_k^*\wedge x_j^*](\theta))\otimes x_1
q([x_j\mu](\widetilde{\Phi}))$, and \\
$S_8=\phantom{+}\pmb\delta\sum\limits_{j,k,\ell=2}^d x_k\otimes (x_1\wedge([x_\ell^*\wedge x_k^*\wedge x_j^*](\theta)))\otimes x_\ell q([x_j\mu](\widetilde{\Phi}))$.
\end{longtable}
In a similar manner, we see that  
$$- (g_r\circ \tau_{r})(\Theta)=\pmb\delta^{r-1} g_r\left.\begin{cases}
-\pmb\delta \sum\limits_{j=2}^d x_j^*(\theta)\otimes x_j\mu\\
+\sum\limits_{j=2}^d x_j^* (\theta)\otimes x_1q([x_j\mu](\widetilde{\Phi}))\\
-\sum\limits_{j,k=2}^d (x_1\wedge [x_k^*\wedge x_j^*](\theta))\otimes x_kq([x_j\mu](\widetilde{\Phi}))
\end{cases}\right\}= \pmb\delta^{r-2}\sum_{i=9}^{12}S_i,$$
for 
\begin{longtable}{l}
$\phantom{_1}S_9=-\pmb\delta^2 \sum\limits_{j,k=2}^d x_k\otimes ([x_k^*\wedge x_j^*](\theta))\otimes x_j\mu$,\\
$S_{10}=\phantom{+}\pmb\delta \sum\limits_{j,k=2}^d x_k\otimes ([x_k^*\wedge x_j^*](\theta))\otimes x_1q([x_j\mu](\widetilde{\Phi}))$,\\
$S_{11}=-\pmb\delta x_1 \otimes\sum\limits_{j,k=2}^d ([x_k^*\wedge x_j^*](\theta))\otimes x_kq([x_j\mu](\widetilde{\Phi}))$, \\
$S_{12}=\phantom{+}\pmb\delta \sum\limits_{j,k,\ell=2}^d x_\ell\otimes (x_1\wedge [x_\ell^*\wedge x_k^*\wedge x_j^*](\theta))\otimes x_kq([x_j\mu](\widetilde{\Phi}))$.
\end{longtable}
We see that
$$S_4+S_9=S_7+S_{10}=S_8+S_{12}=0,$$
\begin{longtable}{l}
$
S_1+S_5=\phantom{+}x_1\otimes \sum\limits_{j,k=2}^d    ([x_k^*\wedge x_j^*](\theta))\otimes 
x_1q\left(\left[x_k(1-\operatorname{proj})q([x_j\mu](\widetilde{\Phi}))\right](\widetilde{\Phi})\right)$,
\\
$S_2+S_6=-x_1\otimes \sum\limits_{j,k,\ell=2}^d  (x_1\wedge  [x_\ell^*\wedge x_k^*\wedge x_j^*](\theta))\otimes 
x_\ell q\left(\left[x_k(1-\operatorname{proj})q([x_j\mu](\widetilde{\Phi}))\right](\widetilde{\Phi})\right)$, and
\\
$S_3+S_{11}=-x_1\otimes \pmb\delta \sum\limits_{j,k=2}^d ([x_k^*\wedge x_j^*](\theta))\otimes x_k(1-\operatorname{proj})q([x_j\mu](\widetilde{\Phi}))$.
\end{longtable}
\noindent If $Y\in \operatorname{Sym}_{n-1}U$, then $(1-\operatorname{proj})(Y)\in x_1\operatorname{Sym}_{n-2}U$; so, $(1-\operatorname{proj})(Y)=x_1Y_0$ for a unique element $Y_0$ in   $\operatorname{Sym}_{n-2}U$. It is convenient to write $\frac{(1-\operatorname{proj})(Y)}{x_1}$ for the unique element $Y_0$ in   $\operatorname{Sym}_{n-2}U$ with $(1-\operatorname{proj})(Y)=x_1Y_0$. We apply this technique to $S_1+S_5$ with 
$q([x_j\mu](\widetilde{\Phi}))$ playing the role of $Y$. Recall that $x_1(\widetilde{\Phi})=\Phi$ and that $[q(\nu)](\Phi)=\pmb \delta \nu$ for all $\nu\in D_{n-1}(U^*)$. We see that 
\begingroup
\allowdisplaybreaks
\begin{eqnarray*}
S_1+S_5&=&x_1\otimes \sum\limits_{j,k=2}^d    ([x_k^*\wedge x_j^*](\theta))\otimes 
x_1q\left(\left[x_k(1-\operatorname{proj})q([x_j\mu](\widetilde{\Phi}))\right](\widetilde{\Phi})\right)\\
&=&x_1\otimes \sum\limits_{j,k=2}^d    ([x_k^*\wedge x_j^*](\theta))\otimes 
x_1q\left(\left[x_k\frac{(1-\operatorname{proj})q([x_j\mu](\widetilde{\Phi}))}{x_1}x_1\right](\widetilde{\Phi})\right)\\
&=&x_1\otimes \sum\limits_{j,k=2}^d    ([x_k^*\wedge x_j^*](\theta))\otimes 
x_1q\left(\left[x_k\frac{(1-\operatorname{proj})q([x_j\mu](\widetilde{\Phi}))}{x_1}\right](\Phi)\right)\\
&=&\pmb\delta x_1\otimes \sum\limits_{j,k=2}^d    ([x_k^*\wedge x_j^*](\theta))\otimes 
x_1x_k\frac{(1-\operatorname{proj})q([x_j\mu](\widetilde{\Phi}))}{x_1}\\
&=&\pmb\delta x_1\otimes \sum\limits_{j,k=2}^d    ([x_k^*\wedge x_j^*](\theta))\otimes 
x_k(1-\operatorname{proj})q([x_j\mu](\widetilde{\Phi}));
\end{eqnarray*}
\endgroup
thus, $(S_1+S_5)+(S_3+S_{11})=0$. A similar calculation yields
\begingroup
\allowdisplaybreaks
\begin{eqnarray*}
S_2+S_6&=&-x_1\otimes \sum\limits_{j,k,\ell=2}^d  (x_1\wedge  [x_\ell^*\wedge x_k^*\wedge x_j^*](\theta))\otimes 
x_\ell q\left(\left[x_k(1-\operatorname{proj})q([x_j\mu](\widetilde{\Phi}))\right](\widetilde{\Phi})\right)\\
&=&-x_1\otimes \sum\limits_{j,k,\ell=2}^d  (x_1\wedge  [x_\ell^*\wedge x_k^*\wedge x_j^*](\theta))\otimes 
x_\ell q\left(\left[x_k\frac{(1-\operatorname{proj})q([x_j\mu](\widetilde{\Phi}))}{x_1}x_1\right](\widetilde{\Phi})\right)\\
&=&-x_1\otimes \sum\limits_{j,k,\ell=2}^d  (x_1\wedge  [x_\ell^*\wedge x_k^*\wedge x_j^*](\theta))\otimes 
x_\ell q\left(\left[x_k\frac{(1-\operatorname{proj})q([x_j\mu](\widetilde{\Phi}))}{x_1}\right](\Phi)\right)\\
&=&-\pmb\delta x_1\otimes \sum\limits_{j,k,\ell=2}^d  (x_1\wedge  [x_\ell^*\wedge x_k^*\wedge x_j^*](\theta))\otimes 
x_\ell x_k\frac{(1-\operatorname{proj})q([x_j\mu](\widetilde{\Phi}))}{x_1}=0.
\end{eqnarray*}
\endgroup
We have established (\ref{goal-a!}) in Case~\ref{case4}.

\medskip\noindent Case~\ref{case5}. Let $\Theta=\omega$ for some element $\omega$ of $\bigwedge^{d-1}U_0$. We establish (\ref{goal-a!}) in Case~\ref{case5} by showing that $$X=(1\otimes 1\otimes \nu)
\left([(\tau_{d-1}\circ b_d)- (g_d\circ \tau_d)](\omega)\vphantom{\widetilde{\Phi}}\right)$$is zero, where  $\nu$ is a fixed, but arbitrary, element in $D_nU^*$. 
One quickly calculates that $X$ is equal to $\pmb\delta^{d-2}\sum_{i=1}^8X_i$, with

\begin{longtable} {l}
$X_1=\phantom{+}\sum\limits_{j=2}^d\sum\limits_{m\in\binom{x_2,\dots,x_d}{n}} 
m^*(x_jq[x_1(\nu)])\cdot
[\pmb\delta m-x_1q(m(\widetilde{\Phi}))]\otimes x_j^*(\omega)$,
\vspace{5pt}\\

$X_2=-\sum\limits_{j,k=2}^d\sum\limits_{m\in\binom{x_2,\dots,x_d}{n}} m^*(x_jq[x_k(\nu)])\cdot[\pmb\delta m-x_1q(m(\widetilde{\Phi}))]\otimes (x_1 \wedge [x_k^*\wedge x_j^*](\omega))$,
\vspace{5pt}\\

$X_3=-\pmb\delta \sum\limits_{j=2}^d\sum\limits_{m\in\binom{x_2,\dots,x_d}{n-1}} 
m[x_j(\nu)]\cdot  x_1q(m^*)\otimes x_j^*(\omega)$,
\vspace{5pt}\\ 

$X_4=\phantom{+}\sum\limits_{j=2}^d\sum\limits_{m\in\binom{x_2,\dots,x_d}{n-1}} 
 m\left((x_jq[x_1(\nu)])(\widetilde{\Phi})\right) \cdot
x_1q(m^*)\otimes x_j^*(\omega)$,
\vspace{5pt}\\

$X_5=-\sum\limits_{j,k=2}^d\sum\limits_{m\in\binom{x_2,\dots,x_d}{n-1}} 
m\left((x_jq[x_k(\nu)])(\widetilde{\Phi})\right) \cdot
x_1q(m^*)\otimes (x_1\wedge [x_k^*\wedge x_j^*](\omega))$,
\vspace{5pt}\\

$X_6=\phantom{+}\pmb\delta \sum\limits_{j=2}^d x_1   q(x_j(\nu))\otimes x_j^*(\omega)$,
\vspace{5pt}\\

$X_7=-\pmb\delta \sum\limits_{k=2}^d x_k  q(x_1(\nu)) \otimes x_k^*(\omega)$, and
\vspace{5pt}\\

$X_8=-\pmb\delta \sum\limits_{j,k=2}^d  x_kq(x_j(\nu))\otimes (x_1\wedge [x_k^*\wedge x_j^*](\omega))$.

\end{longtable}

\noindent Use the fact that $\binom{x_1,\dots,x_d}r$ is the disjoint union of $\binom{x_2,\dots,x_d}{r}$ and $x_1\binom{x_1,\dots,x_d}{r-1}$ to re-write $X_i$ as $X_i'+X_i''$, for $1\le i\le 5$, where the sum in $X_i'$ is taken over $\binom{x_1,\dots,x_d}r$ and the sum in $X_i''$ is taken over $x\binom{x_1,\dots,x_d}{r-1}$ to obtain 
\begin{longtable}{ll}
 
$X_1'=\phantom{+}\sum\limits_{j=2}^d
\left[\pmb\delta (x_jq[x_1(\nu)])-x_1q\left((x_jq[x_1(\nu)])(\widetilde{\Phi})\right)\right]\otimes x_j^*(\omega)$,
\vspace{5pt}\\

$X_2'=-\sum\limits_{j,k=2}^d\left[\pmb\delta (x_jq[x_k(\nu)])-x_1q\left((x_jq[x_k(\nu)])(\widetilde{\Phi})\right)\right]\otimes (x_1 \wedge [x_k^*\wedge x_j^*](\omega))$,
\vspace{5pt}\\

$X_3'=-\pmb\delta \sum\limits_{j=2}^d x_1q(x_j(\nu))\otimes x_j^*(\omega)$,
\vspace{5pt}\\ 

$X_4'=\phantom{+}\sum\limits_{j=2}^d
x_1q\left((x_jq[x_1(\nu)])(\widetilde{\Phi})\right)\otimes x_j^*(\omega)$,
\vspace{5pt}\\

$X_5'=-\sum\limits_{j,k=2}^d
x_1q\left((x_jq[x_k(\nu)])(\widetilde{\Phi})\right)\otimes (x_1\wedge [x_k^*\wedge x_j^*](\omega))$,
\vspace{5pt}\\

$X_1''=-\sum\limits_{j=2}^d\sum\limits_{m\in\binom{x_1,\dots,x_d}{n-1}} 
(x_1m)^*(x_jq[x_1(\nu)])\cdot
[\pmb\delta (x_1m)-x_1q((x_1m)(\widetilde{\Phi}))]\otimes x_j^*(\omega)$,
\vspace{5pt}\\

$X_2''=\phantom{+}\sum\limits_{j,k=2}^d\sum\limits_{m\in\binom{x_1,\dots,x_d}{n-1}} (x_1m)^*(x_jq[x_k(\nu)])\cdot[\pmb\delta (x_1m)-x_1q((x_1m)(\widetilde{\Phi}))]\otimes (x_1 \wedge [x_k^*\wedge x_j^*](\omega))$,
\vspace{5pt}\\

$X_3''=\phantom{+}\pmb\delta \sum\limits_{j=2}^d\sum\limits_{m\in\binom{x_1,\dots,x_d}{n-2}} 
(x_1m)[x_j(\nu)]\cdot  x_1q((x_1m)^*)\otimes x_j^*(\omega)$,
\vspace{5pt}\\ 

$X_4''=-\sum\limits_{j=2}^d\sum\limits_{m\in\binom{x_1,\dots,x_d}{n-2}} 
 (x_1m)\left((x_jq[x_1(\nu)])(\widetilde{\Phi})\right) \cdot
x_1q((x_1m)^*)\otimes x_j^*(\omega)$, and 
\vspace{5pt}\\

$X_5''=\phantom{+}\sum\limits_{j,k=2}^d\sum\limits_{m\in\binom{x_1,\dots,x_d}{n-2}} 
(x_1m)\left((x_jq[x_k(\nu)])(\widetilde{\Phi})\right) \cdot
x_1q((x_1m)^*)\otimes (x_1\wedge [x_k^*\wedge x_j^*](\omega))$.

\end{longtable}

\noindent Observe that $X_1'+X_4'+X_7=X_2'+X_5'+X_8=X_3'+X_6=X_3''+X_4''=X_1''=X_2''=X_5''=0$. For example, the coefficient $(x_1m)\left((x_jq[x_1(\nu)])(\widetilde{\Phi})\right)$, from $X_4''$, is equal to
$$(mx_j)\left((q[x_1(\nu)])(x_1(\widetilde{\Phi}))\right)=(mx_j)\left((q[x_1(\nu)])(\Phi)\right)
=(mx_j)\left(\pmb\delta x_1(\nu)\right)=\pmb \delta (mx_jx_1)(\nu),$$and this explains the equation 
$X_3''+X_4''=0$. A similar calculation shows that $X_5''$ is both symmetric and alternating in the symbols $x_j$ and $x_k$. The factor $\pmb\delta (x_1m)-x_1q((x_1m)(\widetilde{\Phi}))$ of $X_1''$ and $X_2''$ is zero. At any rate,
 $X$ is equal to zero and the proof is complete.
\end{proof}

\section{\bf The main theorem.}\label{main-theorem}
Theorem~\ref{thm-1} is the heart of the paper. This is where we identify a basis for each $(G_r)_{\pmb \delta}$. The complex $\mathbb G_{\pmb\delta}$ of \cite{EKK} and Theorem 3.4 has all of the desired properties, except it is not clear exactly what $\mathbb G_{\pmb\delta}$ is. In Theorem~\ref{thm-1} we determine a precise description of $\mathbb G_{\pmb\delta}$. Reader-friendly reformulations of Theorem~\ref{thm-1} are recorded as Corollaries~\ref{main-lemma} and \ref{main} at the end of the section. 
The main result in the paper is Corollary~\ref{main} which states that if 
$A$ is  a standard-graded Artinian Gorenstein algebra over a field $\pmb k$ whose minimal resolution is Gorenstein-linear, then there is a ring homomorphism $\rho:\mathfrak R\to \operatorname{Sym}_{\bullet}^{\pmb k} A_1$ so that $\rho\otimes \mathbb B$ is the minimal homogeneous resolution of $A$, where the maps and modules of $\mathbb B$ are explicitly given in Section \ref{describe-B}.
\begin{comment}

the determinant of (\ref{crit-det}) is non-zero, then 
${\mathbf B}=S\otimes_{\rho} \mathbb B$ is a minimal homogeneous resolution of $S/\operatorname{ann}(\phi)$ 
by free $S$-modules.

Let $\pmb k$ be a field, $A$  a standard-graded Artinian Gorenstein $\pmb k$-algebra, $S$   the standard-graded polynomial ring $\operatorname{Sym}_{\bullet}^{\pmb k}A_1$, $I$ the kernel of the natural map $\xymatrix{S\ar@{->>}[r]&A}$, $d$  the vector space dimension $\dim_{\pmb k}A_1$, and $n$  the least index with $I_n\neq 0$. 
Assume that $3\le d$ and $2\le n$. Let $(\mathbb B,b)$ be  
the  complex  of Definition~{\rm{\ref{diff}}} and Corollary~{\rm\ref{main-lemma}.\ref{main-lemma-e}}, $\alpha:\pmb k\otimes U\to S_1$ any vector space isomorphism, $\phi\in D_{2n-2}^{\pmb k}S_1$  a Macaulay inverse system for $A$, and $\rho: {\mathfrak R}\to S$ the ring homomorphism of Proposition~{\rm\ref{transition}.\ref{trans-a}} which corresponds to the data $(\alpha,\phi)$. If  
the minimal homogeneous resolution of $A$ by free $S$-modules is Gorenstein-linear, then ${\mathbf B}=S\otimes_{\rho} \mathbb B$ is a minimal homogeneous resolution of $A$ 
by free $S$-modules.
Furthermore, ${\mathbf B}$ is explicitly constructed in a polynomial manner from the coefficients of
$\phi$.
\end{comment}

In the $d=3$ case, the two formulations, \ref{thm-1} and \ref{main-lemma}, are gathered together as \cite[Lem.~4.6]{EK-K-2} and the formulation \ref{main} is called \cite[Thm.~4.1]{EK-K-2}.
\begin{theorem}\label{thm-1} Adopt Data~{\rm\ref{data1}} 
 and recall the free ${\mathfrak R}$-modules $B_r$ and $G_r$ of Definitions~{\rm\ref{diff}.\ref{def1-a}} and {\rm\ref{defG}} and 
the ${\mathfrak R}$-module homomorphisms $\tau_r: B_r\to G_r$ of Definition~{\rm \ref{def1}} and Observation~{\rm \ref{ob2}}.  Then, for each $r$ with $0\le r\le d$, 
\begin{enumerate}[\rm(a)]
\item\label{thm-1-a} $\tau_r(B_r)$ is a free ${\mathfrak R}$-submodule of $G_r$,
\item\label{thm-1-b} $\tau_r: B_r \to \tau_r(B_r)$ is an isomorphism, and
\item\label{thm-1-c} $(\tau_rB_r)_{\pmb \delta}=(G_r)_{\pmb \delta}$.\end{enumerate} 
 \end{theorem}

\begin{proof}The assertions are obvious when $r=0$ or $r=d$. Henceforth, we fix an integer $r$, with $1\le r\le d-1$.
Recall, from Definition~\ref{defG}.\ref{back2}, that $G_r=\ker\left(\underline{\mathrm v}_r:{\mathfrak R}\otimes L_{r-1,n}U\to {\mathfrak R}\otimes K_{r-1,n-2}U\right)$. We showed in Observation~\ref{ob2} that $\tau_r(B_r)\subseteq G_r$.
We
identify a $\mathbb Z$-free sub-module $$\textstyle C_{r-1,n-1}\text{ of }\bigwedge^{r-1}U_0\otimes D_{n-1}(U^*)$$ and an ${\mathfrak R}$-module homomorphism $$\zeta_{r-1}: {\mathfrak R} \otimes  C_{r-1,n-1}\to {\mathfrak R} \otimes L_{r-1,n}U$$ so that

\begin{chunk}\label{Xplan-a} $\bigwedge^{r-1}U_0\otimes D_{n-1}(U^*)=C_{r-1,n-1}\oplus K_{r-1,n-1}U_0$ as an internal direct sum of free ${\mathbb Z}$-modules,\end{chunk}

\begin{chunk}\label{Xplan-b} the ${\mathfrak R}_{\pmb\delta}$-module homomorphism $$ ({\mathfrak R}_{\pmb\delta}\otimes C_{r-1,n-1})\oplus (B_r)_{\pmb\delta}\xrightarrow{\bmatrix \zeta_{r-1}&\tau_r\endbmatrix} {\mathfrak R}_{\pmb\delta} \otimes L_{r-1,n}U$$is an isomorphism of free ${\mathfrak R}_{\pmb\delta}$-modules,\end{chunk}

\begin{chunk}\label{Xplan-c} $\underline{\mathrm v}_r$ carries $\zeta_{r-1}({\mathfrak R}_{\pmb\delta} \otimes  C_{r-1,n-1})$ isomorphically onto ${\mathfrak R}_{\pmb\delta}\otimes K_{r-1,n-2}U$, and \end{chunk}

\begin{chunk}\label{Xplan-d} $\underline{\mathrm v}_r$ carries $\tau_r(B_r)$ to zero.\end{chunk}

\noindent It is clear that (\ref{Xplan-a})~--~(\ref{Xplan-d}) complete the proof of Theorem~\ref{thm-1}; indeed, once (\ref{Xplan-a})~--~(\ref{Xplan-d}) are established, then the diagram
$$\xymatrix{(\mathfrak R_{\pmb \delta}\otimes C_{r-1,n-1})\oplus (B_r)_{\pmb\delta}\ar[rrr]^{\bmatrix \zeta_{r-1}&\tau_r\endbmatrix}_{\simeq}\ar[rrrdd]_{\bmatrix \simeq&0\endbmatrix}&&&\mathfrak R_{\pmb \delta}\otimes L_{r-1,n}U\ar[dd]^{{\underline{\mathrm v}}_r}\\\\&&&\mathfrak R_{\pmb \delta}\otimes K_{r-1,n-2}U}$$ commutes and each of the four listed modules is a free $\mathfrak R_{\pmb\delta}$-module.
 We define $C_{r-1,n-1}$ and $\zeta_{r-1}$ in Definitions~\ref{{def2}} and \ref{def4}. 
We establish (\ref{Xplan-a}), (\ref{Xplan-b}), and (\ref{Xplan-c}) 
in Corollary~\ref{cor1}\ref{cor1-b}, Lemma~\ref{lem3}, and Observation~\ref{bij}, 
respectively. As noted above, the assertion (\ref{Xplan-d}) is equivalent to the assertion that $\tau_r(B_r)\subseteq G_r$ and therefore, (\ref{Xplan-d}) has already been established in Observation~\ref{ob2}.
\end{proof}

\begin{remarks}\label{rmk1} We note two strange features about the free ${\mathbb Z}$-module $C_{r-1,n-1}$. \begin{enumerate}[\rm(a)]\item Two different modules, $U_0$ and $U$, are used in the construction of $C_{r-1,n-1}$. For this reason we did not decorate the name with either a $U_0$ or a $U$. We think of ``$C$'' as standing for complement.  \item The module $C_{r-1,n-1}$ is dependent on a choice of  basis for $U_0$. On the other hand, $C_{r-1,n-1}$ is merely a tool for proving the assertions of Theorem~\ref{thm-1} and Theorem~\ref{thm-1}   does not depend on a choice of basis for $U_0$.
\end{enumerate}
\end{remarks}
We have carried out a similar proof twice before. It might be helpful to consult \cite[just above (6.26)]{EKK}, where we apply column operations in order to identify a basis for $\ker (\underline{\mathrm v}_r)_{\pmb\delta}$ (when $d=3$ and $n=2$), or \cite[Lem.~4.6.a]{EK-K-2}, where we consider a series of bases for 
${\mathfrak R}_{\pmb\delta}\otimes L_{r-1,n}U$ in order to identify a basis for $\ker (\underline{\mathrm v}_r)_{\pmb\delta}$ (when $d=3$ and $2\le n$ is arbitrary). 
The present argument for identifying a basis for $\ker (\underline{\mathrm v}_r)_{\pmb\delta}$ amounts to decomposing the Schur and Weyl modules of a direct sum into a sum of tensor products of Schur and Weyl modules of the summands. In practice, we only deal with hooks, and one of our summands has rank one.

The next step in the proof of Theorem~\ref{thm-1} is to introduce the sub-module $$C_{r-1,n-1}\text{ of } {\textstyle\bigwedge^{r-1}U_0\otimes D_{n-1}(U^*)}.$$ To that end, we recall the standard basis for $K_{i,j}U$ and we reformulate that basis in a language 
that is compatible with our conventions.

\begin{remark}\label{rmk2} Adopt Data~{\rm\ref{data1}} and Convention~{\rm\ref{conv1}}. Let $i$ and $j$ be positive integers. One basis for $K_{i,j}U$
is 
$$\{\eta\left((x_1\wedge \dots \wedge x_{\ell}\wedge x_{\lambda_1}\wedge\dots \wedge x_{\lambda_{i+1-\ell}})\otimes (x_{\ell}m)^*\right) \in \textstyle\bigwedge^iU\otimes  D_jU^*\mid \text{
$(\ell,\underline \lambda,m)$ satisfy
(\ref{eq2})}\}$$for
\begin{equation}\label{eq2}\textstyle\{(\ell,\underline \lambda,m)\mid 1\le \ell\le d,\quad  \ell<\lambda_1<\dots <\lambda_{i+1-\ell}\le d,\quad \text{and}\quad m\in \binom{x_{\ell},\dots,x_{d}}{j}\}.  \end{equation}(It is our intention that each basis element $x_k$, with $1\le k \le\ell$, appears in the product 
$x_1\wedge \dots \wedge x_{\ell}$, in ascending order.)
\end{remark}
\begin{proof} Recall from \cite[(5.5)]{EKK} (or Examples 2.1.3.h and 2.1.17.h in \cite{W} or \cite[Sect.III.1]{BB}) that the ``standard basis'' for $K_{i,j}U$ is 
\begin{equation}\label{eq3}\left\{k_{\pmb a;\pmb b}\mid \pmb a\text{ is }a_1<\dots<a_{d-i-1},\quad \pmb b\text{ is }b_1\le \dots\le b_{j+1},\quad\text{and}\quad b_1< a_1 \right\},\end{equation}
where 
\begin{equation}\label{eq4} k_{\pmb a;\pmb b}=\eta\left((x_{a_1}^*\wedge\dots\wedge x_{a_{d-i-1}}^*)(\omega)\otimes {x_1^*}^{(\beta_1)}\cdot \ldots\cdot  {x_d^*}^{(\beta_d)}\right)\in K_{i,j}U\subseteq {\textstyle\bigwedge}^i U\otimes D_{j}(U^*),\end{equation}for
$$\pmb b=(\underbrace{1,\ldots,1}_{\beta_1},\underbrace{2,\ldots,2}_{\beta_2},\cdots,\underbrace{d,\ldots,d}_{\beta_d}), {\textstyle\text{ with $\sum\beta_i=j+1$}},$$and $\omega$ is a basis for $\bigwedge^dU$.  
Notice that if $k_{\pmb a;\pmb b}$, as described in (\ref{eq4}), is an element of (\ref{eq3}), then
$$\left((x_{a_1}^*\wedge\dots\wedge x_{a_{d-i-1}}^*)(\omega)\otimes {x_1^*}^{(\beta_1)}\cdot \ldots\cdot  {x_d^*}^{(\beta_d)}\right)=\pm(x_1\wedge \cdots\wedge x_{b_1}\wedge x_{\lambda_1}\wedge \cdots \wedge x_{\lambda_{i+1-b_1}}\otimes (x_{b_1}m)^*$$ in $\textstyle \bigwedge^{i+1}U\otimes D_{j+1}(U^*)$, 
for some $(\underline{\lambda},m)$ with
$$\textstyle b_1<\lambda_1<\dots<\lambda_{i+1-b_1}\le d\quad \text{and}\quad m\in \binom{x_{b_1},\dots,x_d}{j}.$$
\end{proof}

\begin{definition}\label{{def2}}Adopt Data~{\rm\ref{data1}} and Convention~{\rm\ref{conv1}}. For each pair of positive integers $i$ and $j$, define $C_{i,j}$ to be the following ${\mathbb Z}$-free  sub-module of $\bigwedge^iU_0\otimes D_{j}(U^*)$:
$$C_{i,j}= \bigoplus_{(\ell,\underline {\lambda},m)}
\mathbb Z  
\left((x_2\wedge \dots \wedge x_{\ell}\wedge x_{\lambda_1}\wedge\dots \wedge x_{\lambda_{i+1-\ell}})\otimes (x_{\ell}m)^*\right) 
,$$
where the sum is taken over 
\begin{equation}\label{eq5}\textstyle\{(\ell,\underline \lambda,m)\mid 1\le \ell\le d,\quad  \ell<\lambda_1<\dots <\lambda_{i+1-\ell}\le d,\quad \text{and}\quad m\in \binom{x_{\ell},\dots,x_{d}}{j-1}\}.\end{equation} If $2\le \ell$, then  each basis element $x_k$, with $2\le k \le\ell$, appears in the product 
$x_2\wedge \dots \wedge x_{\ell}$, in ascending order. If $\ell<2$, then $x_2\wedge \dots \wedge x_{\ell}$
represents $1$. \end{definition}

 \begin{corollary}\label{cor1}Adopt Data~{\rm\ref{data1}}, let $i$ and $j$ be positive integers, and recall the ${\mathbb Z}$-free sub-module $C_{i,j}$ of $\bigwedge^iU_0\otimes D_j (U^*)$ from Definition~{\rm\ref{{def2}}}.
\begin{enumerate}[\rm(a)] \item The ${\mathbb Z}$-module  $\bigwedge^iU_0\otimes D_j (U_0^*)$
 is the direct sum 
$$K_{i,j}U_0\oplus  \left(\bigoplus_{(\ell,\underline {\lambda},m)}
\mathbb Z  
\left((x_2\wedge \dots \wedge x_{\ell}\wedge x_{\lambda_1}\wedge\dots \wedge x_{\lambda_{i+1-\ell}})\otimes (x_{\ell}m)^*\right) 
\right),$$
where the sum is taken over 
\begin{equation}\label{eq6}\textstyle\{(\ell,\underline \lambda,m)\mid 2\le \ell\le d,\quad  \ell<\lambda_1<\dots <\lambda_{i+1-\ell}\le d,\quad \text{and}\quad m\in \binom{x_{\ell},\dots,x_{d}}{j-1}\}.\end{equation}
\item\label{cor1-b} The ${\mathbb Z}$-module  $\bigwedge^iU_0\otimes D_j (U^*)$
 is the direct sum $K_{i,j}U_0\oplus C_{i,j}$. 
\end{enumerate}
\end{corollary} 
\begin{Remark} Item~(\ref{cor1-b}) establishes (\ref{Xplan-a}).
\end{Remark}
\begin{proof} (a) One knows that 
$$\textstyle 0\to K_{i,j}U_0\hookrightarrow \bigwedge^iU_0\otimes D_j (U_0^*)\xrightarrow{\ \eta\ } K_{i-1,j-1}U_0\to 0$$is a split exact sequence and, according to Remark~{\rm\ref{rmk2}}. 
$K_{i-1,j-1}U_0$ is a free ${\mathbb Z}$-module with basis 
$$\{\eta\left((x_2\wedge \dots \wedge x_{\ell}\wedge x_{\lambda_1}\wedge\dots \wedge x_{\lambda_{i+1-\ell}})\otimes (x_{\ell}m)^*\right)\mid \text{$(\ell,\underline \lambda,m)$ satisfy (\ref{eq6})}\}.$$

\medskip\noindent(b) A monomial in $\{x_1,\dots,x_d\}$ either contains $x_1$ or it doesn't; and therefore, the ${\mathbb Z}$-module $D_j(U^*)$ decomposes as 
$$D_j(U^*)=\bigoplus_{m\in \binom{x_1,\dots,x_d}{j-1}}{\mathbb Z}((x_1m)^*)\oplus D_j(U_0^*).$$In light of (a), it suffices  to observe that 
\begin{equation}\label{eq7}{\textstyle \bigwedge^iU_0}\otimes \bigoplus_{m\in \binom{x_1,\dots,x_d}{j-1}}{\mathbb Z}((x_1m)^*)=
\bigoplus_{(1,\underline {\lambda},m)}
\mathbb Z  
\left((x_2\wedge \dots \wedge x_{\ell}\wedge x_{\lambda_1}\wedge\dots \wedge x_{\lambda_{i+1-\ell}})\otimes (x_{\ell}m)^*\right) 
,\end{equation} with $(1,\underline {\lambda},m)$ in  (\ref{eq5}). In particular, 
$(x_2\wedge \dots \wedge x_{\ell}\wedge x_{\lambda_1}\wedge\dots \wedge x_{\lambda_{i+1-\ell}})\otimes (x_{\ell}m)^*$ from (\ref{eq7}), with $\ell=1$,  is equal to $(x_{\lambda_1}\wedge\dots \wedge x_{\lambda_{i}})\otimes (x_{1}m)^*$, with $2\le \lambda_1$ and $m\in\binom{x_1,\dots,x_{d}}{j-1}$.
\end{proof}
Our next objective is to decompose ${\mathfrak R}_{\pmb\delta} \otimes L_{i,j}U$ as a direct sum of $\tau_r((B_r)_{\pmb\delta})$ and a complementary summand, as described in (\ref{Xplan-b}). The complementary summand is, of course, $\zeta_{r-1}({\mathfrak R}\otimes C_{r-1,n-1})$. We treat the coordinate-free aspects of the decomposition in Definition~\ref{def3} and Lemmas~\ref{lem1} and \ref{lem2}, before defining $\zeta_{r-1}$, whose domain is the coordinate-dependent module ${\mathfrak R}\otimes C_{r-1,n-1}$, in Definition~\ref{def4}. 
\begin{definition}\label{def3} Adopt Data~{\rm\ref{data1}}. For any pair of indices $(i,j)$, define the ${\mathbb Z}$-module homomorphism 
$$\textstyle \sigma_{i,j}:\bigwedge ^iU_0\otimes \operatorname{Sym}_j U\to L_{i,j+1}U$$ by 
$$\sigma_{i,j}(\theta\otimes \mu)= \kappa((x_1\wedge \theta)\otimes \mu),$$ for $\theta\in \bigwedge ^iU_0$ and $\mu\in \operatorname{Sym}_j U$.
\end{definition}

\begin{lemma}\label{lem1} Adopt Data~{\rm\ref{data1}}$;$ let $(i,j)$ be a pair of indices$;$ and recall the homomorphism $\sigma$ of Definition~{\rm\ref{def3}}.
Then the ${\mathbb Z}$-module homomorphism
$$\textstyle  (\bigwedge^i U_0\otimes \operatorname{Sym}_{j-1} U)\oplus L_{i,j}U_0\xrightarrow{\bmatrix \sigma_{i,j-1}
&\text{\rm inclusion}\endbmatrix} L_{i,j}U.$$  is an isomorphism.
 \end{lemma}

\begin{proof} The standard basis for $L_{i,j}U$ is 
\begin{equation}\label{eq8}\left\{\ell_{\pmb a;\pmb b}\mid  \pmb a\text{ is }a_1<\dots<a_{i+1},\quad \pmb b\text{ is }b_1\le \dots\le b_{j-1},\quad \text{and}\quad a_1\le b_1 \right\}\end{equation}
where $$\ell_{\pmb a;\pmb b}=\kappa(x_{a_1}\wedge \dots\wedge x_{a_{i+1}}\otimes x_{b_1}\cdot \ldots\cdot x_{b_{j-1}})\in L_{i,j}U\subseteq {\textstyle\bigwedge}^i U\otimes  \operatorname{Sym}_jU.$$
(See, for example, \cite[(5.5)]{EKK}, or Examples 2.1.3.h and 2.1.17.h in \cite{W}, or \cite[Sect.III.1]{BB}.)
Partition (\ref{eq8}) as $S_1\cup S_2$, where
$$S_1=\{\ell_{\pmb a;\pmb b}\mid a_1=1\}\quad\text{and}\quad S_2=\{\ell_{\pmb a;\pmb b}\mid 2\le a_1\}.$$
Observe that the map $\sigma_{i,j-1}$ carries a basis for $\bigwedge^i U_0\otimes \operatorname{Sym}_{j-1}U$ bijectively onto $S_1$ and $S_2$ is a basis for $L_{i,j}U_0$.  
\end{proof}

\begin{lemma}\label{lem2}Adopt Data~{\rm\ref{data1}}. Fix an integer $r$ with $1\le r\le d-1$. Recall the ${\mathfrak R}$-free summand  ${\mathfrak R}\otimes L_{r-1,n}U_0$  of the module $B_{r}$ from Definition {\rm\ref{diff}.\ref{def1-a}},  the ${\mathfrak R}$-module homomorphism 
$$\tau_{r}:{\mathfrak R} \otimes L_{r-1,n}U_0\to{\mathfrak R} \otimes L_{r-1,n}U$$ of Definition {\rm\ref{def1}.\ref{def1-c}}, and the ${\mathfrak R}$-module homomorphism 
$$\textstyle \sigma_{r-1,n-1}:{\mathfrak R} \otimes \bigwedge ^{r-1}U_0\otimes \operatorname{Sym}_{n-1} U\to {\mathfrak R} \otimes L_{r-1,n}U$$ 
of Definition {\rm\ref{def3}}.
Then the ${\mathfrak R}_{\pmb \delta}$-module homomorphism 
\begin{equation}\label{eq9}\textstyle({\mathfrak R}_{\pmb \delta}\otimes \bigwedge^{r-1}U_0\otimes \operatorname{Sym}_{n-1}U)\oplus ({\mathfrak R}_{\pmb \delta}\otimes L_{r-1,n}U_0)
\xrightarrow{\bmatrix \sigma_{r-1,n-1}&\tau_{r}\endbmatrix}{\mathfrak R}_{\pmb \delta}\otimes L_{r-1,n}U\end{equation}
is an isomorphism.
\end{lemma} 

\begin{proof} It is shown in Lemma~\ref{lem1} that 
$$\textstyle({\mathfrak R}\otimes \bigwedge^{r-1}U_0\otimes \operatorname{Sym}_{n-1}U)\oplus ({\mathfrak R}\otimes L_{r-1,n}U_0)
\xrightarrow{\bmatrix \sigma_{r-1,n-1}&\rm{incl}\endbmatrix}
{\mathfrak R}\otimes L_{r-1,n}U$$ 
 is an isomorphism.

Define $$\textstyle \xi: {\mathfrak R}_{\pmb\delta}\otimes \bigwedge^{r-1}U_0\otimes \operatorname{Sym}_nU_0\to
{\mathfrak R}_{\pmb\delta}\otimes \bigwedge^{r-1}U_0\otimes  \operatorname{Sym}_{n-1}U$$
to be the ${\mathfrak R}_{\pmb\delta}$ homomorphism with  $$\xi(\theta\otimes \mu)=-\theta\otimes q(\mu(\widetilde{\Phi})),$$ for $\theta\in \bigwedge^{r-1}U_0$ and $\mu\in \operatorname{Sym}_nU_0$. Observe that 
the map (\ref{eq9}) is the composition of two isomorphisms:
$$\begin{array}{lcc}
\begin{array}{c}\textstyle({\mathfrak R}_{\pmb \delta}\otimes \bigwedge^{r-1}U_0\otimes \operatorname{Sym}_{n-1}U)\\\oplus \\({\mathfrak R}_{\pmb \delta}\otimes L_{r-1,n}U_0)\end{array}
&\xrightarrow{\bmatrix 1&\pmb\delta^{r-1}\xi\\0&\pmb\delta^r\endbmatrix}&
\begin{array}{c}({\mathfrak R}_{\pmb \delta}\otimes \bigwedge^{r-1}U_0\otimes \operatorname{Sym}_{n-1}U)\\\oplus\\ ({\mathfrak R}_{\pmb \delta}\otimes L_{r-1,n}U_0)\end{array}
\vspace{5pt}\\&\xrightarrow{\bmatrix \sigma_{r-1,n-1}&\rm{incl}\endbmatrix}&
{\mathfrak R}_{\pmb \delta}\otimes L_{r-1,n}U.\end{array}$$
\end{proof}

\begin{definition}\label{def4} Fix an integer $r$, with $1\le r\le d-1$. Adopt Data~{\rm\ref{data1}} and Convention~{\rm\ref{conv1}}. Recall the $\mathbb Z$-free sub-module $C_{r-1,n-1}$ of  $\bigwedge^{r-1}U_0\otimes D_{n-1} (U^*)$
from Definition~{\rm\ref{{def2}}} and the homomorphism $\sigma_{r-1,n-1}$  of Definition~{\rm\ref{def3}}.
Define $$\zeta_{r-1}: {\mathfrak R} \otimes C_{r-1,n-1}\to {\mathfrak R} \otimes L_{r-1,n}U$$ to be the composition
$$\textstyle \begin{array}{lcc}{\mathfrak R} \otimes C_{r-1,n-1}\xrightarrow{\ \text{incl}\ }
{\mathfrak R} \otimes \bigwedge^{r-1}U_0\otimes D_{n-1}(U^*)&\xrightarrow{\ 1\otimes1\otimes q\ }&
{\mathfrak R} \otimes \bigwedge^{r-1}U_0\otimes \operatorname{Sym}_{n-1}U\vspace{5pt}\\&\xrightarrow{\ 1\otimes \sigma_{r-1,n-1}\ }& {\mathfrak R} \otimes L_{r-1,n}U.\end{array}$$
\end{definition}

\begin{lemma}\label{lem3} Adopt Data~{\rm\ref{data1}} and Convention~{\rm\ref{conv1}}$;$ let $r$ be an integer with $1\le r\le d-1;$ recall the free ${\mathfrak R}$-modules $B_r$ and ${\mathfrak R}\otimes C_{r-1,n-1}$ of Definitions~{\rm\ref{diff}.\ref{def1-a}} and {\rm\ref{cor1}} and the ${\mathfrak R}$-module homomorphisms $\tau_r$ and $\zeta_{r-1}$ of
Definitions~{\rm\ref{def1}.\ref{def1-b}}, {\rm\ref{def1}.\ref{def1-c}}, and {\rm\ref{def4}}. Then the ${\mathfrak R}_{\pmb \delta}$-module homomorphism \begin{equation}\label{20.22.1} ({\mathfrak R}_{\pmb \delta}\otimes C_{r-1,n-1})\oplus (B_r)_{\pmb\delta}\xrightarrow{\bmatrix \zeta_{r-1}&\tau_r\endbmatrix} {\mathfrak R}_{\pmb \delta} \otimes L_{r-1,n}U\end{equation}is an isomorphism.
\end{lemma}\begin{Remark} Lemma~\ref{lem3} establishes (\ref{Xplan-b}).\end{Remark}

\begin{proof}Recall from Definition~\ref{diff}.\ref{def1-a} 
and Corollary~\ref{cor1}.\ref{cor1-b} that
$$ \textstyle (B_r)_{\pmb\delta}=\left\{\begin{array}{c}{\mathfrak R}_{\pmb\delta}\otimes K_{r-1,n-1}U_0\\\oplus\\{\mathfrak R}_{\pmb\delta}\otimes L_{r-1,n}U_0,\end{array}\right.\qquad\text{and}\qquad \left.\begin{array}{c}{\mathfrak R}_{\pmb\delta}\otimes C_{r-1,n-1}
\\\oplus\\{\mathfrak R}_{\pmb\delta}\otimes K_{r-1,n-1}U_0\end{array}\right\}={\mathfrak R}_{\pmb\delta}\otimes \bigwedge^{r-1}U_0\otimes D_{n-1}(U^*).$$
It follows that the composition 
\begingroup
\allowdisplaybreaks\begin{eqnarray}\label{equality}
\begin{array}{c}{\mathfrak R}_{\pmb \delta}\otimes C_{r-1,n-1}\\\oplus\\ (B_r)_{\pmb\delta}\end{array}&\xrightarrow{\ =\ } &\begin{array}{c}
{\mathfrak R}_{\pmb \delta}\otimes C_{r-1,n-1}\\\oplus\\
{\mathfrak R}_{\pmb\delta}\otimes K_{r-1,n-1}U_0\\\oplus\\{\mathfrak R}_{\pmb\delta}\otimes L_{r-1,n}U_0\end{array} 
\xrightarrow{\bmatrix 1&0&0\\0&\pmb\delta^{r-1}&0\\0&0&1\endbmatrix}\vspace{15pt}\\
\begin{array}{c}
{\mathfrak R}_{\pmb \delta}\otimes C_{r-1,n-1}\\\oplus\\
{\mathfrak R}_{\pmb\delta}\otimes K_{r-1,n-1}U_0\\\oplus\\{\mathfrak R}_{\pmb\delta}\otimes L_{r-1,n}U_0\end{array}
&\xrightarrow{\  =\ }  &
\begin{array}{c}
{\mathfrak R}_{\pmb\delta}\otimes \bigwedge^{r-1}U_0\otimes D_{n-1}(U^*)
\\\oplus\\{\mathfrak R}_{\pmb\delta}\otimes L_{r-1,n}U_0\end{array}\nonumber\end{eqnarray}\endgroup
is an isomorphism. On the other hand,
Remark~\ref{rmk3} and Lemma~\ref{lem2} yield that both of the following maps are isomorphisms:
\begin{equation}\label{isomorphism}
\begin{array}{ccc}
\begin{array}{c}
{\mathfrak R}_{\pmb\delta}\otimes \bigwedge^{r-1}U_0\otimes D_{n-1}(U^*)
\\\oplus\\{\mathfrak R}_{\pmb\delta}\otimes L_{r-1,n}U_0\end{array}
&\xrightarrow{\bmatrix q&0\\0&1\endbmatrix}&
\begin{array}{c}
{\mathfrak R}_{\pmb\delta}\otimes \bigwedge^{r-1}U_0\otimes \operatorname{Sym}_{n-1}U
\\\oplus\\{\mathfrak R}_{\pmb\delta}\otimes L_{r-1,n}U_0\end{array}\vspace{15pt}\\&\xrightarrow{\bmatrix\sigma_{r-1,n-1},\tau_r\endbmatrix}& {\mathfrak R}_{\pmb\delta}\otimes L_{r-1,n}U.\end{array}\end{equation}Combine the isomorphisms (\ref{equality}) and (\ref{isomorphism})  to produce an isomorphism
\begin{equation}\label{combine}({\mathfrak R}_{\pmb \delta}\otimes C_{r-1,n-1})\oplus (B_r)_{\pmb\delta}\xrightarrow{\ (\ref{isomorphism})\circ(\ref{equality})\ }
 {\mathfrak R}_{\pmb \delta} \otimes L_{r-1,n}U\end{equation}
Observe that  
(\ref{combine}) and (\ref{20.22.1}) agree on all three summands
$$ {\mathfrak R}_{\pmb \delta}\otimes C_{r-1,n-1}, \quad
{\mathfrak R}_{\pmb\delta}\otimes K_{r-1,n-1}U_0,\quad \text{and}\quad {\mathfrak R}_{\pmb\delta}\otimes L_{r-1,n}U_0$$
 of the domain.
\end{proof}

\begin{lemma}\label{bij} Adopt Data~{\rm\ref{data1}} and Convention~{\rm\ref{conv1}}$;$ fix an integer $r$ with $1\le r\le d-1;$ and recall the free $\mathbb Z$-module $C_{r-1,n-1}$ of Definition~{\rm\ref{{def2}}} and the ${\mathfrak R}$-module homomorphism $$\zeta_{r-1}: {\mathfrak R} \otimes C_{r-1,n-1}\to {\mathfrak R} \otimes L_{r-1,n}U$$ of Definition~{\rm\ref{def4}}. Then the 
${\mathfrak R}_{\delta}$-module homomorphism $$(\underline{\mathrm v}_r\circ \zeta_{r-1}): {\mathfrak R}_{\pmb\delta}\otimes C_{r-1,n-1}\to 
{\mathfrak R}_{\pmb\delta}\otimes K_{r-1,n-2}U$$ is an isomorphism.\end{lemma}
\begin{remark} Lemma~\ref{bij} establishes (\ref{Xplan-c}).\end{remark}
\begin{proof} Let $\Theta$ be  a basis element  of $C_{r-1,n-1}$. So $\Theta=\theta\otimes \nu$, with $\theta=x_2\wedge \dots \wedge x_{\ell}\wedge x_{\lambda_1}\wedge\dots \wedge x_{\lambda_{r-\ell}}$, and $\nu=(x_{\ell}m)^*$. The parameters  $(\ell,\underline\lambda,m)$ satisfy 
$$\textstyle 1\le \ell\le d,\quad  \ell<\lambda_1<\dots <\lambda_{r-\ell}\le d,\quad \text{and}\quad m\in \binom{x_{\ell},\dots,x_{d}}{n-2},$$ as described in
 (\ref{eq5}). Use Observation~\ref{ob1} to see  that 
$$\begin{array}{ll}
&(\underline{\mathrm v}_r\circ \zeta_{r-1})\left(x_2\wedge \dots \wedge x_{\ell}\wedge x_{\lambda_1}\wedge\dots \wedge x_{\lambda_{r-\ell}}\otimes ((x_{\ell}m)^*)\right)\vspace{5pt}\\=&
(\underline{\mathrm v}_r\circ \zeta_{r-1})(\theta\otimes \nu)=\underline{\mathrm v}_r\left(\kappa\left(\vphantom{\widetilde{\Phi}}(x_1\wedge \theta)\otimes q(\nu)\right)\right)
=\pmb \delta\eta((x_1\wedge \theta)\otimes \nu)\vspace{5pt}\\
=&\pmb \delta\eta\left((x_1\wedge \dots \wedge x_{\ell}\wedge x_{\lambda_1}\wedge\dots \wedge x_{\lambda_{r-\ell}})\otimes (x_{\ell}m)^*\right) 
. \end{array}$$
The basis 
$$\left\{\begin{array}{l}\eta\left((x_1\wedge \dots \wedge x_{\ell}\wedge x_{\lambda_1}\wedge\dots \wedge x_{\lambda_{r-\ell}})\otimes (x_{\ell}m)^*\right) \\\hfill\in \textstyle\bigwedge^{r-1}U\otimes  D_{n-2}(U^*)\end{array}\left\vert 
\begin{array}{l}
1\le \ell\le d,\\  \ell<\lambda_1<\dots <\lambda_{r-\ell}\le d,\text{ and}\\  m\in \binom{x_{\ell},\dots,x_{d}}{n-2}\end{array}\right.\right\}$$
for $K_{r-1,n-2}U$ is given in Remark~\ref{rmk2}.
We see that $\underline{\mathrm v}_r\circ \zeta_{r-1}$ carries the basis for  ${\mathfrak R}_{\delta}\otimes C_{r-1,n-1}$ bijectively onto a basis of ${\mathfrak R}_{\delta}\otimes K_{r-1,n-2}U$. 
\end{proof}

\begin{corollary}\label{main-lemma}Adopt Data~{\rm\ref{data1}} 
 and recall the free ${\mathfrak R}$-modules $B_r$ and $G_r$ of Definitions~{\rm\ref{diff}.\ref{def1-a}} and {\rm\ref{defG}} and 
the ${\mathfrak R}$-module homomorphisms $\tau_r: B_r\to G_r$ of Definition~{\rm \ref{def1}} and Observation~{\rm \ref{ob2}}.   For $0\le r\le d$, define $E_r=\tau_r(B_r)$, and, for $1\le r\le d$, let $e_r$ be the restriction of $g_r:G_r\to G_{r-1}$ to $E_r$.  The following statements hold.

\begin{enumerate}[\rm(a)]
\item\label{main-lemma-a}Each module $E_r$ is a free ${\mathfrak R}$-submodule of $G_r$.
\item\label{main-lemma-c} Each ${\mathfrak R}$-module homomorphism $\tau_r: B_r\to E_r$ is an isomorphism.
\item\label{main-lemma-b} For each $r$, with $0\le r\le d$, $(E_r)_{\pmb \delta}=(G_r)_{\pmb \delta}$.
\item\label{main-lemma-d} The sequence of homomorphisms $$(\mathbb E,e):\quad   \xymatrix{0\ar[r]& E_d\ar[r]^{e_d}&E_{d-1}\ar[r]^{e_{d-1}}&\cdots\ar[r]^{e_3}&E_2\ar[r]^{e_2}&E_1\ar[r]^{e_1}&E_0}
$$ is a complex  of free  ${\mathfrak R}$-modules. Furthermore, $(\mathbb E,e)$ is a subcomplex of $(\mathbb G,g)$.
\item\label{main-lemma-e} The sequence of homomorphisms $(\mathbb B,b)$ of Definition~{\rm\ref{diff}} is a complex of free  ${\mathfrak R}$-modules.
\item\label{main-lemma-f}  The ${\mathfrak R}$-module homomorphisms $\tau_r: B_r\to E_r$  give 
an isomorphism of complexes $(\mathbb B,b)\simeq (\mathbb E,e):$
$$\xymatrix{0\ar[r]&B_d\ar[r]^{b_d}\ar[d]^{\tau_d}&B_{d-1}\ar[r]^{b_{d-1}}\ar[d]^{\tau_{d-1}}&\cdots\ar[r]^{b_3}&B_2\ar[r]^{b_2}\ar[d]^{\tau_2}&B_1\ar[r]^{b_1}\ar[d]^{\tau_1}&B_0\ar[d]^{b_0}\vspace{5pt}\\
0\ar[r]&E_d\ar[r]^{e_d}&E_{d-1}\ar[r]^{e_{d-1}}&\cdots\ar[r]^{e_3}&E_2\ar[r]^{e_2}&E_1\ar[r]^{e_1}&E_0.}$$ Furthermore, the localizations $(\mathbb E,e)_{\pmb\delta}$ and $(\mathbb G,g)_{\pmb \delta}$ are equal.
In particular, the complexes $\mathbb B_{\pmb\delta}$ and $\mathbb G_{\pmb \delta}$ are isomorphic.

\item\label{main-lemma-g} All of the assertions of Theorem~{\rm\ref{Properties-of-G}} hold for the explicitly constructed complex $(\mathbb B,b)$ in place of $(\mathbb G,g)$. In particular, $\mathbb B_{\pmb \delta}$ is a resolution of ${\mathfrak R}_{\pmb \delta}/I{\mathfrak R}_{\pmb \delta}$ by free ${\mathfrak R}_{\pmb \delta}$-modules for the ${\mathfrak R}$-ideal $I$ of Definition~{\rm\ref{I}}.

\end{enumerate}
\end{corollary}

\begin{proof} Assertions (\ref{main-lemma-a}), (\ref{main-lemma-c}), and  (\ref{main-lemma-b}) are re-statements of the corresponding assertions in Theorem~\ref{thm-1}.

\medskip\noindent(\ref{main-lemma-d}) We must verify that $e_r(E_r)\subseteq E_{r-1}$ and this follows from (\ref{main-lemma-c}) and Proposition~\ref{prop}:
$$e_r(E_r)=g_r(E_r)=g_r(\operatorname{im} \tau_r)=\operatorname{im}(g_r\circ \tau_r)=\operatorname{im}(\tau_{r-1}\circ b_r)=\tau_{r-1}\operatorname{im}(b_r)\subseteq \tau_{r-1}(B_{r-1})=E_{r-1}.$$

\medskip\noindent(\ref{main-lemma-e}) We must verify that $b_r\circ b_{r+1}=0$. One may apply the fact that  $\tau_{r-1}$ is injective, together with  Proposition~\ref{prop}, to the complex $(\mathbb G,g)$, in order  to see that  $$\tau_{r-1}\circ b_r\circ b_{r+1}=g_r\circ g_{r+1}\circ \tau_{r+1}=0;$$ hence, $b_r\circ b_{r+1}=0$.

\medskip\noindent(\ref{main-lemma-f}) We know from (\ref{main-lemma-d}) and (\ref{main-lemma-e})
that $(\mathbb E,e)$ and $(\mathbb B,b)$ are complexes; from Proposition~\ref{prop} that ${\tau:\mathbb B\to\mathbb E}$ is a map of complexes; from (\ref{main-lemma-c}) that $\tau:\mathbb B\to\mathbb E$ is an isomorphism  of complexes; and from (\ref{main-lemma-b}) that $\mathbb E_{\pmb \delta}=\mathbb G_{\pmb \delta}$.

\medskip\noindent(\ref{main-lemma-g}) We see in (\ref{main-lemma-f}) that $(\mathbb B,b)$ to isomorphic to a free sub-complex of $(\mathbb G,g)$ and that $(\mathbb B,b)_{\pmb \delta}$ and $(\mathbb G,g)_{\pmb \delta}$ are isomorphic complexes. 
\end{proof}

\begin{corollary}\label{main}Let $\pmb k$ be a field, $A$  a standard-graded Artinian Gorenstein $\pmb k$-algebra, $S$   the standard-graded polynomial ring $\operatorname{Sym}_{\bullet}^{\pmb k}A_1$, $I$ the kernel of the natural map $\xymatrix{S\ar@{->>}[r]&A}$, $d$  the vector space dimension $\dim_{\pmb k}A_1$, and $n$  the least index with $I_n\neq 0$. 
Assume that $3\le d$ and $2\le n$. Let $(\mathbb B,b)$ be  
the  complex  of Definition~{\rm{\ref{diff}}} and Corollary~{\rm\ref{main-lemma}.\ref{main-lemma-e}}, $\alpha:\pmb k\otimes U\to S_1$ any vector space isomorphism, $\phi\in D_{2n-2}^{\pmb k}S_1$  a Macaulay inverse system for $A$, and $\rho: {\mathfrak R}\to S$ the ring homomorphism of Proposition~{\rm\ref{transition}.\ref{trans-a}} which corresponds to the data $(\alpha,\phi)$. If  
the minimal homogeneous resolution of $A$ by free $S$-modules is Gorenstein-linear, then ${\mathbf B}=S\otimes_{\rho} \mathbb B$ is a minimal homogeneous resolution of $A$ 
by free $S$-modules.
Furthermore, ${\mathbf B}$ is explicitly constructed in a polynomial manner from the coefficients of
$\phi$.
\end{corollary}

\begin{proof} Corollary~\ref{main} is an immediate consequence of  Corollary~\ref{main-lemma}.\ref{main-lemma-g} by way of Theorem~{\rm\ref{Properties-of-G}.\ref{by-way-of}}.\end{proof}

Recall that Proposition~\ref{J18} provides numerous alternatives to the hypothesis ``has a Gorenstein-linear minimal resolution''.

\section{\bf Applications and examples.}\label{Apps}

In Proposition~\ref{6.1} and Remark~\ref{WLEisEZ} we apply our results in order to prove the existence of weak Lefschetz elements. In Proposition~\ref{6.4} we reverse directions and point out that if one knew the existence  of weak Lefschetz elements ahead of time, then one could deduce the form of skeleton of the minimal resolution before actually knowing the entire minimal resolution. In Example~\ref{d=3} we compare the  $\mathbb B$ of \cite{EK-K-2}, where $d=3$, to the  $\mathbb B$ of 
Definition~\ref{diff}  and Corollary~{\rm\ref{main-lemma}.\ref{main-lemma-e}}, where $d$ is arbitrary. Example~\ref{bi-gr} describes $\mathbb B$ as  a bi-homogeneous complex of free ${\mathfrak R}$-modules. The section concludes with a discussion of the natural perfect pairing $\operatorname{pp}: B_r\otimes B_{d-r}\to B_d$, which is induced by the skeleton of $\mathbb B$. This pairing is
  graded-commutative and satisfies the graded product rule.

Let $A$ be a standard graded Artinian algebra over a field $\pmb k$. The linear form $x_1$ in $A_1$ is a {\it weak Lefschetz element} of $A$ if each multiplication map \begin{equation}\label{max-rank}x_1:A_i\to A_{i+1}\end{equation} has maximal rank. In other words, the multiplication map (\ref{max-rank}) is $$\begin{cases}\text{injective}, &\text{whenever $\dim_{\pmb k} A_i\le \dim_{\pmb k} A_{i+1}$, and}\\  
\text{surjective}, &\text{whenever $\dim_{\pmb k} A_{i+1}\le \dim_{\pmb k} A_{i}$}.\end{cases}$$

\begin{proposition}\label{6.1} Let $A$ be a standard graded Artinian  Gorenstein algebra over a field $\pmb k$. Assume that the minimal homogeneous resolution of $A$ by free $\operatorname{Sym}_{\bullet}^{\pmb k}A_1$-modules is Gorenstein-linear. Then every non-zero linear form in $A$ is a weak Lefschetz element of $A$.\end{proposition}
\smallskip \noindent (Again, we remind the reader that Proposition~\ref{J18} provides numerous alternatives to the hypothesis ``has a Gorenstein-linear minimal resolution''.) 

\begin{proof}Let $S$ be the standard graded polynomial ring $\operatorname{Sym}_{\bullet}^{\pmb k}A_1$,  $I$ be the kernel of natural map 
 $\xymatrix{S\ar@{->>}[r]&A}$, $d=\dim_{\pmb k} A_1$, and $n$ be the least degree with $I_n\neq 0$. Let $x_1$ be an arbitrary non-zero element of $A_1$. The multiplication map (\ref{max-rank}) is automatically injective for $i\le n-2$. If (\ref{max-rank}) is surjective at some $i_0$, then (\ref{max-rank}) is surjective for all $i$ with $i_0\le i$. Consequently, $x_1$ is a weak Lefschetz element of $A$ if and only if the multiplication map $x_1:A_{n-1}\to A_n$  is surjective. On the other hand, according to Corollary~\ref{main},  $I$ is generated by the image  of $S\otimes_{{\mathfrak R}} b_1$. A quick examination of Definition~\ref{diff}.\ref{diff-gen} shows that
$(I,x_1)=(\operatorname{Sym}_nU_0,x_1)$ where $S_1=\pmb k x_1\oplus U_0$ for some $\pmb k$-vector space $U_0$. Thus, $S_n=x_1 S_{n-1}+I_n$ and $A_n=x_1A_{n-1}$. \end{proof} 

\begin{remark}\label{WLEisEZ} 
The proof of Proposition~\ref{6.1} does not require knowledge of the entire resolution of $A$. One only needs to know that $\pmb \delta \mu -x_1q(\mu(\widetilde{\Phi}))$ is in $I$ for each $\mu\in \operatorname{Sym}_n(U_0)$. A self-contained proof of this fact is 
$$[\pmb \delta \mu -x_1q(\mu(\widetilde{\Phi}))](\Phi)=\pmb \delta \mu(\Phi) -x_1\left([q(\mu(\widetilde{\Phi}))](\Phi)\right)=\pmb \delta \mu(\Phi) -\pmb \delta x_1(\mu(\widetilde{\Phi}))=\pmb \delta \mu(\Phi) -\pmb \delta \mu(\Phi)=0.$$
\end{remark}

We find the skeleton of $\mathbb B$ (see Remark~\ref{skeleton-2.2}) to be a rather striking feature of this complex. In Proposition~\ref{6.2} we record what this skeleton tells about $\operatorname{Tor}$; in Proposition~\ref{6.3} we observe that one need not know all of $\mathbb B$ in order  to know the conclusions of Proposition~\ref{6.2} --- one ``only'' needs to know that $x_1$ is a weak Lefschetz element; see Proposition~\ref{6.1} and especially, Remark~\ref{WLEisEZ}.
In Proposition~\ref{6.4} we take this reverse engineering one step further; we prove that if $x_1$ is a weak Lefschetz element, then $S\otimes_{{\mathfrak R}}\mathbb B$ must contain a copy of the skeleton. 

\begin{data}\label{data6}Let $A$ be a standard graded Artinian  Gorenstein algebra over a field $\pmb k$,   $S$ be the standard graded polynomial ring $\operatorname{Sym}_{\bullet}^{\pmb k}A_1$,  $I$ be the kernel of natural map 
 $\xymatrix{S\ar@{->>}[r]&A}$, $d=\dim_{\pmb k} A_1$, and $n$ be the least degree with $I_n\neq 0$.  Assume that the minimal homogeneous resolution of $A$ by free $S$-modules is Gorenstein-linear. Let $x_1$ be an arbitrary non-zero element of $A_1$ and $\overline{S}=S/(x_1)$. View $S$ as $\pmb k[x_1,x_2,\dots,x_d]$ and 
$\overline{S}=\pmb k[x_2,\dots,x_d]$ for some basis $x_1,x_2,\dots,x_d$ of $A_1$ (which includes the previously identified element $x_1$).
\end{data}

\begin{proposition}\label{6.2} Adopt Data~{\rm\ref{data6}}. Then
$$\operatorname{Tor}_i^S(A,\overline{S})=\begin{cases} \overline{S}/(x_2,\dots,x_d)^n,&\text{if $i=0$,}\\
\operatorname{Ext}^{d-1}_{\overline S}(\overline{S}/(x_2,\dots,x_d)^n,\overline{S}),&\text{if $i=1$, and}\\
0,&\text{otherwise.}\end{cases}$$
\end{proposition}
\begin{proof} According to Corollary~\ref{main}, the minimal homogeneous resolution of $A$ by free $S$-modules is  $S\otimes_{{\mathfrak R}}\mathbb B$; and so $\operatorname{Tor}_i^S(A,\overline{S})=\operatorname{H}_i(\overline{S}\otimes_{{\mathfrak R}}\mathbb B)$. Now the conclusion is obvious from (\ref{skeleton}).
Indeed, the resolution of $\overline{S}/(x_2,\dots,x_d)^n$ is the bottom row of $\overline{S}\otimes_{{\mathfrak R}}\mathbb B$ (when 
$\overline{S}\otimes_{{\mathfrak R}}\mathbb B$ is written in the form of (\ref{skeleton})) and the resolution of $\operatorname{Ext}^{d-1}_{\overline S}(\overline{S}/(x_2,\dots,x_d)^n,\overline{S})$ is the top row of $\overline{S}\otimes_{{\mathfrak R}}\mathbb B$ (when 
$\overline{S}\otimes_{{\mathfrak R}}\mathbb B$ is written in the form of (\ref{skeleton})).
\end{proof}
\begin{proposition}\label{6.3}Adopt Data~{\rm\ref{data6}}. Assume that $x_1:A_{n-1}\to A_{n}$ is surjective. Then the conclusion of Proposition~\ref{6.2} can be established without any knowledge of the minimal resolution of $A$ by free $S$-modules. 
\end{proposition}
\begin{proof} 
The hypothesis $x_1A_{n-1}=A_n$ guarantees that \begin{equation}\label{WLP}x_1S_{n-1}+I_n=x_1S_{n-1}+[(x_2,\dots,x_d)^n]_n.\end{equation} It follows that $A/(x_1)=\overline{S}/(x_2,\dots,x_d)^n$. The value of $\operatorname{Tor}_0^S(A,\overline{S})$ is now clear. Furthermore, we calculate
$$\textstyle\operatorname{Tor}_1^S(A,\overline{S})=\frac{I:_{S}x_1}{I}=\operatorname{ann}_Ax_1=\operatorname{Hom}_A(\frac{A}{(x_1)},A)
\simeq \operatorname{Ext}_{\overline{S}}^{d-1}(
\frac{A}{(x_1)},\overline{S}),$$ 
where the final isomorphism holds because the rings $A$ and $\overline{S}$ are both Gorenstein and the canonical module of $A/(x_1)$ may be computed using
either $\xymatrix{A\ar@{->>}[r]&A/(x_1)}$ or $\xymatrix{\overline{S}\ar@{->>}[r]&A/(x_1)}$.
\end{proof} 
\begin{proposition}\label{6.4} Adopt Data~{\rm\ref{data6}}. Assume that $x_1:A_{n-1}\to A_{n}$ is surjective. Let $\mathbb F$ be a minimal homogeneous resolution of $A$ by free $S$-modules and let $\mathbb S$ be the complex $\overline{S}\otimes_{{\mathfrak R}} \mathbb B$ written in the form of {\rm(\ref{skeleton})}. Do not assume any knowledge of $\mathbb B$ or Corollary~{\rm\ref{main}}. Then $\overline{S}\otimes_S\mathbb F\simeq \mathbb S$.
\end{proposition}

\smallskip\noindent(The point of Proposition \ref{6.4} is that the weak Lefschetz hypothesis establishes that $\mathbb F$ has the skeleton of $\mathbb B$ without first proving that $\mathbb F\simeq S\otimes_{{\mathfrak R}}\mathbb B$.)
\begin{proof} Let $(\overline{\mathbb F},\overline{f})$ represent $\overline{S}\otimes_S \mathbb F$. The resolution $\mathbb F$ is pure; hence the Betti numbers of $\mathbb F$ are known a priori (for example, by Herzog and K\"uhl \cite{HK}); furthermore, these Betti numbers satisfy
\begin{equation}\label{no-inspire}\operatorname{rank} F_r=\operatorname{rank} K_{r-1,n-1}U_0\oplus \operatorname{rank} L_{r-1,n}U_0,\end{equation} for $1\le r\le d-1$, a priori; see Remark~\ref{meantime}. (The arithmetic of Remark~\ref{meantime} was inspired by knowledge that $\mathbb B$ is interesting; nonetheless, the arithmetic continues to be legitimate even if it is done without inspiration. Curiously, in the present proof, (\ref{no-inspire}) is of crucial importance when $r=1$; for larger values of $r$ the present proof, in fact, reproves (\ref{no-inspire}).)  Let $$k_r =\operatorname{rank} K_{r-1,n-1}U_0\quad \text{and}\quad \ell_r=\operatorname{rank} L_{r-1,n}U_0,\quad \text{for $1\le r\le d$}.$$ Combine 
(\ref{WLP}) and (\ref{no-inspire}) to see that $\overline{f}_1:\overline{F}_1\to \overline {S}$ can be put in the form 
$$\overline{S}^{k_1}\oplus \overline{S}^{\ell_1}
\xrightarrow{\ \bmatrix 0&\dots&0&|&x_2^n&\dots&x_d^n\endbmatrix\ } \overline{S},$$
where all of the $\ell_1$ monomials in the set  $\binom{x_2,\dots,x_d}{n}$ appear in the list $x_2^n,\dots, x_d^n$. It is clear that the generators of $\overline{S}^{k_1}$ represent part of a minimal generating set of $$\operatorname{H}_1(\overline{\mathbb F})=\operatorname{Tor}_1^S(A,\overline{S})\simeq \operatorname{Ext}^{d-1}_{\overline S}(\overline{S}/(x_2,\dots,x_d)^n,\overline{S}).$$(The final isomorphism is due to Proposition~\ref{6.3}.) On the other hand, $\operatorname{Ext}^{d-1}_{\overline S}(\overline{S}/(x_2,\dots,x_d)^n,\overline{S})$ is minimally generated by $k_1$ elements. So, $\overline{S}^{k_1}$ maps onto $\operatorname{H}_1(\overline{\mathbb F})$. It follows that the augmented complex 
\begin{equation}\label{Aug}
0\to \overline{F}_d  \to\overline{F}_{d-1} \to\cdots \to\overline{F}_1\to \begin{array}{c}\operatorname{H}_1(\overline{\mathbb F})\\\oplus\\ \overline{F}_0\end{array}\to \operatorname{H}_0(\overline{\mathbb F})\to 0\end{equation} is exact. The beginning of (\ref{Aug}) is the total complex of
$$\xymatrix{ &\overline{S}^{k_1}\ar[r]\ar[d]^{0}&\operatorname{H}_1(\overline{\mathbb F})
\ar[r]\ar[d]^{0}&0
\\
\overline{S}^{\ell_1}\ar[r]&\overline{S}\ar[r]&\operatorname{H}_0(\overline{\mathbb F})\ar[r]&0}$$
Let $z_1,\dots,z_{\ell_2}$ be a set of homogeneous elements in $\overline{S}^{\ell_1}$ which represent a  minimal generating set for $\ker(\overline{S}^{\ell_1}\to\overline{S})$ and let $z_1',\dots,z_{k_2}'$ be a set of homogeneous elements in $\overline{S}^{k_1}$ which represent a  minimal generating set for $\ker(\overline{S}^{k_1}\to\operatorname{H}_1(\overline{\mathbb F}))$. It is clear that
$z_1,\dots,z_{\ell_2},z_1',\dots,z_{k_2}'$ represents a minimal generating set 
for $$\ker\left(\overline{F}_1\to 
(\operatorname{H}_1(\overline{\mathbb F})
\oplus 
\overline{F}_0)
\right).$$ The complex (\ref{Aug}) has homology zero; so it is possible to decompose $\overline{F}_2$ as $\overline{S}^{\ell_2}\oplus \overline{S}^{k_2}$so that the generators of 
$\overline{S}^{\ell_2}$ map to $z_1,\dots,z_{\ell_2}$ and the generators of 
$\overline{S}^{k_2}$ map to $z_1',\dots,z'_{k_2}$. At this point, the beginning of (\ref{Aug}) is the total complex of
$$\xymatrix{ &\overline{S}^{k_2}\ar[r]\ar[d]^{0}&\overline{S}^{k_1}\ar[r]\ar[d]^{0}&\operatorname{H}_1(\overline{\mathbb F})
\ar[r]\ar[d]^{0}&0
\\
\overline{S}^{\ell_2}\ar[r]&\overline{S}^{\ell_1}\ar[r]&\overline{S}\ar[r]&\operatorname{H}_0(\overline{\mathbb F})\ar[r]&0.}$$Continue in this manner to finish the proof.
\end{proof}

\begin{example}\label{d=3} The complex $(\mathbb B,b)$ of \cite[Def.~2.7 and Obs.~4.4]{EK-K-2} (which we  now write as $(\underline{\mathbb B},\underline{b})$)  is isomorphic to the complex $(\mathbb B,b)$ of Definition~\ref{diff}  and Corollary~{\rm\ref{main-lemma}.\ref{main-lemma-e}}, when $d=3$. Indeed, the isomorphism is given by $\gamma: (\underline{\mathbb B},\underline{b})\to (\mathbb B,b)$
$$\xymatrix{0\ar[r]&\underline{B_3}\ar[r]^{\underline{b_3}}\ar[d]^{\gamma_3}&
\underline{B_2}\ar[r]^{\underline{b_2}}\ar[d]^{\gamma_2}&
\underline{B_1}\ar[r]^{\underline{b_1}}\ar[d]^{\gamma_1}&
\underline{B_0}\ar[d]^{\gamma_0}\\
0\ar[r]& B_3 \ar[r]^{ b_3 } &
 B_2 \ar[r]^{ b_2 } &
 B_1 \ar[r]^{ b_1 } &
 B_0,}$$ where  the maps 
$$\gamma_0:\underline{B_0}={\mathfrak R}\xrightarrow{
} B_0={\mathfrak R}\quad\text{and}\quad \gamma_1: \underline{B_1}=\left\{\begin{array}{c} {\mathfrak R}\otimes D_{n-1}(U_0^*)\\\oplus\\ {\mathfrak R}\otimes \operatorname{Sym}_n U_0\end{array}\right\}\xrightarrow{
} B_1=\left\{\begin{array}{c} {\mathfrak R}\otimes K_{0,n-1}U_0\\\oplus\\ {\mathfrak R}\otimes L_{0,n} U_0\end{array}\right\}$$ are the identity maps;  the map 
$$\gamma_2: \underline{B_2}=\left\{\begin{array}{c}
{\mathfrak R}\otimes \operatorname{Sym}_{n-1}U_0\\
\oplus\\  {\mathfrak R}\otimes D_n(U_0^*)\end{array}\right\}\xrightarrow{
} B_2=\left\{\begin{array}{c} {\mathfrak R}\otimes K_{1,n-1}U_0\\\oplus\\ {\mathfrak R}\otimes L_{2,n} U_0\end{array}\right\}$$
 is $$\gamma_2\left( \bmatrix \mu\\\nu\endbmatrix\right)=\bmatrix \eta(x_2\wedge x_3\otimes \nu) \\ -\kappa(x_2\wedge x_3\otimes \mu)\endbmatrix,$$for $\mu \in \operatorname{Sym}_{n-1}U_0$ and $\nu\in D_n(U_0^*)$;
and the map
$$\gamma_3:\underline{B_3}={\mathfrak R}\xrightarrow{
} B_3={\mathfrak R} \otimes {\textstyle\bigwedge^2}U_0$$ is 
$\gamma_3(1)= x_2\wedge x_3$.  It follows that the examples of \cite[Sect.~6]{EK-K-2} are also examples of the complex $(\mathbb B, b)$ of Definition~\ref{diff} and Corollary~{\rm\ref{main-lemma}.\ref{main-lemma-e}}. A few further comments about the complex $\mathbb B$ when $d=3$ are contained in Example~\ref{d=3matrices}. 
\end{example}

\begin{example}\label{bi-gr} Recall from Data~\ref{Opening-Data} that ${\mathfrak R}$ is a bi-graded ring. We describe the complex $\mathbb B$ of Definition~\ref{diff} and Corollary~\ref{main-lemma}.\ref{main-lemma-e} as a bi-homogeneous complex of free ${\mathfrak R}$-modules. Recall the constant $\mathrm{top}=\binom{n+d-2}{d-1}=\deg \pmb\delta$
from Definition~\ref{manuf-data}. Observe that if $\nu\in D_{n-1}(U^*)$, then $q(\nu)$ is in ${\mathfrak R}_{(0,\mathrm{top}-1)}\otimes \operatorname{Sym}_{n-1}U$; and if $\mu\in \operatorname{Sym}_rU$, then $\mu(\widetilde{\Phi})\in {\mathfrak R}_{(0,1)}\otimes D_{2n-1-r}(U^*)$. It follows that the entries of a matrix representing $b_r$ are homogeneous and have degree
$$\begin{cases} 
\bmatrix (n,\mathrm{top}-1)&(n,\mathrm{top})\endbmatrix, &\text{if $r=1$,}\vspace{5pt}\\
\bmatrix (1,\mathrm{top})&(1,\mathrm{top}+1)\\(1,\mathrm{top}-1)&(1,\mathrm{top})\endbmatrix,&\text{if $2\le r\le d-1$, and}\vspace{5pt}\\
\bmatrix (n,\mathrm{top})\\(n,\mathrm{top}-1)\endbmatrix,&\text{if $r=d$.}\end{cases}$$
Therefore, as bi-graded free ${\mathfrak R}$-modules,
$$
B_r\simeq \begin{cases} \phantom{x; {\mathfrak R}(-2n-d}{\mathfrak R},&\text{if $r=0$,}\vspace{7pt}\\
\left.\begin{array}{c}{\mathfrak R}(-n-r+1,-r(\mathrm{top})+1)^{k_r}\\\oplus\\
{\mathfrak R}(-n-r+1,-r(\mathrm{top}))^{\ell_r}\end{array}\right\},&\text{if $1\le r\le d-1$, and}\vspace{7pt}\\
\phantom{x;} {\mathfrak R}(-2n-d+2,-d(\mathrm{top})+1),&\text{if $r=d$,}\end{cases}$$ where
$$\begin{array}{rcccl}
k_r&=&\operatorname{rank} K_{r-1,n-1}U_0&=& \binom{d+n-2}{r-1}\binom{d+n-r-2}{n-1}\text{ and}\vspace{5pt}\\
\ell_r&=&\operatorname{rank} L_{r-1,n}U_0&=& \binom{d+n-2}{r-1+n}\binom{r+n-2}{r-1};\end{array}$$
see for example \ref{meantime} or \cite[(2.3) and (5.8)]{EKK}.
\end{example}

\bigskip
We next exhibit the self-duality of the complex $(\mathbb B,b)$.
\begin{definition}\label{ppr-def} Adopt Data~\ref{data1} and recall the complex $(\mathbb B,b)$ of Definition~\ref{diff} and Corollary~\ref{main-lemma}.\ref{main-lemma-e}. For each integer $r$ with $0\le r\le d$, define the ${\mathfrak R}$-module homomorphism $$\operatorname{pp}_r:B_r\otimes_{{\mathfrak R}} B_{d-r}\to B_d$$ as follows.
\begin{enumerate}[\rm(a)]
\item If $r=0$, then $\operatorname{pp}_0(\alpha \otimes \theta)=\alpha \theta$ for $\alpha \in B_0={\mathfrak R}$ and $\theta\in B_d={\mathfrak R}\otimes \bigwedge ^{d-1}U_0$.
\item If $1\le r\le d-1$, then
\begin{align}&\operatorname{pp}_r \left([\eta(\theta_r\otimes \nu_n)+\kappa(\theta_r'\otimes \mu_{n-1})]
 \otimes
 [\eta(\theta_{d-r}\otimes \nu_n') +\kappa(\theta_{d-r}'\otimes \mu'_{n-1})]\right)\notag\\
{}={}&\label{ppr}[\mu_{n-1}(\nu_n')](\theta_r')\wedge \theta_{d-r}- [\mu_{n-1}'(\nu_n)](\theta_r)\wedge \theta_{d-r}',
\end{align}
 
\noindent for 
$$\begin{array}{ll}
\eta(\theta_r\otimes \nu_n)\in {\mathfrak R}\otimes K_{r-1,n-1}U_0\subseteq B_r,&
\kappa(\theta_r'\otimes \mu_{n-1})\in {\mathfrak R}\otimes L_{r-1,n}U_0\subseteq B_r,\vspace{5pt}\\
\eta(\theta_{d-r}\otimes \nu_n')\in {\mathfrak R}\otimes K_{d-r-1,n-1}U_0\subseteq B_{d-r},\text{ and}&
\kappa(\theta_{d-r}'\otimes \mu'_{n-1}) \in {\mathfrak R}\otimes L_{d-r-1,n}U_0\subseteq B_{d-r}.\end{array}
$$

\noindent The right side of (\ref{ppr}) is in $B_d= {\mathfrak R}\otimes {\textstyle\bigwedge^{d-1}}U_0$. In particular, $\theta_r$ and $\theta_{r}'$ are in $\bigwedge^rU_0$, $\theta_{d-r}$ and $\theta_{d-r}'$ are in $\bigwedge^{d-r}U_0$, 
$\nu_n$ and $\nu_n'$ are in $D_n(U_0^*)$, and $\mu_{n-1}$ and $\mu_{n-1}'$ are in $\operatorname{Sym}_{n-1}U_0$.
\item If $r=d$, then $\operatorname{pp}_d(\theta \otimes \alpha)=\theta\alpha $ for $\theta\in B_d={\mathfrak R}\otimes \bigwedge ^{d-1}U_0$   and $\alpha \in B_0={\mathfrak R}$.
\end{enumerate}
(In the future we will often write $\operatorname{pp}$ in place of $\operatorname{pp}_r$.)
\end{definition}

\begin{proposition}\label{ppr-prop}Retain the notation of Definition~{\rm\ref{ppr-def}}. The following statements hold.
\begin{enumerate}[\rm(a)]
\item\label{ppr-prop-a} Each homomorphism $\operatorname{pp}_r$, with $0\le r\le d$, is a perfect pairing.
\item\label{ppr-prop-b} Each homomorphism $\operatorname{pp}_r$, with $0\le r\le d$, satisfies $$\operatorname{pp}_r(\Theta_r\otimes \Theta_{d-r})=(-1)^{r(d-r)} \operatorname{pp}_{d-r}(\Theta_{d-r}\otimes\Theta_r),$$ for $\Theta_r\in B_r$ and $\Theta_{d-r}\in B_{d-r}$.
\item\label{ppr-prop-c} If $0\le r\le d-1$, then $$\operatorname{pp}_r(b_{r+1}(\Theta_{r+1})\otimes \Theta_{d-r})+(-1)^{r+1}\operatorname{pp}_{r+1}(\Theta_{r+1}\otimes b_{d-r}(\Theta_{d-r}))=0,$$ for $\Theta_{r+1}\in B_{r+1}$ and $\Theta_{d-r}\in B_{d-r}$.
\end{enumerate}
\end{proposition}

\begin{Remark} We guess that the complex $(\mathbb B,b)$ is an associative DG-algebra and that the perfect pairing $\operatorname{pp}_r$ of Definition~\ref{ppr-def} describes the multiplication $B_r\otimes B_{d-r}\to B_d$. (Please notice that a ``guess'' is even weaker than a ``conjecture''.) If our guess is correct, then (\ref{ppr-prop-b}) would show that the multiplication $B_r\otimes B_{d-r}\to B_d$ is graded-commutative and (\ref{ppr-prop-c}) would show that the graded product rule holds for $B_{r+1}\otimes B_{d-r}\to B_{d+1}=0$. \end{Remark}

\begin{proof} To prove (\ref{ppr-prop-a}) it suffices to see that for each $r$, with $1\le r\le d-1$, the $\mathbb Z$-module homomorphism \begin{equation}\label{sts}L_{r-1,n}U_0\otimes K_{d-r-1,n-1}U_0\longrightarrow {\mathbb Z},\end{equation} which is given by
$$\kappa(\theta_r\otimes \mu_{n-1})\otimes \eta(\theta_{d-r}\otimes \nu_n)\mapsto [\mu_{n-1}(\nu_n)](\theta_r)\wedge \theta_{d-r},$$ is a perfect pairing. This assertion is clear and well-known. 
Indeed, the standard bases for $L_{r-1,n}U_0$ and $K_{d-r-1,n-1}U_0$ (see, for example, \cite[(5.4) and (5.5)]{EKK} or Remark~\ref{rmk2} or Definition~\ref{basis-47})
act like dual bases (up to sign) under this homomorphism. Indeed, the basis element  dual (up to sign) 
 to 
$$\kappa(x_{a_1}\wedge \ldots\wedge x_{a_r} \otimes x_{b_1}\cdots x_{b_{n-1}})$$ with $a_1<a_2<\dots <a_r$ and $a_1\le b_1\le b_2\le \dots \le b_{n-1}$ is
\begin{equation}\label{promise}\eta (x_{c_1}\wedge \ldots\wedge x_{c_{d-r}}\otimes (x_{a_1}\cdot x_{b_1}\cdots x_{b_{n-1}})^*),\end{equation} where 
$c_1<\dots<c_{d-r}$ is the complement of $\{a_2,\dots, a_r\}$ in $\{2,\dots,d\}$.
We give a complete proof in Observation~\ref{behavior}.

We prove part of (\ref{ppr-prop-b}). Take $1\le r\le d-1$ and let $\Theta_r=\kappa(\theta_r'\otimes \mu_{n-1})$ and $\Theta_{d-r}=\eta(\theta_{d-r}\otimes \nu_n')$. Apply Definition~\ref{ppr-def}, the graded-commutativity of the exterior algebra $\bigwedge^{\bullet}U_0$, and the fact that action of each element of $U_0^*$ on $\bigwedge^{\bullet}U_0$ satisfies the graded product rule (and $\bigwedge^{d}U_0=0$) to  see that 
$$\begin{array}{lll}\operatorname{pp}_{d-r}(\Theta_{d-r}\otimes \Theta_r)&=&-[\mu_{n-1}(\nu_n')](\theta_{d-r})\wedge \theta_r'\vspace{5pt}\\
&=&-(-1)^{r(d-r-1)}\theta_r'\wedge  [\mu_{n-1}(\nu_n')](\theta_{d-r})\vspace{5pt}\\
&=&-(-1)^{r(d-r-1)}(-1)^{r+1}[\mu_{n-1}(\nu_n')](\theta_r')\wedge  \theta_{d-r}\vspace{5pt}\\
&=&(-1)^{r(d-r)}[\mu_{n-1}(\nu_n')](\theta_r')\wedge  \theta_{d-r}\vspace{5pt}\\
&=& (-1)^{r(d-r)}\operatorname{pp}_r(\Theta_r\otimes\Theta_{d-r}).\end{array}$$ The rest of the proof of (\ref{ppr-prop-b}) proceeds in a similar manner.

We prove part of (\ref{ppr-prop-c}). Fix $r$ with $1\le r\le d-2$ and let 
$$\Theta_{r+1}=\eta(\theta_{r+1}\otimes \nu_n)\in B_{r+1}\quad \text{and}\quad
\Theta_{d-r}=\eta(\theta_{d-r}\otimes \nu_n')\in B_{d-r},$$with $\theta_{i}\in \bigwedge^iU_0$ and $\nu_n$ and $\nu'_n$ in $D_n(U_0^*)$.
We compute
\begingroup
\allowdisplaybreaks
\begin{eqnarray}
&&\operatorname{pp}_r(b_{r+1}(\Theta_{r+1})\otimes \Theta_{d-r})+(-1)^{r+1}\operatorname{pp}_{r+1}(\Theta_{r+1}\otimes b_{d-r}(\Theta_{d-r}))\notag\\
&=&\begin{cases} -x_1\otimes \sum\limits_{j=2}^d \operatorname{pp}_r\left(\kappa(x_j^*(\theta_{r+1})\otimes(\operatorname{proj}\circ q)([x_j(\nu_n)]))\otimes \eta(\theta_{d-r}\otimes \nu_n')\right)\vspace{5pt}\\
-(-1)^{r+1}x_1\otimes \sum\limits_{j=2}^d \operatorname{pp}_{r+1}\left( \eta(\theta_{r+1}\otimes \nu_n)\otimes 
\kappa(x_j^*(\theta_{d-r})\otimes (\operatorname{proj}\circ q)(x_j(\nu_n')))\right)\end{cases}\notag\\
&=&\begin{cases}
-x_1\otimes\sum\limits_{j=2}^d 
\left(\vphantom{\widetilde{\Phi}}\left[
(\operatorname{proj}\circ q)[x_j(\nu_n)]\right](\nu_n')\right)[x_j^*(\theta_{r+1})]\wedge \theta_{d-r} \\
-(-1)^{r+1}(-1) x_1\otimes\sum\limits_{j=2}^d 
\left[ \vphantom{\widetilde{\Phi}}
\left((\operatorname{proj}\circ q)[x_j(\nu_n')]\right)
 (\nu_n)\right](\theta_{r+1}) \wedge x_j^*(\theta_{d-r}).
\end{cases}\label{??}
\end{eqnarray}
\endgroup
We re-configure the element $\left[(\operatorname{proj}\circ q)[x_j(\nu_n)]\right](\nu_n')$ of $U_0^*$. First of all, 
the element $\nu_n'$ is in $D_n(U_0^*)$; so $[q(\nu_{n-1})](\nu'_n)=[(\operatorname{proj}\circ q)(\nu_{n-1})](\nu'_n)$ for any $\nu_{n-1}\in D_{n-1}(U^*)$; thus,  
$$\left[(\operatorname{proj}\circ q)[x_j(\nu_n)]\right](\nu_n')=\left[q[x_j(\nu_n)]\right](\nu_n').$$ Now we write this element
 in terms of the basis $\{x_2^*,\dots,x_d^*\}$ of $U_0^*$ and obtain
$$\left[(\operatorname{proj}\circ q)[x_j(\nu_n)]\right](\nu_n')=\sum\limits_{k=2}^d x_k\left(\vphantom{\widetilde{\Phi}}\left[q[x_j(\nu_n)]\right](\nu_n')\right) \cdot x_k^*.$$ 
In a similar manner, 
\begin{equation}\label{in-a-sim}\left[(\operatorname{proj}\circ q)[x_j(\nu_n')]\right](\nu_n)=\sum\limits_{k=2}^d x_k\left(\vphantom{\widetilde{\Phi}}\left[q[x_j(\nu_n')]\right](\nu_n)\right) \cdot x_k^*.\end{equation}
We  take advantage of the module action of $\operatorname{Sym}_\bullet U$ on $D_{\bullet}(U^*)$, the fact that $\operatorname{Sym}_\bullet U$ is a commutative ring, and Remark~\ref{rmk3}, to see that
\begin{equation}\label{to-see-that}\begin{array}{lll}\left[(\operatorname{proj}\circ q)[x_j(\nu_n)]\right](\nu_n')&=&\sum\limits_{k=2}^d x_k\left(\vphantom{\widetilde{\Phi}}\left[q[x_j(\nu_n)]\right](\nu_n')\right) \cdot x_k^*\\
&=&\sum\limits_{k=2}^d \left(\vphantom{\widetilde{\Phi}}\left[q[x_j(\nu_n)]\right](x_k(\nu_n'))\right) \cdot x_k^*\\
&=&\sum\limits_{k=2}^d \left(\vphantom{\widetilde{\Phi}}\left[q[x_k(\nu_n')]\right](x_j(\nu_n))\right) \cdot x_k^*\\
&=&\sum\limits_{k=2}^d x_j\left(\vphantom{\widetilde{\Phi}}\left[q[x_k(\nu_n')]\right](\nu_n)\right) \cdot x_k^*
.\end{array}\end{equation}
Use (\ref{to-see-that}) and (\ref{in-a-sim}) to see that (\ref{??}) is equal to
\begin{equation}\label{penult}\begin{cases}
-x_1\otimes\sum\limits_{j,k=2}^d 
x_j\left(\vphantom{\widetilde{\Phi}}\left[q[x_k(\nu_n')]\right](\nu_n)\right) \cdot [(x_k^*\wedge x_j^*)(\theta_{r+1})]\wedge \theta_{d-r} \\
+(-1)^{r+1} x_1\otimes\sum\limits_{j,k=2}^d 
x_k\left(\vphantom{\widetilde{\Phi}}\left[q[x_j(\nu_n')]\right](\nu_n)\right) \cdot x_k^*(\theta_{r+1}) \wedge x_j^*(\theta_{d-r}).
\end{cases}\end{equation}The differential $x_j^*$ exhibits the graded product rule on $\bigwedge^{\bullet}U_0$ and $\bigwedge^{d}U_0=0$; so,
$$0=x_j^*(x_k^*(\theta_{r+1}) \wedge \theta_{d-r})=
[x_j^*\wedge x_k^*](\theta_{r+1}) \wedge \theta_{d-r}+(-1)^r
x_k^*(\theta_{r+1}) \wedge x_j^*(\theta_{d-r});$$ so, (\ref{penult}) is 
$$\begin{cases}
-x_1\otimes\sum\limits_{j,k=2}^d 
x_j\left(\vphantom{\widetilde{\Phi}}\left[q[x_k(\nu_n')]\right]((\nu_n))\right) \cdot [(x_k^*\wedge x_j^*)(\theta_{r+1})]\wedge \theta_{d-r} \\
+ x_1\otimes\sum\limits_{j,k=2}^d 
x_k\left(\vphantom{\widetilde{\Phi}}\left[q[x_j(\nu_n')]\right](\nu_n)\right) \cdot [x_j^*\wedge x_k^*](\theta_{r+1}) \wedge \theta_{d-r};
\end{cases}$$ and this is zero. The rest of the proof of (\ref{ppr-prop-c}) proceeds in a similar manner.
\end{proof}

\section{\bf A matrix  description of $\mathbb B$.} \label{expl-desc}

In this section we make the complex $\mathbb B$ significantly more explicit. We describe $\mathbb B$ in terms of  elements of the ring ${\mathfrak R}$, rather than in terms of the maps $q$, $\operatorname{proj}\circ q$, and $\widetilde{\Phi}$. Proposition~\ref{Best} contains a version of $b_r$ which is close to an explicit matrix for all $r$ and $d$. In Proposition~\ref{Best}, the value of $b_r$ applied to a standard basis element of
$B_r$ is given explicitly as a linear combination of elementary generators of $B_{r-1}$ with coefficients in ${\mathfrak R}$. (Each elementary generator of $B_{r-1}$
can be expressed in terms of the standard basis elements of $B_{r-1}$;  that step is carried out in Theorem~\ref{47.4}, but is not carried out in Proposition~\ref{Best}.) 
Theorem~\ref{47.4} gives an explicit matrix version of $b_r$ for all $r$ and $d$. The calculation that gets from $\mathbb B$, as described in Definition~\ref{diff}, to the $\mathbb B$ of  Proposition~\ref{Best} follows the corresponding calculation in \cite{EK-K-2} fairly closely and we omit most details. The proof of Theorem~\ref{47.4} requires a careful analysis of the standard bases for various Schur and Weyl modules, and the duality among theses bases. This part of the proof is recorded in complete detail. 
 In Examples \ref{d=3matrices}, \ref{d=4matrices},  and \ref{d=4matrices-ext}, we emphasize the self-duality of matrices in the middle of the resolution $\mathbb B$ when $d$ is $3$ or $4$. 

The following  notation and elementary results are very similar to those given in   \cite[Sect.~5]{EK-K-2}; see in particular \cite[(5.1.1) and Obs.~5.4]{EK-K-2}.
\begin{data}\label{data-exp} Recall the ring ${\mathfrak R}={\operatorname{Sym}}_{\bullet}(U\oplus {\operatorname{Sym}}_{2n-2}U)$ of Data~\ref{Opening-Data} and Data~\ref{data1}. Fix a basis $x_2,\dots,x_d$ for $U_0$, and, for each $m\in \binom{x_1,\dots,x_d}{2n-2}$, let $t_m$ be the element $\Phi(m^*)$ in ${\mathfrak R}$. In this language, 
$$\begin{array}{l}
\text{${\mathfrak R}$ is the bi-graded polynomial ring $\textstyle \mathbb Z[x_1,\dots,x_d,\{t_m\mid m\in\binom{x_1,\dots,x_d}{2n-2}\}]$},\vspace{5pt}\\
\Phi=\sum\limits_{m\in \binom{x_1,\dots,x_d}{2n-2}}t_m\otimes m^*\in {\mathfrak R}\otimes D_{2n-2}(U^*),\text{ and}\vspace{5pt}\\
\widetilde{\Phi}=\sum\limits_{m\in \binom{x_1,\dots,x_d}{2n-2}}t_m\otimes (x_1m)^*\in {\mathfrak R}\otimes D_{2n-1}(U^*).\end{array}$$
Let $T$ be the matrix $(t_{m_1m_2})$ where $m_1$ and $m_2$ roam over  $\binom{x_1,\dots,x_d}{n-1}$ in the same order, $\pmb \delta$ be the determinant of $T$, and $Q$ be the classical adjoint of $T$. We refer to the entries of $Q$ as $Q_{m_1,m_2}$ for $m_1$ and $m_2$ in $\binom{x_1,\dots,x_d}{n-1}$.
The matrices $T$ and $Q$ are both symmetric,   
\begingroup
\allowdisplaybreaks
\begin{align}\label{TQ2}\textstyle &\sum\limits_{m\in\binom{x_1,\dots,x_d}{n-1}} t_{m'm}Q_{m,m''}=\chi(m'=m'')\pmb \delta, &&\text{for all $m'$ and $m''$ in $\binom{x_1,\dots,x_d}{n-1}$,}\\
\label{q-} &q(m_2^*)= \sum\limits_{m_1\in \binom{x_1,\dots,x_d}{n-1}} Q_{m_1,m_2}\otimes m_1\in {\mathfrak R}\otimes \operatorname{Sym}_{n-1}U&&\text{for $m_2\in \binom{x_1,\dots,x_d}{n-1}$},\\
\label{m-of-superPhi-}&m_1(\widetilde{\Phi})=\sum\limits_{m_2\in \binom{x_1,\dots,x_d}{2n-2-r}} t_{m_1m_2}\otimes (x_1m_2)^*\in {\mathfrak R}\otimes D_{n-1}(U^*)&&\text{for $m_1\in \binom{x_2,\dots,x_d}{r}$},\text{ and}\\
\label{m-of-Phi-}&m_1(\Phi)=\sum\limits_{m_2\in \binom{x_1,\dots,x_d}{2n-2-r}} t_{m_1m_2}\otimes m_2^*\in {\mathfrak R}\otimes D_{n-1}(U^*)&&\text{for $m_1\in \binom{x_1,\dots,x_d}{r}$},\\
&&&\text{for $0\le r\le 2n-2$.}\notag
\end{align}
\endgroup 
\end{data}

\begin{proposition}\label{Best} If $(\mathbb B,b)$ is the complex of Definition~{\rm\ref{diff}} and Corollary~{\rm\ref{main-lemma}.\ref{main-lemma-e}}, then, in the language of Data~{\rm\ref{data-exp}}, the differentials of $\mathbb B$ are explicitly given by the following formulas. 

\begin{enumerate}[\rm(a)] 
\item\label{Best-a} The map $b_1$ is explicitly given as follows.
\begin{enumerate}[\rm(i)] 
\item If $m\in \binom{x_2, \ldots, x_d}{n-1}$, then $m^*\in K_{0,n-1}U$ and $$b_1(m^*)=\sum\limits_{m_1\in \binom{x_1, \ldots, x_d}{n-1}} x_1m_1Q_{m_1,m}.$$

\item If $m\in \binom{x_2, \ldots, x_d}{n}$, then $m\in L_{0,n}U$ and $$b_1(m)=\pmb\delta m -\sum\limits_{m_1\in \binom{x_1, \ldots, x_d}{n-2}}\sum\limits_{m_2\in \binom{x_1, \ldots, x_d}{n-1}} x_1m_2Q_{m_2,x_1m_1}
t_{m_1m}.$$
\end{enumerate}

\item If $2\le r\le d-1$, then the map $b_r$ is explicitly given as follows.
\begin{enumerate}[\rm(i)] 
\item\label{Best-b-i} If 
$\theta\in \bigwedge^rU_0$,  and $m\in \binom{x_2,\dots,x_d}{n}$, then $\eta(\theta\otimes m^*)\in K_{r-1,n-1}U_0^*$ and
$$ b_r(\eta(\theta\otimes m^*))=\begin{cases} - \sum\limits_{j=2}^{d}\sum\limits_{m_1\in \binom{x_1,\dots,x_d}{n-2}}\sum\limits_{m_2\in \binom{x_2,\dots,x_d}{n}}\chi(x_j|m)x_1t_{m_1m_2} Q_{x_1m_1,\frac{m}{x_j}}\otimes \eta(x_j^*(\theta)\otimes  m_2^*)\\
-\pmb\delta \sum\limits_{j=2}^{d}x_j\otimes  \eta(x_j^*(\theta)\otimes m^*)\\
\hline
-\sum\limits_{j=2}^{d}\sum\limits_{m_1\in\binom{x_2,\dots,x_d}{n-1}}\chi(x_j|m)x_1Q_{m_1,\frac{m}{x_j}}\otimes  \kappa (x_j^*(\theta)\otimes m_1).
\end{cases}$$
 \item\label{prove-me} If  
$\theta\in \bigwedge^rU_0$,  and $m\in \binom{x_2,\dots,x_d}{n-1}$, then 
$\kappa(\theta\otimes m)\in L_{r-1,n}U_0$ and
$$ b_r(\kappa(\theta\otimes m))=\begin{cases}
+ \sum\limits_{j=2}^d\sum\limits_{m_1,m_2\in\binom{x_1,\dots,x_d}{n-2}}
\sum\limits_{m_3\in\binom{x_2,\dots,x_d}{n}}
x_1t_{x_jmm_2}Q_{x_1m_1,x_1m_2}t_{m_1m_3} \otimes \eta(x_j^*(\theta)\otimes m_3^*)\\
\hline
+\sum\limits_{j=2}^d\sum\limits_{m_1\in\binom{x_2,\dots,x_d}{n-1}}\sum\limits_{m_2\in\binom{x_1,\dots,x_d}{n-2}}x_1t_{x_jmm_2}Q_{m_1,x_1m_2} \otimes \kappa\left(x_j^*(\theta)\otimes m_1\right)\\
-\pmb \delta\sum\limits_{j=2}^d x_j\otimes \kappa (x_j^*(\theta)\otimes m).
\end{cases}$$
\end{enumerate}
\item\label{Best-c} If $\theta\in \bigwedge^{d-1}U_0$, then $\theta\in B_d$ and
$$b_d(\theta)=\begin{cases} 
\phantom{+}\sum\limits_{m\in\binom{x_2,\dots,x_d}{n}} 
[\pmb\delta m -\sum\limits_{m_1\in \binom{x_1, \ldots, x_d}{n-2}}\sum\limits_{m_2\in \binom{x_1, \ldots, x_d}{n-1}} x_1m_2Q_{m_2,x_1m_1}
t_{m_1m}]
\otimes \eta(\theta\otimes m^*)
\vspace{5pt}\\\hline -\sum\limits_{m\in\binom{x_2,\dots,x_d}{n-1}} 
\sum\limits_{m_1\in \binom{x_1, \ldots, x_d}{n-1}} x_1m_1Q_{m_1,m}
\otimes \kappa(\theta \otimes m).
\end{cases}$$
\end{enumerate}\end{proposition}

\begin{proof}
Assertions (\ref{Best-a}) and (\ref{Best-c}) are not difficult to prove. The proofs of (\ref{Best-b-i}) and (\ref{prove-me}) are very similar. We prove (\ref{prove-me}) and suppress the proof of (\ref{Best-b-i}). Fix $r$, $\theta$, and $m$ with 
 $2\le r\le d-1$, $\theta\in \bigwedge^rU_0$,  and $m\in \binom{x_2,\dots,x_d}{n-1}$. Let $\Theta=\kappa(\theta\otimes m)$. We compute
$b_r(\Theta)$. We know that $\kappa(\theta\otimes m)=\sum_{j=2}^d x_j^*(\theta)\otimes x_jm$; so we  apply Definition~\ref{diff}.\ref{diff-b} to see that
$b_r(\Theta)=\sum_{i=1}^3S_i$, with
\begingroup
\allowdisplaybreaks
 \begin{align}
S_1&=x_1\otimes \sum\limits_{j=2}^d \eta\left(x_j^*(\theta)\otimes (q[[x_jm](\widetilde{\Phi})])(\widetilde{\Phi})\right),\notag\\
S_2&=x_1\otimes \sum\limits_{j=2}^d \kappa\left(x_j^*(\theta)\otimes (\operatorname{proj}\circ q)[(x_jm)(\widetilde{\Phi})]\right),\text{ and}\notag\\
S_3&=\pmb \delta\sum\limits_{j=2}^d {{\mathrm {Kos}}}^{\Psi}(x_j^*(\theta))\otimes x_jm.\notag
\end{align}
\endgroup
The homomorphisms ${{\mathrm {Kos}}}^{\Psi}$ and $\kappa$ anti-commute; so,
$$S_3=\pmb \delta ({{\mathrm {Kos}}}^{\Psi}\circ \kappa)(\Theta)=-\pmb \delta (\kappa\circ {{\mathrm {Kos}}}^{\Psi})(\Theta);$$ which is the third summand of $b_r(\Theta)$ as given in the statement of (\ref{prove-me}). Apply (\ref{m-of-superPhi-}) to $(x_jm)(\widetilde{\Phi})$, followed by (\ref{q-}), to write
\begingroup\allowdisplaybreaks
\begin{align}
S_1&=x_1\otimes \sum\limits_{j=2}^d\sum\limits_{m_2\in\binom{x_1,\dots,x_d}{n-2}}t_{x_jmm_2} \eta\left(x_j^*(\theta)\otimes (q[(x_1m_2)^*])(\widetilde{\Phi})\right)\notag\\&=x_1\otimes \sum\limits_{j=2}^d\sum\limits_{m_2\in\binom{x_1,\dots,x_d}{n-2}}
\sum\limits_{m_1\in\binom{x_1,\dots,x_d}{n-1}}t_{x_jmm_2}Q_{m_1,x_1m_2} \eta\left(x_j^*(\theta)\otimes m_1(\widetilde{\Phi})\right)\text{ and}\notag\\
S_2&=x_1\otimes \sum\limits_{j=2}^d\sum\limits_{m_2\in\binom{x_1,\dots,x_d}{n-2}}t_{x_jmm_2} \kappa\left(x_j^*(\theta)\otimes (\operatorname{proj}\circ q)[(x_1m_2)^*]\right)\notag\\
&=x_1\otimes \sum\limits_{j=2}^d\sum\limits_{m_2\in\binom{x_1,\dots,x_d}{n-2}}\sum\limits_{m_1\in\binom{x_1,\dots,x_d}{n-1}}t_{x_jmm_2}Q_{m_1,x_1m_2} \kappa\left(x_j^*(\theta)\otimes \operatorname{proj}(m_1)\right)\notag\\
&=x_1\otimes \sum\limits_{j=2}^d\sum\limits_{m_2\in\binom{x_1,\dots,x_d}{n-2}}\sum\limits_{m_1\in\binom{x_2,\dots,x_d}{n-1}}t_{x_jmm_2}Q_{m_1,x_1m_2} \kappa(x_j^*(\theta)\otimes m_1).\notag\end{align}\endgroup 
Observe that $S_2$ is the second summand of $b_r(\Theta)$ as given in the statement of (\ref{prove-me}).
The set $\binom{x_1,\dots,x_d}{n-1}$ is the disjoint union  $x_1\binom{x_1,\dots,x_d}{n-2}\cup \binom{x_2,\dots,x_d}{n-1};$ and therefore, $S_1=S_1'+S_1''$ with
$$\begin{array}{lll}
S_1'&=&x_1\otimes \sum\limits_{j=2}^d\sum\limits_{m_2\in\binom{x_1,\dots,x_d}{n-2}}
\sum\limits_{m_1\in\binom{x_1,\dots,x_d}{n-2}}t_{x_jmm_2}Q_{x_1m_1,x_1m_2} \eta\left(x_j^*(\theta)\otimes (x_1m_1)(\widetilde{\Phi})\right)\text{ and}\\
S_1''&=&x_1\otimes \sum\limits_{j=2}^d\sum\limits_{m_2\in\binom{x_1,\dots,x_d}{n-2}}
\sum\limits_{m_1\in\binom{x_2,\dots,x_d}{n-1}}t_{x_jmm_2}Q_{m_1,x_1m_2} \eta\left(x_j^*(\theta)\otimes m_1(\widetilde{\Phi})\right).
\end{array}$$
The main property of $\widetilde{\Phi}$ is $x_1(\widetilde{\Phi})=\Phi$. Apply (\ref{m-of-Phi-}) to $S_1'$ and (\ref{m-of-superPhi-}) to $S_1''$ to obtain
$$\begin{array}{lll}
S_1'&=&x_1\otimes \sum\limits_{j=2}^d\sum\limits_{m_2\in\binom{x_1,\dots,x_d}{n-2}}
\sum\limits_{m_1\in\binom{x_1,\dots,x_d}{n-2}}
\sum\limits_{m_3\in\binom{x_1,\dots,x_d}{n}}t_{x_jmm_2}t_{m_1m_3}Q_{x_1m_1,x_1m_2} \eta(x_j^*(\theta)\otimes m_3^*)\text{ and}\\
S_1''&=&x_1\otimes \sum\limits_{j=2}^d\sum\limits_{m_2\in\binom{x_1,\dots,x_d}{n-2}}
\sum\limits_{m_1\in\binom{x_2,\dots,x_d}{n-1}}\sum\limits_{m_3\in\binom{x_1,\dots,x_d}{n-1}}t_{x_jmm_2}t_{m_1m_3}Q_{m_1,x_1m_2} \eta(x_j^*(\theta)\otimes (x_1m_3)^*).
\end{array}$$Partition the set $\binom{x_1,\dots,x_d}{n}$  into two subsets and write $S_1'=A+B$ with
$$\begin{array}{lll}
A&=&x_1\otimes \sum\limits_{j=2}^d\sum\limits_{m_2\in\binom{x_1,\dots,x_d}{n-2}}
\sum\limits_{m_1\in\binom{x_1,\dots,x_d}{n-2}}
\sum\limits_{m_3\in\binom{x_2,\dots,x_d}{n}}t_{x_jmm_2}t_{m_1m_3}Q_{x_1m_1,x_1m_2} \eta(x_j^*(\theta)\otimes m_3^*)\text{ and}\\
B&=&x_1\otimes \sum\limits_{j=2}^d\sum\limits_{m_2\in\binom{x_1,\dots,x_d}{n-2}}
\sum\limits_{m_1\in\binom{x_1,\dots,x_d}{n-2}}
\sum\limits_{m_3\in\binom{x_1,\dots,x_d}{n-1}}t_{x_jmm_2}t_{m_1x_1m_3}Q_{x_1m_1,x_1m_2} \eta(x_j^*(\theta)\otimes (x_1m_3)^*).\end{array}$$ Observe that $A$ is the first summand of $b_r(\Theta)$ as given in the statement of (\ref{prove-me}). 

We now prove that $S_1''+B=0$. Combine $x_1m_1\in x_1\binom {x_1,\dots,x_d}{n-2}$ from $B$ and $m_1\in\binom{x_2,\dots,x_d}{n-1}$ from $S_1''$ to see that 
$$S_1''+B=x_1\otimes \sum\limits_{j=2}^d\sum\limits_{m_2\in\binom{x_1,\dots,x_d}{n-2}}
\sum\limits_{m_1\in\binom{x_1,\dots,x_d}{n-1}}\sum\limits_{m_3\in\binom{x_1,\dots,x_d}{n-1}}t_{x_jmm_2}t_{m_1m_3}Q_{m_1,x_1m_2} \eta(x_j^*(\theta)\otimes (x_1m_3)^*).$$ Now use (\ref{TQ2}) in the form
$$\sum\limits_{m_1\in\binom{x_1,\dots,x_d}{n-1}}t_{m_1m_3}Q_{m_1,x_1m_2} =\pmb\delta\chi(m_3=x_1m_2)$$ to conclude that
$$S_1''+B=\pmb \delta x_1\otimes \sum\limits_{j=2}^d\sum\limits_{m_2\in\binom{x_1,\dots,x_d}{n-2}}
t_{x_jmm_2} \eta(x_j^*(\theta)\otimes (x_1^2m_2)^*).$$
Notice that 
$$\begin{array}{lll}
\eta(\sum\limits_{m_2\in\binom{x_1,\dots,x_d}{n-1}}
t_{mm_2} 
\theta
\otimes (x_1^2m_2)^*))&=&
\sum\limits_{j=2}^d
\sum\limits_{m_2\in\binom{x_1,\dots,x_d}{n-1}}
t_{mm_2}\chi(x_j|m_2) (
x_j^*(
\theta
)
\otimes (x_1^2\frac{m_2}{x_j})^*)\\&=&\sum\limits_{j=2}^d
\sum\limits_{m_2\in\binom{x_1,\dots,x_d}{n-2}}
t_{x_jmm_2} 
x_j^*(
\theta
)
\otimes (x_1^2m_2)^*;\end{array}$$hence, 
$$S_1''+B=\pmb \delta x_1\otimes (\eta\circ \eta)(\sum\limits_{m_2\in\binom{x_1,\dots,x_d}{n-1}}
t_{mm_2} 
\theta
\otimes (x_1^2m_2)^*))=0$$ and the proof of (\ref{prove-me}) is complete.
\end{proof}

 Our conventions pertaining to monomials are contained in Conventions~\ref{conv1}. In particular, in  \ref{conv1}.\ref{conv-init} we state 
that if $m$ is a monomial of positive degree in the variables $x_1,\dots,x_d$, then   ``${\operatorname{init}}(m)=x_i$'' and ``$\operatorname{least}(m)=i$'' both  mean that $i$ is the least index for which $x_i|m$.
 Consider $2\le r\le d$ and let $\Theta$ be a standard basis element of $B_r$. (Our use of this expression is explained below.) Proposition~\ref{Best} expresses $b_r(\Theta)$ as a linear combination of elements of $B_{r-1}$ of the form \begin{equation}\label{el-gen}\eta(\theta\otimes m^*)\quad\text{and}\quad  \kappa(\theta\otimes m),\end{equation} where each $\theta$ has the form $x_{a_1}\wedge \dots \wedge x_{a_{r-1}}$, with $2\le a_1<\dots<a_{r-1}\le d$, and each $m$ is a monomial  in $\{x_2,\dots,x_d\}$. (The degree of $m$ is $n$ if $m$ appears in $\eta(\theta\otimes m^*)$ and the degree of $m$ is $n-1$ if $m$ appears in $\kappa(\theta\otimes m)$.) We call the  elements with form given by (\ref{el-gen}) {\it elementary generators} of $B_{r-1}$. 
The idea of {\it standard basis elements} of $K_{*,*}$ and $L_{*,*}$ appears many places; see for example,  \cite[(5.4) and (5.5)]{EKK},  Remark~\ref{rmk2},  \cite[Examples 2.1.3.h and 2.1.17.h]{W}, or \cite[Sect.III.1]{BB}.

\begin{itemize}
\item The 
elementary generator $\eta(\theta\otimes m^*)$ of the form (\ref{el-gen})  
is a standard basis element for $K_{r-2,n-1}U_0$ if and only if 
$\{i\mid 2\le i\le \operatorname{least} (m)\} \subseteq \{a_j\mid 1\le j\le r-1\}$.
\item The elementary generator $\kappa(\theta\otimes m)$ of the form (\ref{el-gen})
is a standard basis element for $L_{r-2,n}U_0$ if and only if $a_1\le \operatorname{least}(m)$. 
\end{itemize}

\noindent In Definition~\ref{basis-47} we collect the standard basis elements of the relevant $K_{*,*}$ and $L_{*,*}$ and call these standard basis the basis $\mathfrak B$ for $\mathbb B$. Recall our use of  ``$[a,b]$'' to represent the set of all of  integers in a closed interval; see Convention~\ref{conv1}.\ref{conv1-i}.
 \begin{definition}\label{basis-47} Let $(\mathbb B,b)$ be the complex of Definition~{\rm\ref{diff}} and Corollary~{\rm\ref{main-lemma}.\ref{main-lemma-e}}. 
\begin{enumerate}[\rm(a)]
\item The basis for $B_0={\mathfrak R}$ is $Y^{(0)}=1$.
\item
If $1\le r\le d-1$, then the basis for $B_r$ is 
\begin{align}&\left\{X^{(r)}_{a_1,\dots,a_r,m}\left| 2\le a_1<\dots< a_r\le d,\  m\in \binom{x_2,\dots,x_d}{n},\text{ and }  [2,\operatorname{least} (m)]\subseteq\{a_1,\dots,a_r\}\right.\right\}\notag\\
\cup&\left\{Y^{(r)}_{a_1,\dots,a_r,m} \left| 2\le a_1<\dots< a_r\le d,\text{ and } m\in \binom{x_{a_1},\dots,x_d}{n-1}\right.\right\} \notag\end{align}
with
\begin{align}X^{(r)}_{a_1,\dots,a_r,m}&=\eta(x_{a_1}\wedge\dots\wedge x_{a_r}\otimes m^*)\in K_{r-1,n-1}U_0\subseteq B_{r} \text{ and}\notag\\
Y^{(r)}_{a_1,\dots,a_r,m}&=\kappa(x_{a_1}\wedge\dots\wedge x_{a_r}\otimes m)\in L_{r-1,n}U_0\subseteq B_{r}. \notag\end{align}

\item The basis for $B_d={\mathfrak R}\otimes \bigwedge^{d-1}U_0$ is $X^{(d)}=x_2\wedge \dots \wedge x_d$.

\item Let $\mathfrak B$ represent the union of all of the above bases; so $\mathfrak B$ is a basis for $\mathbb B$.

\end{enumerate}\end{definition}

Theorem~\ref{47.4} gives an explicit matrix version of $b_r$ in terms of the basis $\mathfrak B$ for all $r$ and $d$.  Lemma~\ref{May31-impt}  expresses the elementary generators of $\mathbb B$ in terms of the standard basis $\mathfrak B$. This result is used often in the proof of Theorem~\ref{47.4}. Observation~\ref{behavior} describes how the perfect pairings ${\operatorname{pp}_r: B_r\otimes B_{d-r}\to B_d}$ of Definition~\ref{ppr-def} behave on the basis $\mathfrak B$ of Definition~\ref{basis-47}. 
This result was promised at (\ref{promise}).

\begin{lemma}\label{May31-impt} Adopt the language of Definition~{\rm\ref{basis-47}} and fix an integer $r$ with $2\le r\le d-1$.  Let
 $m\in \binom{x_2,\dots,x_d}{n}$ be a monomial and $2\le c_1<\dots<c_{r-1}\le d$ be integers. 
\begin{enumerate}[\rm(a)]
\item\label{May31-impt-a}Identify the largest integer $g$   with $[2,g]\subseteq \{c_1,\dots,c_{r-1}\}$. Then $\eta(x_{c_1}\wedge\dots\wedge x_{c_{r-1}}\otimes m^*)$ is equal to $$\begin{cases}
X^{(r-1)}_{c_1,\dots,c_{r-1},m},&\text{if $\operatorname{least}(m)\le g$, and}\\
\sum\limits_{k=g}^{r-1}(-1)^{k+g}\chi(x_{c_k}|m)X^{(r-1)}_{c_1,\dots,c_{g-1},g+1,c_g,\dots,\widehat{c_k},\dots, c_{r-1},
\frac{x_{g+1}m}{x_{c_k}}
},&\text{if $g+1\le \operatorname{least}(m)$.}\end{cases}$$
\item\label{May31-impt-53.2} The following equality holds$:$  $$\kappa(x_{c_1}\wedge\dots  \wedge x_{c_{r-1}}\otimes m)=\begin{cases}
Y^{(r-1)}_{c_1,\dots,c_{r-1},m},&\text{if $c_1\le \operatorname{least}(m)$, and}\\
\sum\limits^{r-1}_{k=1}(-1)^{k+1}Y^{(r-1)}_{\operatorname{least}(m),c_1,\dots,\widehat{c_k},\dots,c_{r-1},\frac{x_{c_k}m}{{\operatorname{init}}(m)}},
&\text{if $\operatorname{least}(m)<c_1$}.\end{cases}$$
\end{enumerate}
\end{lemma}
\begin{Remark} The parameter $g$ in (\ref{May31-impt-a}) satisfies $1\le g\le r$. If $g=r$ and $g+1\le \operatorname{least}(m)$, then the sum in line 2 in (\ref{May31-impt-a}) is taken over the empty set and is therefore zero.\end{Remark} 
\begin{proof} (\ref{May31-impt-a}) If $\operatorname{least}(m)\le g$, then $[2,\operatorname{least}(m)]\subseteq[2,g]\subseteq \{c_1,\dots,c_{r-1}\}$, so $$\eta(x_{c_1}\wedge\dots\wedge x_{c_{r-1}}\otimes m^*)=X^{(r-1)}_{c_1,\dots,c_{r-1},m}\in \mathfrak B.$$ If 
$g+1\le \operatorname{least}(m)$, then 
\begin{align}0&=\eta\eta(x_{c_1}\wedge\dots\wedge x_{c_{g-1}}\wedge x_{g+1}\wedge x_{c_g}\wedge\dots\wedge x_{c_{r-1}}\otimes (x_{g+1}m)^*)\notag\\
&=\begin{cases}(-1)^{g+1}\eta(x_{c_1}\wedge\dots\wedge x_{c_{r-1}}\otimes m^*)\\
+\sum\limits_{k=g}^{r-1}(-1)^k\eta(x_{c_1}\wedge\dots\wedge x_{c_{g-1}}\wedge x_{g+1}\wedge x_{c_g}\wedge\dots\wedge\widehat{x_{c_k}}\wedge\dots\wedge x_{c_{r-1}}\otimes x_{c_k}(x_{g+1}m)^*).\end{cases}\notag\end{align}
So 
$$\begin{array}{ll}&\eta(x_{c_1}\wedge\dots\wedge x_{c_{r-1}}\otimes m^*)\\
=&(-1)^g\sum\limits_{k=g}^{r-1}(-1)^k\eta(x_{c_1}\wedge\dots\wedge x_{c_{g-1}}\wedge x_{g+1}\wedge x_{c_g}\wedge\dots\wedge\widehat{x_{c_k}}\wedge\dots\wedge x_{c_{r-1}}\otimes x_{c_k}(x_{g+1}m)^*)
\\
=&(-1)^g\sum\limits_{k=g}^{r-1}(-1)^k\chi(x_{c_k}|m)\eta(x_{c_1}\wedge\dots\wedge x_{c_{g-1}}\wedge x_{g+1}\wedge x_{c_g}\wedge\dots\wedge\widehat{x_{c_k}}\wedge\dots\wedge x_{c_{r-1}}\otimes (\frac{x_{g+1}m}{x_{c_k}})^*)\\
=&(-1)^g\sum\limits_{k=g}^{r-1}(-1)^k\chi(x_{c_k}|m)
X^{(r-1)}_{c_1,\dots, c_{g-1},g+1,c_g,\dots,\widehat{c_k},\dots,c_{r-1},\frac{x_{g+1}m}{x_{c_k}}}.
\end{array}$$

\medskip\noindent (\ref{May31-impt-53.2}) If $c_1\le \operatorname{least}(m)$, then there is nothing to prove. If $\operatorname{least}(m)<c_1$, then
$$\textstyle 0=\kappa\kappa({\operatorname{init}}(m)\wedge x_{c_1}\wedge\dots  \wedge x_{c_{r-1}}\otimes \frac{m}{{\operatorname{init}}(m)})$$$$\textstyle =
\kappa(x_{c_1}\wedge\dots  \wedge x_{c_{r-1}}\otimes m)+\sum\limits_{k=1}^{r-1} (-1)^{k}
\kappa({\operatorname{init}}(m)\wedge x_{c_1}\wedge\dots\wedge \widehat{x_{c_k}}\wedge\dots  \wedge x_{c_{r-1}}\otimes \frac{x_{c_k}m}{{\operatorname{init}}(m)}).$$\end{proof}

\bigskip Observation~\ref{behavior} describes how the perfect pairings $\operatorname{pp}_r: B_r\otimes B_{d-r}\to B_d$ of Definition~\ref{ppr-def} behave on the basis $\mathfrak B$ of Definition~\ref{basis-47}. 
This result was promised at (\ref{promise}). 

\begin{observation}\label{behavior}If $Y_{a_1,\dots,a_r,m_1}^{(r)}$ and $X_{c_1,\dots,c_{d-r},m_2}^{(d-r)}$ are in $\mathfrak B$, then
\begin{align}\operatorname{pp}(Y_{a_1,\dots,a_r,m_1}^{(r)}\otimes X_{c_1,\dots,c_{d-r},m_2}^{(d-r)})&=\chi(x_{a_1}m_1=m_2) x_{a_2}\wedge\dots\wedge x_{a_r}\wedge x_{c_1}\wedge \dots \wedge x_{c_{d-r}}\text{ and}\notag\\
\operatorname{pp}(X_{c_1,\dots,c_{d-r},m_2}^{(d-r)}\otimes Y_{a_1,\dots,a_r,m_1}^{(r)}) 
&=(-1)^{r+d}\chi(x_{a_1}m_1=m_2) x_{c_1}\wedge \dots \wedge x_{c_{d-r}}\wedge x_{a_2}\wedge\dots\wedge x_{a_r}.\notag \end{align}
\end{observation}

\begin{proof} Most of the argument is involved in proving
\begin{equation}\label{assu}
\operatorname{pp}(Y_{a_1,\dots,a_r,m_1}^{(r)}\otimes X_{c_1,\dots,c_{d-r},m_2}^{(d-r)})\neq 0\implies x_{a_1}m_1=m_2.\end{equation}
Assume \begin{equation}\label{not-zero}\operatorname{pp}(Y_{a_1,\dots,a_r,m_1}^{(r)}\otimes X_{c_1,\dots,c_{d-r},m_2}^{(d-r)})\neq 0.\end{equation} 
Recall, from Definition~\ref{ppr-def}, that the left side of (\ref{not-zero}) is
\begin{equation}\label{exp}[m_1(m_2^*)](x_{a_1}\wedge\dots\wedge x_{a_r})\wedge (x_{c_1}\wedge\dots\wedge x_{c_{d-r}});\end{equation}
and therefore, the hypothesis (\ref{not-zero}) forces \begin{equation}\label{concl}[2,d]\subseteq \{a_1,\dots,a_r\}\cup\{c_1,\dots,c_{d-r}\}.\end{equation} Hence, there exists exactly one pair $(j,j')$ with $1\le j\le r$,  $1\le j'\le d-r$, and $a_j=c_{j'}$. The expression (\ref{exp}) is non-zero; hence, $[m_1(m_2^*)](x_{a_j})\neq 0$ and \begin{equation}\label{baby-step}x_{a_j}m_1=m_2.\end{equation}

To complete the proof of (\ref{assu}) we must show that $j=1$. To that end, let $\gamma$ be the largest integer with $[2,\gamma]\subseteq \{c_1,\dots,c_{d-r}\}$. It follows, in particular, that $\gamma+1\in [2,d]\setminus\{c_1,\dots,c_{d-r}\}$; and therefore, by (\ref{concl}), $\gamma+1\in \{a_1,\dots,a_r\}$. We show that $\gamma+1\neq a_1$. Indeed, if $\gamma+1=a_1$, then
$$\gamma+1=a_1\le \operatorname{least}(a_jm_1)=\operatorname{least}(m_2).$$The inequality is due to the fact that $a_1\le a_j$ and $a_1\le\operatorname{least}(m_1)$ according to the definition of $Y^{(*)}_*$ in  $\mathfrak B$. The equality on the right follows from (\ref{baby-step}). In this case, $$\gamma+1\in [2,\operatorname{least}(m_2)]\subseteq\{c_1,\dots,c_{d-r}\};$$which is a contradiction. (For the inclusion of $[2,\operatorname{least}(m_2)]$ in the set of $c$'s, we have used  the definition of $X^{(*)}_{*}$ in $\mathfrak B$.) We now know that $\gamma+1=a_i$ for some $i$ with $2\le i$. Thus, 
$$a_1<a_i=\gamma+1\implies a_1\le \gamma\implies a_1\in[2,\gamma]\subseteq \{c_1,\dots,c_{d-r}\}.$$The subscript $j$ is defined by the property that $j$ is the unique subscript with $a_j\in\{c_1,\dots,c_{d-r}\}$. We have proven that $j=1$; so, (\ref{baby-step}) completes the proof of (\ref{assu}).

The implication (\ref{assu}) may be rephrased as
$$\operatorname{pp}(Y_{a_1,\dots,a_r,m_1}^{(r)}\otimes X_{c_1,\dots,c_{d-r},m_2}^{(d-r)})
=\chi(x_{a_1}m_1=m_2)\operatorname{pp}(Y_{a_1,\dots,a_r,m_1}^{(r)}\otimes X_{c_1,\dots,c_{d-r},m_2}^{(d-r)}).$$
Use Definition \ref{ppr-def} to see that the most recent expression  
\begin{align}&=\chi(x_{a_1}m_1=m_2)[m_1(m_2^*)](x_{a_1}\wedge\dots\wedge x_{a_r})\wedge x_{c_1}\wedge \dots x_{c_{d-r}}\notag\\
&=\chi(x_{a_1}m_1=m_2)(x_{a_2}\wedge\dots\wedge x_{a_r})\wedge x_{c_1}\wedge \dots x_{c_{d-r}}.\notag\end{align}
Similarly, $\operatorname{pp}(X_{c_1,\dots,c_{d-r},m_2}^{(d-r)}\otimes Y_{a_1,\dots,a_r,m_1}^{(r)})$ is equal to
\begin{align}& =\chi(x_{a_1}m_1=m_2)
\operatorname{pp}(X_{c_1,\dots,c_{d-r},m_2}^{(d-r)}\otimes Y_{a_1,\dots,a_r,m_1}^{(r)})\notag\\
&=-\chi(x_{a_1}m_1=m_2) [m_1(m_2^*)](x_{c_1}\wedge \dots x_{c_{d-r}})\wedge x_{a_1}\wedge\dots\wedge x_{a_r}\notag\\
&=-(-1)^{d-r+1}\chi(x_{a_1}m_1=m_2) (x_{c_1}\wedge \dots x_{c_{d-r}})\wedge [m_1(m_2^*)](x_{a_1}\wedge\dots\wedge x_{a_r})\notag\\
&=(-1)^{d+r}\chi(x_{a_1}m_1=m_2) (x_{c_1}\wedge \dots x_{c_{d-r}})\wedge (x_{a_2}\wedge\dots\wedge x_{a_r}).\notag\end{align}
\end{proof}

\begin{theorem}\label{47.4}
Let $(\mathbb B,b)$ be the complex of Definition~{\rm\ref{diff}} and Corollary~{\rm\ref{main-lemma}.\ref{main-lemma-e}} and $\mathfrak B$ be the basis for $\mathbb B$  as given in
 Definition~{\rm \ref{basis-47}}.
\begin{enumerate}[\rm(a)] 
\item\label{47.4-a} The matrix for $b_1$ with respect to the basis $\mathfrak B$ is described as follows.
\begin{enumerate}[\rm(i)]
\item If $X^{(1)}_{a_1,m}$ is in $\mathfrak B$,  then $b_1(X_{a_1,m}^{(1)})=x_1\sum\limits_{m_1\in \binom{x_1, \ldots, x_d}{n-1}} m_1Q_{m_1,\frac{m}{x_{a_1}}}$.

\item If $Y_{a_1,m}^{(1)}$ is in $\mathfrak B$,   then $b_1(Y_{a_1,m}^{(1)})=\pmb\delta x_{a_1}m -x_1\sum\limits_{m_1\in \binom{x_1, \ldots, x_d}{n-2}}\sum\limits_{m_2\in \binom{x_1, \ldots, x_d}{n-1}} m_2Q_{m_2,x_1m_1}
t_{x_{a_1}m_1m}$.
\end{enumerate}
\item If $2\le r\le d-1$, then 
the matrix for $b_r$ with respect to the basis $\mathfrak B$ is described as follows.
\begin{enumerate}[\rm(i)]
\item\label{47.4-bi} If $X^{(r)}_{a_1,\dots,a_r,m}$ is an element of $\mathfrak B$ and 
  $\gamma$ is the largest integer with $[2,\gamma]\subseteq\{a_1,\dots,a_r\}$,
then 
$b_r(X^{(r)}_{a_1,\dots,a_r,m})$ 
is given in Table~{\rm\ref{br-of-X}}.

\item\label{47.4-bii} If  $Y^{(r)}_{a_1,\dots,a_r,m}$ is an element of $\mathfrak B$ and
  $\gamma$ is the largest integer with $[2,\gamma]\subseteq\{a_1,\dots,a_r\}$,
then $b_r(Y^{(r)}_{a_1,\dots,a_r,m})$ is given in Table~{\rm\ref{br-of-Y}}.
\end{enumerate}

\item\label{47.4-d} The matrix for $b_d$ with respect to the basis $\mathfrak B$ is described by
$$b_d(X^{(d)})=\begin{cases} 
\phantom{+}\sum\limits_{m\in\binom{x_2,\dots,x_d}{n}} 
[\pmb\delta m -\sum\limits_{m_1\in \binom{x_1, \ldots, x_d}{n-2}}\sum\limits_{m_2\in \binom{x_1, \ldots, x_d}{n-1}} x_1m_2Q_{m_2,x_1m_1}
t_{m_1m}]
  X^{(d-1)}_{2,\dots,d,m}\vspace{5pt}\\\hline -\sum\limits_{m\in\binom{x_2,\dots,x_d}{n-1}} 
\sum\limits_{m_1\in \binom{x_1, \ldots, x_d}{n-1}} x_1m_1Q_{m_1,m}
  Y^{(d-1)}_{2,\dots,d,m}.
\end{cases}$$
\end{enumerate}

\begin{table}
\begin{center}
$$\begin{cases}
\phantom{+}x_1 \sum\limits_{\ell=2}^{\gamma}\sum\limits_{k=\ell}^{r}\sum\limits_{m_2\in \binom{x_\ell,\dots,x_d}{n-1}}
\sum\limits_{m_1\in \binom{x_1,\dots,x_d}{n-2}}(-1)^{k}
[\chi(x_{a_k}|m)t_{m_1m_2x_\ell} Q_{x_1m_1,\frac{m}{x_{a_k}}} 
-\chi(x_{\ell}|m)t_{m_1m_2x_{a_k}} Q_{x_1m_1,\frac{m}{x_{\ell}}}]
 \\\hskip.2in  \times X^{(r-1)}_{a_1,\dots,\widehat{a_k}, \dots,a_r,x_\ell m_2}
\\
+x_1\sum\limits_{\gamma\le j<k\le r} \sum\limits_{m_2\in \binom{x_{\gamma+1},\dots,x_d}{n-1}}\sum\limits_{m_1\in \binom{x_1,\dots,x_d}{n-2}}(-1)^{\gamma+j+k} \\\hskip.2in 
\times 
\bmatrix \chi(x_{a_k}|m)t_{m_1m_2x_{a_j}} Q_{x_1m_1,\frac{m}{x_{a_k}}} 
- \chi(x_{a_j}|m)t_{m_1m_2x_{a_k}} Q_{x_1m_1,\frac{m}{x_{a_j}}}  
\endbmatrix 
X^{(r-1)}_{a_1,\dots,a_{\gamma-1},\gamma+1,a_{\gamma},\dots,\widehat{a_j},\dots,\widehat{a_k},\dots,a_r,x_{\gamma+1}m_2}\\
+\pmb\delta \sum\limits_{j=1}^{\operatorname{least}(m)-1}\sum\limits_{k=j+1}^{r} x_{a_j}  (-1)^{k+1}\chi(x_{a_k}|m)X^{(r-1)}_{a_1,\dots,\widehat{a_k},\dots, a_r,
\frac{x_{j+1}m}{x_{a_k}}
}\\
+\pmb\delta \sum\limits_{j=\operatorname{least}(m)}^{r}x_{a_j}  (-1)^{j}X^{(r-1)}_{a_1,\dots,\widehat{a_j},\dots,a_{r},m}\\
\hline
+x_1\sum\limits_{k=2}^{r}\sum\limits_{m_1\in\binom{x_{a_1},\dots,x_d}{n-1}}(-1)^{k}
\bmatrix
\chi(x_{a_k}|m)Q_{m_1,\frac{m}{x_{a_k}}} 
-\chi(x_{a_k}|m_1)\chi(x_{a_1}|m)Q_{\frac{x_{a_1}m_1}{x_{a_k}},\frac{m}{x_{a_1}}}  
\endbmatrix
Y^{(r-1)}_{a_1,\dots,\widehat{a_k},\dots,a_{r},m_1}\\
+x_1\sum\limits_{\ell=a_1+1}^{a_2-1}\sum\limits_{k=2}^{r}\sum\limits_{m_1\in\binom{x_{\ell},\dots,x_d}{n-1}}(-1)^{k+1}\chi(x_{a_k}|m_1)
\chi(x_{a_1}|m)Q_{\frac{x_{\ell}m_1}{x_{a_k}},\frac{m}{x_{a_1}}}  Y^{(r-1)}_{\ell,a_2,\dots,\widehat{a_k},\dots,a_{r},m_1}  \\
-x_1\sum\limits_{m_1\in\binom{x_{a_2},\dots,x_d}{n-1}}\chi(x_{a_1}|m)Q_{m_1,\frac{m}{x_{a_1}}}  Y^{(r-1)}_{a_2,\dots,a_{r},m_1}\\
\end{cases}$$
\caption{{\bf The value of $b_r(X^{(r)}_{a_1,\dots,a_r,m})$ for Theorem~\ref{47.4}.}}\label{br-of-X}
\end{center}\end{table}

\begin{table}
\begin{center}
$$\hskip-1in\begin{cases}
\phantom{+}x_1\sum\limits_{\ell=2}^{\gamma}\sum\limits_{k=\ell}^{r}
\sum\limits_{m_3\in\binom{x_\ell,\dots,x_d}{n-1}}
\sum\limits_{m_1,m_2\in\binom{x_1,\dots,x_d}{n-2}}
(-1)^{k}Q_{x_1m_1,x_1m_2}
[t_{x_{\ell}mm_2}t_{x_{a_k}m_1m_3}  
-t_{x_{a_k}mm_2}t_{x_{\ell}m_1m_3}]
X^{(r-1)}_{a_1,\dots,\widehat{a_k},\dots,a_r,x_{\ell}m_3}\\
+x_1
\sum\limits_{\gamma\le j<k\le r}
\sum\limits_{m_3\in\binom{x_{\gamma+1},\dots,x_d}{n-1}}
\sum\limits_{m_1,m_2\in\binom{x_1,\dots,x_d}{n-2}}
(-1)^{j+\gamma+k}Q_{x_1m_1,x_1m_2}\bmatrix 
t_{x_{a_j}mm_2}t_{x_{a_k}m_1m_3} 
-t_{x_{a_k}mm_2}t_{x_{a_j}m_1m_3}
\endbmatrix\\\hskip.2in\times
X^{(r-1)}_{a_1,\dots,a_{\gamma-1},\gamma+1,a_{\gamma},\dots,\widehat{a_j},
\dots,\widehat{a_k},\dots,a_r,x_{\gamma+1}m_3}
\\
\hline 
+x_1\sum\limits_{\ell=2}^{a_1-1}\sum\limits_{1\le j<k\le r}\sum\limits_{m_1\in\binom{x_{\ell},\dots,x_d}{n-1}}\sum\limits_{m_2\in\binom{x_1,\dots,x_d}{n-2}}(-1)^{k+j}
\bmatrix
\chi(x_{a_j}|m_1)t_{x_{a_k}mm_2}Q_{\frac{x_{\ell}m_1}{x_{a_j}},x_1m_2}
-\chi(x_{a_k}|m_1)t_{x_{a_j}mm_2}Q_{\frac{x_{\ell}m_1}{x_{a_k}},x_1m_2}\endbmatrix
\\\hskip.2in \times Y^{(r-1)}_{\ell,a_1,\dots,
\widehat{a_j},\dots \widehat{a_k},\dots, a_r,m_1}
\\ 
+x_1\sum\limits_{k=2}^r\sum\limits_{m_1\in\binom{x_{a_1},\dots,x_d}{n-1}}\sum\limits_{m_2\in\binom{x_1,\dots,x_d}{n-2}}(-1)^{k}
[\chi(x_{a_k}|m_1)t_{x_{a_1}mm_2}Q_{\frac{x_{a_1}m_1}{x_{a_k}},x_1m_2} -t_{x_{a_k}mm_2}Q_{m_1,x_1m_2} ]
Y^{(r-1)}_{a_1,\dots,\widehat{a_k},\dots, a_r,m_1}\\ 
+x_1\sum\limits_{\ell=a_1+1}^{a_2-1}\sum\limits_{k=2}^r\sum\limits_{m_1\in\binom{x_{\ell},\dots,x_d}{n-1}}\sum\limits_{m_2\in\binom{x_1,\dots,x_d}{n-2}}(-1)^{k}\chi(x_{a_k}|m_1)t_{x_{a_1}mm_2}Q_{\frac{x_{\ell}m_1}{x_{a_k}},x_1m_2} Y^{(r-1)}_{\ell,a_2,\dots,\widehat{a_k},\dots,a_r,
m_1}\\
+x_1\sum\limits_{m_1\in\binom{x_{a_2},\dots,x_d}{n-1}}\sum\limits_{m_2\in\binom{x_1,\dots,x_d}{n-2}}t_{x_{a_1}mm_2}Q_{m_1,x_1m_2}  Y^{(r-1)}_{a_2,\dots,a_r,m_1}\\
+\pmb \delta\sum\limits_{j=2}^r(-1)^{j} x_{a_j} 
Y^{(r-1)}_{a_1,\dots,\widehat{a_j},\dots,a_r,m}\\
-\pmb \delta x_{a_1}\chi(a_2\le \operatorname{least} (m))Y^{(r-1)}_{a_2,\dots,a_r,m}\\
-\pmb \delta x_{a_1}\chi(\operatorname{least}(m)<a_2)\sum\limits_{k=2}^{r}(-1)^k Y^{(r-1)}_{\operatorname{least}(m),a_2,\dots,\widehat{a_k},\dots,a_r,\frac{x_{a_k}m}{{\operatorname{init}}(m)}}
\end{cases}
$$
\caption{{\bf The value of $b_r(Y^{(r)}_{a_1,\dots,a_r,m})$ for Theorem~\ref{47.4}.}}\label{br-of-Y}
\end{center}\end{table}
\end{theorem}

\begin{proof} Assertions (\ref{47.4-a}) and  (\ref{47.4-d}) require no proof. We first prove (\ref{47.4-bi}). Fix an element $X^{(r)}_{a_1,\dots,a_{r},m}$ of $\mathfrak B$ and let $\gamma$ be the largest integer with
$[2,\gamma]\subseteq\{a_1,\dots,a_r\}$. Apply Proposition~\ref{Best} to write $b_r(X^{(r)}_{a_1,\dots,a_{r},m})$ as $\sum_{i=1}^3S_i$, with

\begingroup\allowdisplaybreaks
\begin{align}
S_1&=
 \sum\limits_{j=1}^{r}\sum\limits_{m_1\in \binom{x_1,\dots,x_d}{n-2}}\sum\limits_{m_2\in \binom{x_2,\dots,x_d}{n}}\chi(x_{a_j}|m)x_1t_{m_1m_2} Q_{x_1m_1,\frac{m}{x_{a_j}}}\otimes (-1)^{j}\eta((x_{a_1}\wedge\dots\wedge\widehat{x_{a_j}}\wedge\dots \wedge x_{a_r})\otimes  m_2^*),\notag\\
S_2&=\pmb\delta \sum\limits_{j=1}^{r}x_{a_j}\otimes (-1)^{j} \eta((x_{a_1}\wedge\dots\wedge\widehat{x_{a_j}}\wedge\dots \wedge x_{a_r})\otimes m^*), \text{ and}\notag\\
S_3&=
\sum\limits_{j=1}^{r}\sum\limits_{m_1\in\binom{x_2,\dots,x_d}{n-1}}\chi(x_{a_j}|m)x_1Q_{m_1,\frac{m}{x_{a_j}}}\otimes (-1)^{j} \kappa ((x_{a_1}\wedge\dots\wedge\widehat{x_{a_j}}\wedge\dots \wedge x_{a_r})\otimes m_1).\notag
\end{align}\endgroup
We first show that 
\begin{claim}\label{first-claim}$S_1$ is equal to the sum of the first two summands of $b_r(X^{(r)}_{a_1,\dots,a_r,m})$, as given in Table~{\rm\ref{br-of-X}}.\end{claim} 
\noindent Write $S_1$ as $S_{1,1}+S_{1,2}$. In $S_{1,1}$, the parameter $j$ satisfies $1\le j\le \gamma-1$; whereas, in $S_{1,2}$, $j$ satisfies $\gamma\le j\le r$. Apply Lemma~\ref{May31-impt}.\ref{May31-impt-a} to obtain $S_{1,1}=S_{1,1}'+S_{1,1}''$ with $S_{1,1}'$ equal to
$$\sum\limits_{j=1}^{\gamma-1}\sum\limits_{m_1\in \binom{x_1,\dots,x_d}{n-2}}\sum\limits_{m_2\in \binom{x_2,\dots,x_d}{n}}\chi(x_{a_j}|m)x_1t_{m_1m_2} Q_{x_1m_1,\frac{m}{x_{a_j}}} (-1)^{j}\chi(\operatorname{least}(m_2)\le j)X^{(r-1)}_{a_1,\dots,a_{j-1},\widehat{a_j}, a_{j+1},\dots,a_r,m_2}$$
and 
 $$S_{1,1}''=\begin{cases} \sum\limits_{j=1}^{\gamma-1}\sum\limits_{m_1\in \binom{x_1,\dots,x_d}{n-2}}\sum\limits_{m_2\in \binom{x_2,\dots,x_d}{n}}\chi(x_{a_j}|m)x_1t_{m_1m_2} Q_{x_1m_1,\frac{m}{x_{a_j}}} (-1)^{j}\chi(j+1\le \operatorname{least}(m_2))\\ \hskip.2in \times \sum\limits_{k=j+1}^r(-1)^{j+1+k}\chi(x_{a_k}|m_2)X^{(r-1)}_{a_1,\dots,a_{j-1},j+1,a_{j+1}
\widehat{a_k},\dots,a_r,\frac{x_{j+1}m_2}{x_{a_k}}}.\end{cases}$$
In $S_{1,1}''$, let $\ell=j+1$. Keep in mind that $a_j=j+1$, which becomes $\ell$. Of course, as $j$ roams over $[1,\gamma-1]$, $j+1$ (which becomes $\ell$) roams over $[2,\gamma]$. It follows that 

\begingroup\allowdisplaybreaks
\begin{align}
S_{1,1}''&=\begin{cases}\sum\limits_{\ell=2}^{\gamma}\sum\limits_{m_1\in \binom{x_1,\dots,x_d}{n-2}}\sum\limits_{m_2\in \binom{x_2,\dots,x_d}{n}}\chi(x_{\ell}|m)x_1t_{m_1m_2} Q_{x_1m_1,\frac{m}{x_{\ell}}} \chi(\ell\le \operatorname{least}(m_2))\\ \hskip.2in\times\sum\limits_{k=\ell}^r(-1)^{k+1}\chi(x_{a_k}|m_2)X^{(r-1)}_{a_1,\dots,
\widehat{a_k},\dots,a_r,\frac{x_{\ell}m_2}{x_{a_k}}}\end{cases}\notag\\
&=x_1 \sum\limits_{\ell=2}^{\gamma}\sum\limits_{k=\ell}^r
\sum\limits_{m_2\in \binom{x_{\ell},\dots,x_d}{n}}\sum\limits_{m_1\in \binom{x_1,\dots,x_d}{n-2}}(-1)^{k+1}\chi(x_{\ell}|m)t_{m_1m_2} Q_{x_1m_1,\frac{m}{x_{\ell}}}\chi(x_{a_k}|m_2)X^{(r-1)}_{a_1,\dots,
\widehat{a_k},\dots,a_r,\frac{x_{\ell}m_2}{x_{a_k}}}.\notag
\end{align}\endgroup 
Replace the old $m_2$ with $x_{a_k}$ times the new $m_2$ to obtain 
 \begin{equation}\label{47.4.eq1}
S_{1,1}''=x_1 \sum\limits_{\ell=2}^{\gamma}\sum\limits_{k=\ell}^r
\sum\limits_{m_2\in \binom{x_{\ell},\dots,x_d}{n-1}}\sum\limits_{m_1\in \binom{x_1,\dots,x_d}{n-2}}(-1)^{k+1}\chi(x_{\ell}|m)t_{m_1m_2x_{a_k}} Q_{x_1m_1,\frac{m}{x_{\ell}}}X^{(r-1)}_{a_1,\dots,
\widehat{a_k},\dots,a_r,x_{\ell}m_2}\end{equation}
 
We next transform $S_{1,1}'$.
Let $\ell=\operatorname{least}(m_2)$. Of course $2\le \ell$. On the other hand, $\ell\le j\le \gamma-1$. So, in fact 
the sum could be $\sum\limits_{j=1}^{\gamma-1}\sum\limits_{\ell=2}^j$. Exchange the order of summation. The sum could be
$\sum\limits_{\ell=2}^{\gamma-1}\sum\limits_{j=\ell}^{\gamma-1}$. Indeed, no harm is done if we write
$\sum\limits_{\ell=2}^{\gamma}\sum\limits_{j=\ell}^{\gamma-1}$.
It follows that $S_{1,1}'$ is equal to
\begingroup\allowdisplaybreaks
\begin{align}
&\sum\limits_{\ell=2}^{\gamma}\sum\limits_{j=\ell}^{\gamma-1}
\sum\limits_{m_1\in \binom{x_1,\dots,x_d}{n-2}}\sum\limits_{m_2\in \binom{x_2,\dots,x_d}{n}}
\chi(\operatorname{least}(m_2)=\ell)\chi(x_{a_j}|m)x_1t_{m_1m_2} Q_{x_1m_1,\frac{m}{x_{a_j}}} (-1)^{j}X^{(r-1)}_{a_1,\dots,a_{j-1},\widehat{a_j}, a_{j+1},\dots,a_r,m_2}\notag\\
&=\sum\limits_{\ell=2}^{\gamma}\sum\limits_{j=\ell}^{\gamma-1}
\sum\limits_{m_1\in \binom{x_1,\dots,x_d}{n-2}}\sum\limits_{m_2\in \binom{x_\ell,\dots,x_d}{n}}
\chi(x_\ell|m_2)\chi(x_{a_j}|m)x_1t_{m_1m_2} Q_{x_1m_1,\frac{m}{x_{a_j}}} (-1)^{j}X^{(r-1)}_{a_1,\dots,a_{j-1},\widehat{a_j}, a_{j+1},\dots,a_r,m_2}\notag\\
&= \sum\limits_{\ell=2}^{\gamma}\sum\limits_{j=\ell}^{\gamma-1}
\sum\limits_{m_1\in \binom{x_1,\dots,x_d}{n-2}}\sum\limits_{m_2\in \binom{x_\ell,\dots,x_d}{n-1}}
\chi(x_{a_j}|m)x_1t_{m_1m_2x_\ell} Q_{x_1m_1,\frac{m}{x_{a_j}}} (-1)^{j}X^{(r-1)}_{a_1,\dots,a_{j-1},\widehat{a_j}, a_{j+1},\dots,a_r,x_\ell m_2};
\notag\end{align}\endgroup
and therefore,
\begin{equation}\label{47.4.eq2}S_{1,1}'=x_1 \sum\limits_{\ell=2}^{\gamma}\sum\limits_{k=\ell}^{\gamma-1}\sum\limits_{m_2\in \binom{x_\ell,\dots,x_d}{n-1}}
\sum\limits_{m_1\in \binom{x_1,\dots,x_d}{n-2}}(-1)^{k}
\chi(x_{a_k}|m)t_{m_1m_2x_\ell} Q_{x_1m_1,\frac{m}{x_{a_k}}} X^{(r-1)}_{a_1,\dots,\widehat{a_k}, \dots,a_r,x_\ell m_2}.\end{equation}

We apply the same techniques to $S_{1,2}$. First separate $S_{1,2}$ as the sum 
$S_{1,2}=S_{1,2}'+S_{1,2}''$ where the monomials $m_2$ that appear in $S_{1,2}'$ all satisfy $\operatorname{least}(m_2)\le \gamma$ and the monomials $m_2$ that appear in $S_{1,2}''$ all satisfy $\gamma+1\le \operatorname{least}(m_2)$. 
Apply Lemma~\ref{May31-impt}.\ref{May31-impt-a} to see that

$$S_{1,2}'= \sum\limits_{j=\gamma}^{r}\sum\limits_{m_1\in \binom{x_1,\dots,x_d}{n-2}}\sum\limits_{m_2\in \binom{x_2,\dots,x_d}{n}}\chi(x_{a_j}|m)x_1t_{m_1m_2} Q_{x_1m_1,\frac{m}{x_{a_j}}} (-1)^{j}
X^{(r-1)}_{a_1,\dots,\widehat{a_j},\dots, a_r,m_2}
\chi(\operatorname{least}(m_2)\le \gamma).$$Let $\ell=\operatorname{least}(m_2)$. We see that
\begingroup\allowdisplaybreaks
\begin{align}S_{1,2}'&= \sum\limits_{\ell=2}^{\gamma}\sum\limits_{j=\gamma}^{r}\sum\limits_{m_1\in \binom{x_1,\dots,x_d}{n-2}}\sum\limits_{m_2\in \binom{x_{\ell},\dots,x_d}{n}}\chi(x_{\ell}|m_2)\chi(x_{a_j}|m)x_1t_{m_1m_2} Q_{x_1m_1,\frac{m}{x_{a_j}}} (-1)^{j}
X^{(r-1)}_{a_1,\dots,\widehat{a_j},\dots, x_{a_r},m_2}\notag\\
&= \sum\limits_{\ell=2}^{\gamma}\sum\limits_{j=\gamma}^{r}\sum\limits_{m_1\in \binom{x_1,\dots,x_d}{n-2}}\sum\limits_{m_2\in \binom{x_{\ell},\dots,x_d}{n-1}}\chi(x_{a_j}|m)x_1t_{m_1m_2x_{\ell}} Q_{x_1m_1,\frac{m}{x_{a_j}}} (-1)^{j}
X^{(r-1)}_{a_1,\dots,\widehat{a_j},\dots, x_{a_r},x_{\ell}m_2};\notag
\end{align}\endgroup
and therefore,
\begin{equation}\label{47.4.eq3} S_{1,2}'=x_1 \sum\limits_{\ell=2}^{\gamma}\sum\limits_{k=\gamma}^{r}
\sum\limits_{m_2\in \binom{x_{\ell},\dots,x_d}{n-1}}
\sum\limits_{m_1\in \binom{x_1,\dots,x_d}{n-2}}(-1)^{k}
\chi(x_{a_k}|m)t_{m_1m_2x_{\ell}} Q_{x_1m_1,\frac{m}{x_{a_k}}} 
X^{(r-1)}_{a_1,\dots,\widehat{a_k},\dots, a_r,x_{\ell}m_2}.\end{equation} 
Combine (\ref{47.4.eq1}), (\ref{47.4.eq2}) and (\ref{47.4.eq3}) to see that $S_{1,1}''+S_{1,2}'+S_{1,1}'$ is equal to the first summand of $b_r(X^{(r)}_{a_1,\dots,a_r,m})$, as given in Table~\ref{br-of-X}.

Now we apply   Lemma~\ref{May31-impt}.\ref{May31-impt-a} to
$$S_{1,2}''=\begin{cases} \sum\limits_{j=\gamma}^{r}\sum\limits_{m_1\in \binom{x_1,\dots,x_d}{n-2}}\sum\limits_{m_2\in \binom{x_2,\dots,x_d}{n}}\chi(x_{a_j}|m)x_1t_{m_1m_2} Q_{x_1m_1,\frac{m}{x_{a_j}}}\\\hskip.2in\otimes (-1)^{j}\eta((x_{a_1}\wedge\dots\wedge\widehat{x_{a_j}}\wedge\dots \wedge x_{a_r})\otimes  m_2^*)\chi(\gamma+1\le \operatorname{least}(m_2)).\end{cases}$$
Observe that, when the conditions of $S_{1,2}''$ are in effect, 
$$\begin{array}{ll}&\eta((x_{a_1}\wedge\dots\wedge\widehat{x_{a_j}}\wedge\dots \wedge x_{a_r})\otimes  m_2^*)\\=&\begin{cases}
\phantom{+}\sum\limits_{k=\gamma}^{j-1}(-1)^{\gamma+k} \chi(x_{a_k}|m_2)
X^{(r-1)}_{a_1,\dots,a_{\gamma-1},\gamma+1,a_{\gamma},\dots,\widehat{a_k},\dots,\widehat{a_j},\dots,a_r,\frac{x_{\gamma+1}m_2}{x_{a_k}}}\\
+\sum\limits_{k=j+1}^{r}(-1)^{\gamma+k+1} \chi(x_{a_k}|m_2)
X^{(r-1)}_{a_1,\dots,a_{\gamma-1},\gamma+1,a_{\gamma},\dots,\widehat{a_j},\dots,\widehat{a_k},\dots,a_r,\frac{x_{\gamma+1}m_2}{x_{a_k}}}.
\end{cases}\end{array}$$ So, $S_{1,2}''=T_1+T_2$ with
$$T_1=\begin{cases} \sum\limits_{j=\gamma}^{r}\sum\limits_{m_1\in \binom{x_1,\dots,x_d}{n-2}}\sum\limits_{m_2\in \binom{x_2,\dots,x_d}{n}}\chi(x_{a_j}|m)x_1t_{m_1m_2} Q_{x_1m_1,\frac{m}{x_{a_j}}}\\\times (-1)^{j}
\sum\limits_{k=\gamma}^{j-1}(-1)^{\gamma+k} \chi(x_{a_k}|m_2)
X^{(r-1)}_{a_1,\dots,a_{\gamma-1},\gamma+1,a_{\gamma},\dots,\widehat{a_k},\dots,\widehat{a_j},\dots,a_r,\frac{x_{\gamma+1}m_2}{x_{a_k}}}
\chi(\gamma+1\le \operatorname{least}(m_2))\end{cases}$$and $T_2$ equal to 
$$ \begin{cases} \sum\limits_{j=\gamma}^{r}\sum\limits_{m_1\in \binom{x_1,\dots,x_d}{n-2}}\sum\limits_{m_2\in \binom{x_2,\dots,x_d}{n}}\chi(x_{a_j}|m)x_1t_{m_1m_2} Q_{x_1m_1,\frac{m}{x_{a_j}}}\\\times (-1)^{j}
\sum\limits_{k=j+1}^{r}(-1)^{\gamma+k+1} \chi(x_{a_k}|m_2)
X^{(r-1)}_{a_1,\dots,a_{\gamma-1},\gamma+1,a_{\gamma},\dots,\widehat{a_j},\dots,\widehat{a_k},\dots,a_r,\frac{x_{\gamma+1}m_2}{x_{a_k}}}
\chi(\gamma+1\le \operatorname{least}(m_2)).\end{cases}$$
Replace $m_2$ with $x_{a_k}m_2$.
 The ambient hypothesis $\gamma+1<a_{\gamma}$ guarantees that in both sums ${\gamma+1< a_k}$. It follows that $\chi(\gamma+1\le \operatorname{least}(m_2x_{a_k}))=\chi(\gamma+1\le \operatorname{least}(m_2))$ and this factor may be absorbed into the instruction for choosing $m_2$.
At this point
$$ T_1=\begin{cases}\sum\limits_{\gamma\le k<j\le r} \sum\limits_{m_1\in \binom{x_1,\dots,x_d}{n-2}}\sum\limits_{m_2\in \binom{x_{\gamma+1},\dots,x_d}{n-1}}(-1)^{\gamma+k+j} \chi(x_{a_j}|m)x_1t_{m_1m_2x_{a_k}} Q_{x_1m_1,\frac{m}{x_{a_j}}} \\\hskip.2in \times  
X^{(r-1)}_{a_1,\dots,a_{\gamma-1},\gamma+1,a_{\gamma},\dots,\widehat{a_k},\dots,\widehat{a_j},\dots,a_r,x_{\gamma+1}m_2}\end{cases}
$$
and
 $$ T_2=\begin{cases}\sum\limits_{\gamma\le j<k\le r} \sum\limits_{m_1\in \binom{x_1,\dots,x_d}{n-2}}\sum\limits_{m_2\in \binom{x_{\gamma+1},\dots,x_d}{n-1}}(-1)^{\gamma+k+1+j}\chi(x_{a_j}|m)x_1t_{m_1m_2x_{a_k}} Q_{x_1m_1,\frac{m}{x_{a_j}}}\\\hskip.2in \times   
X^{(r-1)}_{a_1,\dots,a_{\gamma-1},\gamma+1,a_{\gamma},\dots,\widehat{a_j},\dots,\widehat{a_k},\dots,a_r,
x_{\gamma+1}m_2}.\end{cases}
$$ Interchange  $j$ and $k$ in $T_1$ and then observe that $S_{1,2}''$, which is equal to $T_1+T_2$, is equal to the second summand of $b_r(X^{(r)}_{a_1,\dots,a_r,m})$, as given in Table~\ref{br-of-X}.
This, together with the sentence immediately below (\ref{47.4.eq3}), completes the proof of Claim~\ref{first-claim}.

We next show that \begin{claim}\label{second-claim}$S_2$ is equal to the sum of  summands three and four of $b_r(X^{(r)}_{a_1,\dots,a_r,m})$, as given in Table~{\rm\ref{br-of-X}}.\end{claim}

\noindent 
Separate $S_2$ into two summands: one with $1\le j\le \gamma-1$ and the other with $\gamma\le j\le r$. Apply Lemma~\ref{May31-impt}.\ref{May31-impt-a} to both summands. Keep in mind that $X^{(r)}_{a_1,\dots,a_r,m}\in \mathfrak B$; so $[2,\operatorname{least}(m)]\subseteq\{a_1,\dots,a_r\}$. The parameter $\gamma$ is defined to be the largest integer with $[2,\gamma]\subseteq\{a_1,\dots,a_r\}$. Thus, $\operatorname{least}(m)\le \gamma$. At this point, $S_2=\sum_{i=1}^3S_{2,i}$, with 
\begingroup\allowdisplaybreaks
\begin{align}
S_{2,1}&=\pmb\delta \sum\limits_{j=1}^{\gamma-1}x_{a_j}  (-1)^{j}X^{(r-1)}_{a_1,\dots,\widehat{a_j},\dots,a_{r},m}\chi(\operatorname{least}(m)\le j),\notag\\
S_{2,2}&=\pmb\delta \sum\limits_{j=1}^{\gamma-1}x_{a_j}  (-1)^{j}\sum\limits_{k=j+1}^{r}(-1)^{k+j+1}\chi(x_{a_k}|m)X^{(r-1)}_{a_1,\dots,\widehat{a_k},\dots, a_r,
\frac{x_{j+1}m}{x_{a_k}}
}\chi(j+1\le \operatorname{least}(m)), \text{ and}\notag\\
S_{2,3}&=\pmb\delta \sum\limits_{j=\gamma}^{r}x_{a_j}  (-1)^{j}X^{(r-1)}_{a_1,\dots  \widehat{a_j}, \dots , a_r,m}.\notag
\end{align}\endgroup
\noindent Observe that $$S_{2,1}+S_{2,3}=\pmb\delta \sum\limits_{j=\operatorname{least}(m)}^{r}x_{a_j}  (-1)^{j}X^{(r-1)}_{a_1,\dots  \widehat{a_j}, \dots , a_r,m},$$ which is summand $4$ of $b_r(X^{(r)}_{a_1,\dots,a_r,m})$, as given in Table~{\rm\ref{br-of-X}}. Furthermore, we have shown that $\operatorname{least}(m)-1\le \gamma-1$. It follows that
$$
S_{2,2}=\pmb\delta \sum\limits_{j=1}^{\operatorname{least}(m)-1}\sum\limits_{k=j+1}^{r} x_{a_j}  (-1)^{k+1}\chi(x_{a_k}|m)X^{(r-1)}_{a_1,\dots,\widehat{a_k},\dots, a_r,
\frac{x_{j+1}m}{x_{a_k}}
},
$$
which is summand 3 of
$b_r(X^{(r)}_{a_1,\dots,a_r,m})$ as given in Table~{\rm\ref{br-of-X}}.
 This completes the proof of Claim~\ref{second-claim}. 

We next show that \begin{claim}\label{third-claim}$S_3$ is equal to the sum of  summands five, six,  and seven  of $b_r(X^{(r)}_{a_1,\dots,a_r,m})$, as given in Table~{\rm\ref{br-of-X}}.\end{claim}

\noindent Separate $S_3$ into two summands: one with $2\le j\le r$ and the other with $j=1$. Apply Lemma~\ref{May31-impt}.\ref{May31-impt-53.2} to both summands. Keep in mind that 
$X^{(r)}_{a_1,\dots,a_r,m}$ is in $\mathfrak B$; so $[2,\operatorname{least}(m)]\subseteq\{a_1,\dots,a_r\}$, where $m$ is a monomial in $\{x_2,\dots,x_d\}$ of positive degree. It follows that $2\le \operatorname{least}(m)$ and $a_1=2$. Thus, every monomial $m_1\in \binom{x_2,\dots,x_d}{n-1}$
satisfies
 \begin{equation}\label{a1=2}a_1\le \operatorname{least}(m_1).\end{equation} At any rate,
 we obtain $S_3=\sum_{i=1}^3S_{3,i}$ with 
\begingroup\allowdisplaybreaks
\begin{align}
S_{3,1}&=\phantom{-}\sum\limits_{j=2}^{r}\sum\limits_{m_1\in\binom{x_2,\dots,x_d}{n-1}}\chi(x_{a_j}|m)x_1Q_{m_1,\frac{m}{x_{a_j}}} (-1)^{j} Y^{(r-1)}_{a_1,\dots\widehat{a_j},\dots,a_{r},m_1}
 \notag\\
S_{3,2}&=-\sum\limits_{m_1\in\binom{x_2,\dots,x_d}{n-1}}\chi(x_{a_1}|m)x_1Q_{m_1,\frac{m}{x_{a_1}}}  \chi(a_2\le \operatorname{least}(m_1))Y^{(r-1)}_{a_2,\dots,a_{r},m_1}
\notag\\
S_{3,3}&=-\sum\limits_{m_1\in\binom{x_2,\dots,x_d}{n-1}}\chi(x_{a_1}|m)x_1Q_{m_1,\frac{m}{x_{a_1}}}  \chi(\operatorname{least}(m_1)<a_2)\sum\limits_{k=2}^{r}(-1)^{k}Y^{(r-1)}_{\operatorname{least}(m_1),a_2,\dots,\widehat{a_k},\dots,a_{r},\frac{x_{a_k}m_1}{{\operatorname{init}}(m_1)}}\notag
\end{align}\endgroup

\noindent In $S_{3,3}$, replace $x_{a_k}$ times the old $m_1$ with  the new $m_1$ to see that  $S_{3,3}$ is equal to 

$$x_1\sum\limits_{k=2}^{r}\sum\limits_{m_1\in\binom{x_2,\dots,x_d}{n}}(-1)^{k+1}\chi(x_{a_k}|m_1)
\chi(x_{a_1}|m)Q_{\frac{m_1}{x_{a_k}},\frac{m}{x_{a_1}}}  \chi(a_1\le \operatorname{least}(m_1)<a_2)Y^{(r-1)}_{\operatorname{least}(m_1),a_2,\dots,\widehat{a_k},\dots,a_{r},\frac{m_1}{{\operatorname{init}}(m_1)}}.$$
Let $\ell=\operatorname{least}(m_1)$ and 
replace the old $m_1$ by $x_{\ell}$ times the new $m_1$. Notice that $\chi(x_{a_k}|x_{\ell}m_1)=\chi(x_{a_k}|m_1)$ 
 because $\ell<a_k$; and therefore,
$$
S_{3,3}=x_1\sum\limits_{\ell=a_1}^{a_2-1}\sum\limits_{k=2}^{r}\sum\limits_{m_1\in\binom{x_{\ell},\dots,x_d}{n-1}}(-1)^{k+1}\chi(x_{a_k}|m_1)
\chi(x_{a_1}|m)Q_{\frac{x_{\ell}m_1}{x_{a_k}},\frac{m}{x_{a_1}}}  Y^{(r-1)}_{\ell,a_2,\dots,\widehat{a_k},\dots,a_{r},m_1}.$$
Separate $S_{3,3}$ into two summands $S_{3,3}'+S_{3,3}''$. In  $S_{3,3}'$, the parameter $\ell$ satisfies $a_{1}+1\le \ell\le a_{2}-1$ and in $S_{3,3}''$, $\ell$ is equal to $a_1$. We are now able to complete the proof of Claim~\ref{third-claim}: the sum $S_{3,1}+S_{3,3}''$ is summand five of  $b_r(X^{(r)}_{a_1,\dots,a_r,m})$, as given in Table~{\rm\ref{br-of-X}},
the sum $S_{3,3}'$ is  equal to summand six; and the sum $S_{3,2}$ is summand seven. 

Now that Claims \ref{first-claim}, \ref{second-claim}, and \ref{third-claim} have been established, the proof of (\ref{47.4-bi}) is also complete.

 We  prove (\ref{47.4-bii}). This proof is similar to the proof of (\ref{47.4-bi}). Fix an element  $Y^{(r)}_{a_1,\dots,a_r,m}$  of $B_r$. Recall that $2\le a_1<\dots<a_r\le d$ are integers, $m\in \binom{x_2,\dots x_d}{n-1}$ is a monomial, $a_1\le \operatorname{least}(m)$, and 
$Y^{(r)}_{a_1,\dots,a_r,m}$ is equal to $\kappa(x_{a_1}\wedge\dots\wedge x_{a_r}\otimes m)$. Identify the integer $\gamma$ with $[2,\gamma]\subseteq \{a_1,\dots,a_r\}$, but $\gamma+1\not\in \{a_1,\dots,a_r\}$. Apply Proposition~\ref{Best} to see that $b_r(Y^{(r)}_{a_1,\dots,a_r,m})=\sum_{i=1}^3S_i$, with
\begingroup\allowdisplaybreaks
\begin{align}
S_1&= x_1\sum\limits_{j=1}^r \sum\limits_{m_3\in\binom{x_2,\dots,x_d}{n}}
\sum\limits_{m_1,m_2\in\binom{x_1,\dots,x_d}{n-2}}
(-1)^{j+1}
t_{x_{a_j}mm_2}Q_{x_1m_1,x_1m_2}t_{m_1m_3} \notag\\&\hskip.2in \otimes \eta((x_{a_1}\wedge\dots\wedge\widehat{x_{a_j}}\wedge\dots\wedge x_{a_r})\otimes m_3^*),\notag\\
S_2&=x_1\sum\limits_{j=1}^r\sum\limits_{m_1\in\binom{x_2,\dots,x_d}{n-1}}\sum\limits_{m_2\in\binom{x_1,\dots,x_d}{n-2}}(-1)^{j+1}t_{x_{a_j}mm_2}Q_{m_1,x_1m_2} \otimes \kappa\left((x_{a_1}\wedge\dots\wedge\widehat{x_{a_j}}\wedge\dots\wedge x_{a_r})\otimes m_1\right),\text{ and}\notag\\
S_3&=-\pmb \delta\sum\limits_{j=1}^r(-1)^{j+1} x_{a_j}\otimes \kappa ((x_{a_1}\wedge\dots\wedge\widehat{x_{a_j}}\wedge\dots\wedge x_{a_r})\otimes m). \notag\end{align}\endgroup

We first show that 
\begin{claim}\label{first-claim'}$S_1$ is equal to the sum of the first two summands of $b_r(Y^{(r)}_{a_1,\dots,a_r,m})$, as given in Table~{\rm\ref{br-of-Y}}.\end{claim} 

\noindent Separate $S_1$ into two summands: one with $1\le j\le \gamma-1$ and the other with $\gamma\le j\le r$. Apply Lemma~\ref{May31-impt}.\ref{May31-impt-a} to both summands in order to obtain $S_1=S_{1,1}'+S_{1,1}''+S_{1,2}'+S_{1,2}''$ with
\begingroup\allowdisplaybreaks
\begin{align}
S_{1,1}'&=x_1\sum\limits_{j=1}^{\gamma-1} \sum\limits_{m_3\in\binom{x_2,\dots,x_d}{n}}\chi(\operatorname{least}(m_3)\le j)
\sum\limits_{m_1,m_2\in\binom{x_1,\dots,x_d}{n-2}}
(-1)^{j+1}
t_{x_{a_j}mm_2}Q_{x_1m_1,x_1m_2}t_{m_1m_3}  X^{(r-1)}_{a_1,\dots,\widehat{a_j},\dots,a_r,m_3},\notag\\
S_{1,1}''&=x_1\sum\limits_{j=1}^{\gamma-1} \sum\limits_{m_3\in\binom{x_2,\dots,x_d}{n}}
\sum\limits_{m_1,m_2\in\binom{x_1,\dots,x_d}{n-2}}\chi(j+1\le \operatorname{least}(m_3))
(-1)^{j+1}
t_{x_{a_j}mm_2}Q_{x_1m_1,x_1m_2}t_{m_1m_3}\notag\\&\hskip.2in\times  \sum\limits_{k=j+1}^r(-1)^{j+1+k}\chi(x_{a_k}|m_3)X^{(r-1)}_{a_1,\dots,\widehat{a_k},\dots,a_r,\frac{x_{j+1}m_3}{x_{a_k}}},\notag\\
S_{1,2}'&=x_1\sum\limits_{j=\gamma}^r \sum\limits_{m_3\in\binom{x_2,\dots,x_d}{n}}
\sum\limits_{m_1,m_2\in\binom{x_1,\dots,x_d}{n-2}}
(-1)^{j+1}
t_{x_{a_j}mm_2}Q_{x_1m_1,x_1m_2}t_{m_1m_3} X^{(r-1)}_{a_1,\dots,\widehat{a_j},\dots,a_r,m_3}\chi(\operatorname{least}(m_3)\le \gamma),\notag\\
S_{1,2}''&=x_1\sum\limits_{j=\gamma}^r \sum\limits_{m_3\in\binom{x_{\gamma+1},\dots,x_d}{n}}
\sum\limits_{m_1,m_2\in\binom{x_1,\dots,x_d}{n-2}}
(-1)^{j+1}
t_{x_{a_j}mm_2}Q_{x_1m_1,x_1m_2}t_{m_1m_3} \notag\\&\hskip.2in \times
\begin{cases}
\sum\limits_{k=\gamma}^{j-1}(-1)^{\gamma+k}\chi(x_{a_k}|m_3)X^{(r-1)}_{a_1,\dots,a_{\gamma-1},\gamma+1,a_{\gamma},\dots,\widehat{a_k},\dots,\widehat{a_j},\dots,a_r,\frac{x_{\gamma+1}m_3}{x_{a_k}}}\\
+\sum\limits_{k=j+1}^r(-1)^{\gamma+k+1}\chi(x_{a_k}|m_3)X^{(r-1)}_{a_1,\dots,a_{\gamma-1},\gamma+1,a_{\gamma},\dots,\widehat{a_j},\dots,\widehat{a_k},\dots,a_r,\frac{x_{\gamma+1}m_3}{x_{a_k}}}.
\end{cases}
\notag
\end{align}\endgroup

\noindent In $S_{1,1}''$, let $\ell=j+1$ and replace $m_3$ by $x_{a_k}m_3$, exactly as in the proof of (\ref{47.4-bi}), in order to obtain 
\begin{equation}\label{Y-eq1}S_{1,1}''= x_1\sum\limits_{\ell=2}^{\gamma} \sum\limits_{k=\ell}^r\sum\limits_{m_3\in\binom{x_\ell,\dots,x_d}{n-1}}
\sum\limits_{m_1,m_2\in\binom{x_1,\dots,x_d}{n-2}}
(-1)^{k}
t_{x_{\ell}mm_2}Q_{x_1m_1,x_1m_2}t_{m_1m_3x_{a_k}}  X^{(r-1)}_{a_1,\dots,\widehat{a_k},\dots,a_r,
x_{\ell}m_3}.\end{equation}
 In $S_{1,1}'$ and $S_{1,2}'$, let $\ell=\operatorname{least}(m_3)$ and apply the tricks of the proof of (\ref{47.4-bi})
to see that
\begin{equation}\label{Y-eq2} S_{1,1}'= x_1\sum\limits_{\ell=2}^{\gamma}\sum\limits_{k=\ell}^{\gamma-1}
\sum\limits_{m_3\in\binom{x_\ell,\dots,x_d}{n-1}}
\sum\limits_{m_1,m_2\in\binom{x_1,\dots,x_d}{n-2}}
(-1)^{k+1}
t_{x_{a_k}mm_2}Q_{x_1m_1,x_1m_2}t_{m_1m_3x_\ell}  X^{(r-1)}_{a_1,\dots,\widehat{a_k},\dots,a_r,x_\ell m_3}
\end{equation}
and
\begin{equation}\label{Y-eq3} S_{1,2}'= x_1\sum\limits_{\ell=2}^\gamma\sum\limits_{k=\gamma}^r \sum\limits_{m_3\in\binom{x_\ell,\dots,x_d}{n-1}}
\sum\limits_{m_1,m_2\in\binom{x_1,\dots,x_d}{n-2}}
(-1)^{k+1}
t_{x_{a_k}mm_2}Q_{x_1m_1,x_1m_2}t_{m_1m_3x_\ell} X^{(r-1)}_{a_1,\dots,\widehat{a_k},\dots,a_r,x_\ell m_3}.\end{equation}
Combine (\ref{Y-eq1}), (\ref{Y-eq2}), and (\ref{Y-eq3}) to see that $(S_{1,1}'+S_{1,2}')+S_{1,1}''$ is equal to 
$$
x_1\sum\limits_{\ell=2}^{\gamma}\sum\limits_{k=\ell}^{r}
\sum\limits_{m_3\in\binom{x_\ell,\dots,x_d}{n-1}}
\sum\limits_{m_1,m_2\in\binom{x_1,\dots,x_d}{n-2}}
(-1)^{k}Q_{x_1m_1,x_1m_2}
[t_{x_{\ell}mm_2}t_{x_{a_k}m_1m_3}  
-t_{x_{a_k}mm_2}t_{x_{\ell}m_1m_3}]
X^{(r-1)}_{a_1,\dots,\widehat{a_k},\dots,a_r,x_{\ell}m_3},$$which is the first summand of $b_r(Y^{(r)}_{a_1,\dots,a_r,m})$, as given in Table~{\rm\ref{br-of-Y}}. Rearrange $S_{1,2}''$ exactly as was done in (\ref{47.4-bi}) to obtain $$S_{1,2}''=\begin{cases} 
x_1
\sum\limits_{\gamma\le j<k\le r}
\sum\limits_{m_3\in\binom{x_{\gamma+1},\dots,x_d}{n-1}}
\sum\limits_{m_1,m_2\in\binom{x_1,\dots,x_d}{n-2}}
(-1)^{j+\gamma+k}Q_{x_1m_1,x_1m_2}\\\hskip.2in \times\left[ 
t_{x_{a_j}mm_2}t_{x_{a_k}m_1m_3} 
-t_{x_{a_k}mm_2}t_{x_{a_j}m_1m_3}
\right]
X^{(r-1)}_{a_1,\dots,a_{\gamma-1},\gamma+1,a_{\gamma},\dots,\widehat{a_j},
\dots,\widehat{a_k},\dots,a_r,x_{\gamma+1}m_3},\end{cases}
$$which is the second summand of $b_r(Y^{(r)}_{a_1,\dots,a_r,m})$, as given in Table~{\rm\ref{br-of-Y}}.
This completes the proof of Claim~\ref{first-claim'}.

We next show that \begin{claim}\label{second-claim'}$S_2$ is equal to the sum of  summands three, four, five, and six of $b_r(Y^{(r)}_{a_1,\dots,a_r,m})$, as given in Table~{\rm\ref{br-of-Y}}.\end{claim}

\noindent Separate $S_2$ into two summands: one with $2\le j\le r$ and the other with $j=1$. Apply Lemma~\ref{May31-impt}.\ref{May31-impt-53.2} to both summands in order to obtain $S_2=\sum_{i=1}^5S_{2,i}$ with
\begingroup\allowdisplaybreaks
\begin{align}
S_{2,1}&= x_1\sum\limits_{j=2}^r\sum\limits_{m_1\in\binom{x_2,\dots,x_d}{n-1}}\sum\limits_{m_2\in\binom{x_1,\dots,x_d}{n-2}}(-1)^{j+1}t_{x_{a_j}mm_2}Q_{m_1,x_1m_2} \chi(a_1\le \operatorname{least}(m_1))Y^{(r-1)}_{a_1,\dots,\widehat{a_j},\dots, a_r,m_1}\notag\\
S_{2,2}&=\begin{cases}x_1\sum\limits_{j=2}^r\sum\limits_{m_1\in\binom{x_2,\dots,x_d}{n-1}}\sum\limits_{m_2\in\binom{x_1,\dots,x_d}{n-2}}(-1)^{j+1}t_{x_{a_j}mm_2}Q_{m_1,x_1m_2} \chi(\operatorname{least}(m_1)<a_1)
\\ \hskip.2in\times 
\sum\limits_{k=1}^{j-1}(-1)^{k+1}Y^{(r-1)}_{\operatorname{least}(m_1),a_1,\dots,
\widehat{a_k},\dots \widehat{a_j},\dots, a_r,\frac{x_{a_k}m_1}{{\operatorname{init}}(m_1)}}\end{cases}\notag\\
S_{2,3}&=\begin{cases}x_1\sum\limits_{j=2}^r\sum\limits_{m_1\in\binom{x_2,\dots,x_d}{n-1}}\sum\limits_{m_2\in\binom{x_1,\dots,x_d}{n-2}}(-1)^{j+1}t_{x_{a_j}mm_2}Q_{m_1,x_1m_2} \chi(\operatorname{least}(m_1)<a_1)\\ \hskip.2in\times \sum\limits_{k=j+1}^{r}(-1)^{k}Y^{(r-1)}_{\operatorname{least}(m_1),a_1,\dots,
\widehat{a_j},\dots \widehat{a_k},\dots, a_r,\frac{x_{a_k}m_1}{{\operatorname{init}}(m_1)}}\end{cases}\notag\\
S_{2,4}&=x_1\sum\limits_{j=1}\sum\limits_{m_1\in\binom{x_2,\dots,x_d}{n-1}}\sum\limits_{m_2\in\binom{x_1,\dots,x_d}{n-2}}(-1)^{j+1}t_{x_{a_j}mm_2}Q_{m_1,x_1m_2} \chi(a_2\le \operatorname{least} (m_1)) Y^{(r-1)}_{a_2,\dots,a_r,m_1}\notag\\
S_{2,5}&=\begin{cases}x_1\sum\limits_{j=1}\sum\limits_{m_1\in\binom{x_2,\dots,x_d}{n-1}}\sum\limits_{m_2\in\binom{x_1,\dots,x_d}{n-2}}(-1)^{j+1}t_{x_{a_j}mm_2}Q_{m_1,x_1m_2} \chi(\operatorname{least}(m_1)<a_2)\\ \hskip.2in\times \sum\limits_{k=2}^r(-1)^kY^{(r-1)}_{\operatorname{least}(m_1),a_2,\dots,\widehat{a_k},\dots,a_r,
\frac{x_{a_k}m_1}{{\operatorname{init}}(m_1)}}\end{cases}\notag
\end{align}\endgroup

Notice that in the present calculation inequality (\ref{a1=2}) does not necessarily hold. This is why there are no summands of $S_3$ in the proof of Claim~\ref{third-claim} which are analogous to the summands $S_{2,2}$ and $S_{2,3}$ in the present calculation. Separate $S_{2,5}$ into three summands: in $S_{2,5}'$ the monomials $m_1$ satisfy $\operatorname{least}(m_1)< a_1$, in $S_{2,5}''$ the monomials  satisfy $\operatorname{least}(m_1)=a_1$, and in $S_{2,5}'''$ the monomials satisfy $a_1+1\le \operatorname{least}(m_1)$. 
Add $S_{2,3}$ and $S_{2,5}'$ to obtain 
$$S_{2,3}+S_{2,5}'=\begin{cases}x_1\sum\limits_{1\le j<k\le r}\sum\limits_{m_1\in\binom{x_2,\dots,x_d}{n-1}}\sum\limits_{m_2\in\binom{x_1,\dots,x_d}{n-2}}(-1)^{k+j+1}t_{x_{a_j}mm_2}Q_{m_1,x_1m_2} \chi(\operatorname{least}(m_1)<a_1)\\ \hskip.2in\times Y^{(r-1)}_{\operatorname{least}(m_1),a_1,\dots,
\widehat{a_j},\dots \widehat{a_k},\dots, a_r,\frac{x_{a_k}m_1}{{\operatorname{init}}(m_1)}}.\end{cases}$$
Transform $S_{2,2}$, $S_{2,3}+S_{2,5}'$, $S_{2,5}''$, and $S_{2,5}'''$ simultaneously. First replace $m_1$ so that the new
$m_1$ is  equal to  $x_{a_k}$ times the old $m_1$. Notice that 
$$1\le k\implies a_1\le a_k\implies \chi(\operatorname{least}(x_{a_k}m_1)<a_1)=\chi(\operatorname{least}(m_1)<a_1)$$ and
$$2\le k\implies a_2\le a_k\implies \chi(a_1\le\operatorname{least}(x_{a_k}m_1)<a_2)=\chi(a_1\le\operatorname{least}(m_1)<a_2).$$
Furthermore, if $\chi(\operatorname{least}(m_1)<a_1)$ is non-zero, then $\operatorname{least}(x_{a_k}m_1)=\operatorname{least}(m_1)$ 
and ${\operatorname{init}}(x_{a_k}m_1)={\operatorname{init}}(m_1)$
for $1\le k$; and if $\chi(a_1\le\operatorname{least}(m_1)<a_2)$ is non-zero, then $\operatorname{least}(x_{a_k}m_1)=\operatorname{least}(m_1)$ 
and ${\operatorname{init}}(x_{a_k}m_1)={\operatorname{init}}(m_1)$
for $2\le k$. Then let $\ell=\operatorname{least}(m_1)$ and replace 
the old $m_1$ with $x_{\ell}$ times the new $m_1$. At this point, $S_{2,2}+S_{2,3}+S_{2,5}'$ is equal to 
\begingroup\allowdisplaybreaks
\begin{align}
&=\begin{cases}x_1\sum\limits_{\ell=2}^{a_1-1}\sum\limits_{1\le j<k\le r}\sum\limits_{m_1\in\binom{x_{\ell},\dots,x_d}{n-1}}\sum\limits_{m_2\in\binom{x_1,\dots,x_d}{n-2}}(-1)^{k+j}
\\\hskip.2in\times [
\chi(x_{a_j}|m_1)t_{x_{a_k}mm_2}Q_{\frac{x_{\ell}m_1}{x_{a_j}},x_1m_2} 
-\chi(x_{a_k}|m_1)t_{x_{a_j}mm_2}Q_{\frac{x_{\ell}m_1}{x_{a_k}},x_1m_2}]
Y^{(r-1)}_{\ell,a_1,\dots,
\widehat{a_j},\dots \widehat{a_k},\dots, a_r,m_1},\end{cases}\notag\\ 
S_{2,5}''&=x_1\sum\limits_{k=2}^r\sum\limits_{m_1\in\binom{x_{a_1},\dots,x_d}{n-1}}\sum\limits_{m_2\in\binom{x_1,\dots,x_d}{n-2}}(-1)^{k}\chi(x_{a_k}|m_1)t_{x_{a_1}mm_2}Q_{\frac{x_{a_1}m_1}{x_{a_k}},x_1m_2} Y^{(r-1)}_{a_1,a_2,\dots,\widehat{a_k},\dots,a_r,
m_1},\text{ and}\notag\\ 
S_{2,5}'''&=x_1\sum\limits_{\ell=a_1+1}^{a_2-1}\sum\limits_{k=2}^r\sum\limits_{m_1\in\binom{x_{\ell},\dots,x_d}{n-1}}\sum\limits_{m_2\in\binom{x_1,\dots,x_d}{n-2}}(-1)^{k}\chi(x_{a_k}|m_1)t_{x_{a_1}mm_2}Q_{\frac{x_{\ell}m_1}{x_{a_k}},x_1m_2} Y^{(r-1)}_{\ell,a_2,\dots,\widehat{a_k},\dots,a_r,
m_1}.\notag
\end{align}\endgroup

\noindent Observe that $S_{2,2}+S_{2,3}+S_{2,5}'$ is summand three from  $b_r(Y^{(r)}_{a_1,\dots,a_r,m})$, as given in Table~{\rm\ref{br-of-Y}}, $S_{2,1}+S_{2,5}''$ is summand four, $S_{2,5}'''$ is summand five, and $S_{2,4}$ is summand six. This completes the proof of Claim~\ref{second-claim'}.

Finally, we show that 
\begin{claim}\label{third-claim'}$S_3$ is equal to the sum of summands seven, eight, and nine of $b_r(Y^{(r)}_{a_1,\dots,a_r,m})$, as given in Table~{\rm\ref{br-of-Y}}.\end{claim} 

\noindent Separate $S_3$ into two summands: in one of the summands $2\le j\le r$ and in the other summand $j=1$. Apply Lemma~\ref{May31-impt}.\ref{May31-impt-53.2} to each summand. Keep in mind the the ambient hypothesis in (\ref{47.4-bii}) ensures that $a_1\le \operatorname{least}(m)$. At any rate, $S_3=\sum\limits_{i=1}^3S_{3,i}$ with
\begingroup\allowdisplaybreaks
\begin{align}
S_{3,1}&=\pmb \delta\sum\limits_{j=2}^r(-1)^{j} x_{a_j} 
Y^{(r-1)}_{a_1,\dots,\widehat{a_j},\dots,a_r,m},\notag\\
S_{3,2}&=-\pmb \delta x_{a_1}\chi(a_2\le \operatorname{least} (m))Y^{(r-1)}_{a_2,\dots,a_r,m},\text{ and}\notag\\
S_{3,3}&=-\pmb \delta x_{a_1}\chi(\operatorname{least}(m)<a_2)\sum\limits_{k=2}^{r}(-1)^k Y^{(r-1)}_{\operatorname{least}(m),a_2,\dots,\widehat{a_k},\dots,a_r,\frac{x_{a_k}m}{{\operatorname{init}}(m)}};\notag
\end{align}\endgroup furthermore, these sums are summands seven, eight, and nine of $b_r(Y^{(r)}_{a_1,\dots,a_r,m})$, respectively. This completes the proof of Claim~\ref{third-claim'}, (\ref{47.4-bii}), and Theorem~\ref{47.4}. \end{proof}

  \begin{example}\label{d=3matrices} The case $d=3$ is already studied in \cite{EK-K-2}, see also Example~\ref{d=3}.  In the notation of the present paper,  use the bases 
$$\begin{array}{rcllll}
&&\{Y^{(0)}\}&&&\text{for $B_0$},\\
\{X_{2,x_2m}^{(1)}|m\in \binom{x_2,x_3}{n-1}\}&\cup& \{Y_{2,m}^{(1)}|m\in \binom{x_2,x_3}{n-1}\}&\cup&
\{Y_{3,m}^{(1)}|m\in \binom{x_3}{n-1}\}&\text{for $B_1$},\\
\{- Y_{2,3,m}^{(2)}|m\in \binom{x_2,x_3}{n-1}\}&\cup& \{X_{2,3,x_2m}^{(2)}|m\in \binom{x_2,x_3}{n-1}\}&\cup&\{X_{2,3,x_3m}^{(2)}\mid m\in\binom{x_3}{n-1}\}
&\text{for $B_2$, and}\\ 
&&\{ X^{(3)}\}&&&\text{for $B_3$.}\end{array}$$ The monomials of $\binom{x_2,x_3}{n-1}$ may be listed in any order; but use the same order for $\{X_{2,x_2m}^{(1)}\}$ and $\{-Y_{2,3,m}^{(2)}\}$. The analogous instruction is in effect for $\{Y^{(1)}_*\}$ and $\{X^{(2)}_*\}$.
 Let $[b_i]$ be the matrix of $b_i$ with respect to this basis. Then  $[b_2]$ is an alternating matrix and $[b_3]$ is the transpose of $[b_1]$. The listed bases are dual bases (on the nose) with respect to the perfect pairing $\operatorname{pp}: B_r\otimes B_{3-r}\to B_3$; see Observation~\ref{behavior} or (\ref{promise}).
\end{example}

\begin{example}\label{d=4matrices}
Let $d=4$. Use the bases $\{Y^{(0)}\}$ for $B_0$,
$$\begin{array}{cccccl}
\{X_{2,x_2m}^{(1)}|m\in \binom{x_2,x_3,x_4}{n-1}\}&\cup& \{Y_{2,m}^{(1)}|m\in \binom{x_2,x_3,x_4}{n-1}\} &\cup& \{Y_{3,m}^{(1)}|m\in \binom{x_3,x_4}{n-1}\}\\
{}\cup\{Y_{4,m}^{(1)}|m\in \binom{x_4}{n-1}\}&&&&&
\text{for $B_1$},\vspace{5pt}\\
\{X_{2,3,(x_2m)}^{(2)}|m\in \binom{x_2,x_3,x_4}{n-1}\}&\cup&
\{X_{2,3,(x_3m)}^{(2)}|m\in \binom{x_3,x_4}{n-1}\}&\cup&
\{X_{2,4,(x_2m)}^{(2)}|m\in \binom{x_2,x_3,x_4}{n-1}\}\\
{}\cup\{Y_{2,4,m}^{(2)}|m\in \binom{x_2,x_3,x_4}{n-1}\}&\cup&
\{Y_{3,4,m}^{(2)}|m\in \binom{x_3,x_4}{n-1}\}&\cup&
\{-Y_{2,3,m}^{(2)}|m\in \binom{x_2,x_3,x_4}{n-1}\}
&\text{for $B_2$,}\vspace{5pt}\\
\{-Y^{(3)}_{2,3,4,m}\mid m\in \binom{x_2,x_3,x_4}{n-1} &\cup&\{X_{2,3,4,x_2m}^{(3)}|m\in \binom{x_2,x_3,x_4}{n-1}\}&\cup& \{X_{2,3,4,x_3m}^{(3)}|m\in \binom{x_3,x_4}{n-1}\}\\
{}\cup\{X_{2,3,4,x_4m}^{(3)}|m\in \binom{x_4}{n-1}\} &&&&&\text{for $B_3$,}\end{array}$$and
$\{X^{(4)}\}$ for $B_4$. The monomials of $\binom{x_2,x_3,x_3}{n-1}$ and $\binom{x_3,x_4}{n-1}$ 
may be listed in any order; but the same order should be used each time.

For each index $i$, let $[b_i]$ be the matrix of $b_i$ with respect to the chosen  basis. Then  $[b_4]$ is the transpose of $[b_1]$
and if $[b_2]$ has the form
$[b_2]=\bmatrix A&B\endbmatrix$, where $A$ and $B$ each are matrices with $n^2+2n$ columns, then $$[b_3]=-\bmatrix \text{the transpose of $B$}\\\hline \text{the transpose of $A$}\endbmatrix.$$
 Let $e_1,\dots,e_{(n+1)^2}$; $f_1,\dots,f_{2n^2+4n}$, and $g_1,\dots,g_{(n+1)^2}$ be the listed bases for $B_1$, $B_2$, and $B_3$, respectively. Then
\begin{equation}\label{crit-prop}\operatorname{pp}(e_i\otimes g_j)=\chi(i=j)\cdot X^{(4)}\quad\text{and}\quad \operatorname{pp}(f_i\otimes f_j)=\chi(|j-i|=n^2+2n)\cdot  X^{(4)}.\end{equation} The relationship between $[b_2]$ and $[b_3]$ is explained  by Proposition~\ref{ppr-prop}.\ref{ppr-prop-c}. Let 
$$\bmatrix C\\\hline D\endbmatrix$$
be the matrix of $b_3$ with respect to the given bases. (The submatrices  $C$ and $D$ each have $n^2+2n$ rows.) If $1\le j\le n^2+2n$, then
$$\begin{array}{ll}b_2(f_j)=\sum\limits_{k=1}^{(n+1)^2} A_{k,j} e_k,&\quad\text{hence} \quad \operatorname{pp}(b_2(f_j)\otimes g_i)=A_{i,j}\cdot  X^{(4)},\text{ and}\\
b_3(g_i)=\sum\limits_{k=1}^{n^2+2n} C_{k,i} f_k+\sum\limits_{k=1}^{n^2+2n} D_{k,i} f_{n^2+2n+k}
,&\quad\text{hence} \quad \operatorname{pp}(f_j\otimes b_3(g_i))=D_{j,i}\cdot  X^{(4)}.\end{array}
$$The fact that $\operatorname{pp}(b_2(f_j)\otimes g_i)+\operatorname{pp}(f_j\otimes b_3(g_i))=0$ yields that $D$ is minus the transpose of $A$. A slight modification in the argument accounts for the relationship between $C$ and $B$. 

We describe the algorithm we used for recording the basis of $\mathbb B$ that is given at the beginning of the example. The critical property that we aim for is (\ref{crit-prop}). We first choose the basis elements for $B_0$, $B_1$, and ${\mathfrak R}\otimes K_{1,n-1}U_0\subseteq B_2$ as listed in Definition~\ref{basis-47}, then we complete the basis for  ${\mathfrak R}\otimes L_{1,n}U_0\subseteq B_2$, $B_3$, and $B_4$ using the duality of Observation~\ref{behavior}. 
\end{example}

\begin{example}\label{d=4matrices-ext} To illustrate
\begin{table}
\begin{center}
The matrix $[\overline{b_2}]$ from Example~\ref{d=4matrices-ext} is equal to $\pmb\delta$ times 
$$\left[
\text{
\begin{tabular}{cccccccc|cccccccc}
$x_3$&$-x_2$&$0$&$0$&$0$&$x_4$&$0$&$-x_2$&$0$&$0$&$0$&$0$&$0$&$0$&$0$&$0$\\
$0$&$x_3$&$0$&$-x_2$&$0$&$0$&$x_4$&$0$&$0$&$0$&$0$&$0$&$0$&$0$&$0$&$0$\\
$0$&$0$&$x_3$&$0$&$-x_2$&$0$&$0$&$x_4$&$0$&$0$&$0$&$0$&$0$&$0$&$0$&$0$\\
$0$&$0$&$0$&$0$&$0$&$0$&$0$&$0$&$x_4$&$0$&$0$&$0$&$0$&$-x_3$&$0$&$0$\\
$0$&$0$&$0$&$0$&$0$&$0$&$0$&$0$&$0$&$x_4$&$0$&$0$&$0$&$x_2$&$-x_3$&$0$\\
$0$&$0$&$0$&$0$&$0$&$0$&$0$&$0$&$-x_2$&$0$&$x_4$&$0$&$0$&$0$&$0$&$-x_3$\\
$0$&$0$&$0$&$0$&$0$&$0$&$0$&$0$&$0$&$0$&$0$&$x_4$&$0$&$0$&$x_2$&$0$\\
$0$&$0$&$0$&$0$&$0$&$0$&$0$&$0$&$0$&$-x_2$&$0$&$-x_3$&$x_4$&$0$&$0$&$x_2$\\
$0$&$0$&$0$&$0$&$0$&$0$&$0$&$0$&$0$&$0$&$-x_2$&$0$&$-x_3$&$0$&$0$&$0$\\
\end{tabular}
}
\right]$$ 
$$\text{and }[\overline{b_3}]=\pmb \delta \bmatrix 
0&0&0&-x_4&0&x_2&0&0&0\\
0&0&0&0&-x_4&0&0&x_2&0\\
0&0&0&0&0&-x_4&0&0&x_2\\
0&0&0&0&0&0&-x_4&x_3&0\\
0&0&0&0&0&0&0&-x_4&x_3\\
0&0&0&x_3&-x_2&0&0&0&0\\
0&0&0&0&x_3&0&-x_2&0&0\\
0&0&0&0&0&x_3&0&-x_2&0\\\hline
-x_3&0&0&0&0&0&0&0&0\\
x_2&-x_3&0&0&0&0&0&0&0\\
0&0&-x_3&0&0&0&0&0&0\\
0&x_2&0&0&0&0&0&0&0\\
0&0&x_2&0&0&0&0&0&0\\
-x_4&0&0&0&0&0&0&0&0\\
0&-x_4&0&0&0&0&0&0&0\\
x_2&0&-x_4&0&0&0&0&0&0\\
\endbmatrix$$
\caption{{\bf The middle matrices from Example~\ref{d=4matrices-ext}.}}\label{d=4-middle-matrices}
\end{center}\end{table}
  Example \ref{d=4matrices} more thoroughly, we record the skeleton $\overline{\mathbb B}$ (see Remark \ref{skeleton-2.2})  of $\mathbb B$ when $d=4$ and $n=2$.
As described in Example \ref{d=4matrices}, we use the bases $Y^{(0)}$ for $B_0$;
\begingroup
\allowdisplaybreaks\begin{align}&e_1=X^{(1)}_{2,x_2^2},& &e_2=X^{(1)}_{2,x_2x_3},&  &e_3=X^{(1)}_{2,x_2x_4},& &e_4=Y^{(1)}_{2,x_2},& &e_5=Y^{(1)}_{2,x_3},\notag\\ &e_6=Y^{(1)}_{2,x_4},&  &e_7=Y^{(1)}_{3,x_3},& &e_8=Y^{(1)}_{3,x_4},& &e_9=Y^{(1)}_{4,x_4}&&\text{for $B_1$;}\notag\\
&f_1=X^{(2)}_{2,3,x_2^2},&\quad
&f_2=X^{(2)}_{2,3,x_2x_3},&\quad
&f_3=X^{(2)}_{2,3,x_2x_4},&\quad
&f_4=X^{(2)}_{2,3,x_3^2},&\quad
&f_5=X^{(2)}_{2,3,x_3x_4},\notag\\
&f_6=X^{(2)}_{2,4,x_2^2},&\quad
&f_7=X^{(2)}_{2,4,x_2x_3},&\quad
&f_8=X^{(2)}_{2,4,x_2x_4},&\quad
&f_9=Y^{(2)}_{2,4,x_2},&\quad
&f_{10}=Y^{(2)}_{2,4,x_3},\notag\\
&f_{11}=Y^{(2)}_{2,4,x_4},&\quad
&f_{12}=Y^{(2)}_{3,4,x_3},&\quad
&f_{13}=Y^{(2)}_{3,4,x_4},&\quad
&f_{14}=-Y^{(2)}_{2,3,x_2},&\quad
&f_{15}=-Y^{(2)}_{2,3,x_3},\notag\\
&f_{16}=-Y^{(2)}_{2,3,x_4}&&\text{for $B_2$;}
\notag\\
&g_1=- Y^{(3)}_{2,3,4,x_2},& &g_2=- Y^{(3)}_{2,3,4,x_3},& &g_3=- Y^{(3)}_{2,3,4,x_4},& &g_4= X^{(3)}_{2,3,4,x_2^2},& &g_5=
 X^{(3)}_{2,3,4,x_2x_3},\notag\\
&g_6= X^{(3)}_{2,3,4,x_2x_4},& &g_7= X^{(3)}_{2,3,4,x_3^2},& &g_8= X^{(3)}_{2,3,4,x_3x_4},& &g_9= X^{(3)}_{2,3,4,x_4^2}&&\text{for $B_3$;}\notag\end{align}\endgroup  and $X^{(4)}$ for $B_4$. The matrices in $\overline{\mathbb B}$ are
$[\overline{b_1}]=\pmb\delta\bmatrix 0&0&0&x_2^2&x_2x_3&x_2x_4&x_3^2&x_3x_4&x_4^2\endbmatrix$,
$[\overline{b_2}]$ and $[\overline{b_3}]$ are given in Table~\ref{d=4-middle-matrices},
and $[\overline{b_4}]$ is the transpose of $[\overline{b_1}]$. The matrices are calculated using the formulas of Theorem~\ref{47.4}.
\end{example}

\end{document}